\documentclass[12pt,leqno]{article}
\usepackage{amssymb}
\usepackage[mathscr]{eucal}
\usepackage{amsmath,amssymb,latexsym,theorem,bbm}
\usepackage{color,url}
\usepackage{enumitem}

\newcommand{\SC}{\scriptstyle}

\newcommand{\CC}{\mathsf{C}}
\newcommand{\DD}{\mathsf{D}}

\newcommand{\NN}{\mathbb{N}}
\newcommand{\QQ}{\mathbb{Q}}
\newcommand{\RR}{\mathbb{R}}

\newcommand{\ZZ}{\mathbb{Z}}

\newcommand{\bD}{{\boldsymbol{D}}}

\newcommand{\bI}{{\boldsymbol{I}}}

\newcommand{\bP}{{\boldsymbol{P}}}

\newcommand{\bQ}{{\boldsymbol{Q}}}

\newcommand{\bX}{{\boldsymbol{X}}}

\newcommand{\bxi}{{\boldsymbol{\xi}}}

\newcommand{\bvare}{{\boldsymbol{\vare}}}

\newcommand{\bzero}{{\boldsymbol{0}}}
\newcommand{\bone}{{\boldsymbol{1}}}

\newcommand{\cA}{{\mathcal A}}
\newcommand{\cB}{{\mathcal B}}

\newcommand{\cD}{{\mathcal D}}

\newcommand{\cI}{\mathcal{I}}

\newcommand{\cG}{{\mathcal G}}
\newcommand{\cF}{{\mathcal F}}
\newcommand{\cH}{{\mathcal H}}

\newcommand{\cU}{{\mathcal U}}

\newcommand{\bcU}{\boldsymbol{\cU}}

\newcommand{\cX}{{\mathcal X}}

\newcommand{\cY}{{\mathcal Y}}

\newcommand{\cW}{{\mathcal W}}

\newcommand{\bcY}{\boldsymbol{\cY}}

\newcommand{\dd}{\mathrm{d}}
\newcommand{\ee}{\mathrm{e}}

\newcommand{\slu}{{\SC\mathrm{lu}}}

\newcommand{\EE}{\operatorname{\mathbb{E}}}
\newcommand{\PP}{\operatorname{\mathbb{P}}}

\newcommand{\var}{\operatorname{Var}}
\newcommand{\cov}{\operatorname{Cov}}

\newcommand{\vare}{\varepsilon}

\renewcommand{\mid}{\,|\,}

\renewcommand{\leq}{\leqslant}
\renewcommand{\geq}{\geqslant}

\newcommand{\stoch}{\stackrel{\PP}{\longrightarrow}}
\newcommand{\distr}{\stackrel{\cD}{\longrightarrow}}
\newcommand{\distrf}{\stackrel{\cD_f}{\longrightarrow}}
\newcommand{\distre}{\stackrel{\cD}{=}}

\newcommand{\qmean}{\stackrel{L_2}{\longrightarrow}}

\newcommand{\lu}{\stackrel{\slu}{\longrightarrow}}
\newcommand{\Jto}{\stackrel{J_1}{\longrightarrow}}
\newcommand{\Jtoo}[1]{\stackrel{J_1(#1)}{\longrightarrow}}
\newcommand{\as}{\stackrel{{\mathrm{a.s.}}}{\longrightarrow}}
\newcommand{\ase}{\stackrel{{\mathrm{a.s.}}}{=}}

\newcommand{\bbone}{\mathbbm{1}}
\newcommand{\ns}{{\lfloor ns\rfloor}}
\newcommand{\nt}{{\lfloor nt\rfloor}}

\newcommand{\nT}{{\lfloor nT\rfloor}}

\newcommand{\proofend}{\hfill\mbox{$\Box$}}

\numberwithin{equation}{section}

\theoremstyle{change} \theorembodyfont{\em}
\newtheorem{Lem}{Lemma.}[section]
\newtheorem{Thm}[Lem]{Theorem.}
\newtheorem{Pro}[Lem]{Proposition.}
\newtheorem{Cor}[Lem]{Corollary.}

\theorembodyfont{\rm}
\newtheorem{Rem}[Lem]{Remark.}
\newtheorem{Ex}[Lem]{Example.}

\begin{document}

\begin{center}
{\bfseries\Large
Functional limit theorems for Galton--Watson\\[1mm]
 processes with inhomogeneous immigration}

\vspace*{3mm}

 {\sc\large
  M\'aty\'as $\text{Barczy}^{*,\diamond}$ \text{and}
  \ D\'aniel $\text{Bezd\'any}^{**}$ }

\end{center}

\vskip0.2cm

\noindent
 * HUN-REN--SZTE Analysis and Applications Research Group,
   Bolyai Institute, University of Szeged,
   Aradi v\'ertan\'uk tere 1, H--6720 Szeged, Hungary.

\noindent
 ** Bolyai Institute, University of Szeged,
    Aradi v\'ertan\'uk tere 1, H--6720 Szeged, Hungary.

\noindent E-mails: barczy@math.u-szeged.hu (M. Barczy),
                   bezdany@server.math.u-szeged.hu (D. Bezd\'any).

\noindent $\diamond$ Corresponding author.

\vskip0.2cm

\renewcommand{\thefootnote}{}

\vspace*{0.05cm}

\begin{abstract}
We study the asymptotic behavior of a sequence of Galton--Watson processes with inhomogeneous immigration
 when the limit of the means of the offspring distributions is less than $1$ or equal to $1$.
Under growth conditions on the expected values of the immigration distributions
 and the variances of the offspring distributions,
 and assuming the weak convergence of properly scaled immigration processes
 towards a non-negative stochastic process $\cY$ with c\`adl\`ag or continuous sample paths,
 we establish functional limit theorems for the sequence of Galton--Watson processes with inhomogeneous immigration in question.
The limit stochastic processes can be represented as a constant multiple or an integral functional of $\cY$.
\end{abstract}

%\tableofcontents

\section{Introduction}\label{Intro}

Branching processes are mathematical models used to describe the evolution of populations
and have found numerous applications in areas including demography, ecology,
epidemiology, cancer research, and physics as well.
A prominent example is the Galton–Watson process with or without immigration,
where individuals reproduce independently according to a given offspring law
and, in case of immigration, the arrival of immigrants in each generation occurs
independently of the population's evolution according to a given immigration law.

From the view point of applications, it is worthwhile and realistic to consider
several modifications of Galton-Watson processes with immigration,
among others, the case when the number of immigrating individuals
in different generations may be not independent and possibly
having different distributions (i.e., the immigration law can depend on the generation).
In mathematical terms, with $\ZZ_+$ and $\NN$ denoting the set of non-negative
 and positive integers, respectively,
 let $(Y_k)_{k\in\ZZ_+}$ be a $\ZZ_+$-valued stochastic process,
and let $(X_k)_{k\in\ZZ_+}$ be a stochastic process given by
\begin{equation}\label{X_def}
X_{k}=\sum_{j=1}^{X_{k-1}}\xi_{k,j}+Y_{k-1}, \qquad
	\text{$k\in\NN$,}
\end{equation}
 with $X_0=0$ and using the convention $\sum_{j=1}^0:=0$,
 where $\{\xi_{k,j}: k,j\in\NN\}$ are independent,
 identically distributed $\ZZ_+$-valued random variables independent of $\{Y_k: k\in\ZZ_+\}$.
We call $(X_k)_{k\in\ZZ_+}$ a Galton--Watson process with inhomogeneous immigration (GWII process),
 and $(Y_k)_{k\in\ZZ_+}$ is called the immigration process of $(X_k)_{k\in\ZZ_+}$.
If, in addition, $Y_k$, $k\in\ZZ_+$ are independent and identically distributed,
 then we say that $(X_k)_{k\in\ZZ_+}$
 is a Galton--Watson process with immigration (GWI process).
Intuitively, we can interpret $X_0$ as the initial size of a population,
 and, for each $k\in\NN$, $X_k$ as the size of the $k$-th generation of a population,
 $\xi_{k,j}$ as the number of offspring of the $j$-th individual in the $(k-1)$-th generation,
 and $Y_{k-1}$ as the number of immigrants joining the $k$-th generation.

Studying asymptotic behavior of Galton–Watson processes with
 and without immigration has a long and vast history.
The nature of these limit theorems highly depends on whether the
 offspring mean of the branching process in question is
 less than, equal to, or greater than one (subcritical, critical, or supercritical cases).

In this paper, we investigate the asymptotic behavior of a sequence of Galton--Watson processes
 with inhomogeneous immigration in the cases where the limiting mean of
 the offspring distributions is less than $1$ or equal to $1$
(roughly speaking corresponding to the subcritical and critical cases).
In what follows, we only mention and summarize some of those earlier results
that are directly connected to asymptotic behavior of Galton–Watson processes with inhomogeneous immigration, the topic of the present paper.
For a book and a recent survey on several branching processes 
with inhomogeneous immigration both in discrete and continuous time
with special emphasis on their asymptotic behavior, see Rahimov \cite{Rah3}
and Rahimov \cite{Rah2}, respectively. 
In particular, in Sections 4, 5.1 and 6 in Rahimov \cite{Rah2}, 
a lot of papers are collected for the discrete time case.

Foster and Williamson \cite[part (i) of Theorem on page 227]{FosWil}
proved that given a critical GWII process such that $\PP(\xi_{1,1}=0)<1$,
$\PP(\xi_{1,1}=1)<1$, $\var(\xi_{1,1})<\infty$, $\{Y_k: k\in\ZZ_+\}$
are independent and there exists a non-negative random variable $Y$ such that 
 $\frac{1}{n}\sum_{k=0}^n Y_k$ converges to $Y$ in distribution as $n\to\infty$,
 we have that $\frac{X_n}{n}$ converges in distribution as $n\to\infty$
 and the limit random variable is characterized via its Laplace transform.
Nagaev \cite{Nag} studied the case when $(Y_k)_{k\in\ZZ_+}$ is a wide-sense stationary process
such that $\lim_{k\to\infty}\cov(Y_0,Y_k)=0$, and under some assumption on the generating function
of $\xi_{1,1}$, it was proved that $\frac{X_n}{n}$ converges in distribution to
a Gamma distributed random variable as $n\to\infty$.
Later, Asadullin and Nagaev \cite[Theorem 1]{AsaNag} proved a similar result,
 where the Laplace transform of the limit distribution
 can be written as the composition of the Laplace transform of
 a Gamma distributed random variable and
 the moment generating function of another random variable.
Badalbaev and Zubkov \cite[Theorem 1 and Remarks 2 and 4]{BadZub} established 
   a limit theorem for a sequence of some branching processes 
   (including critical GWI processes), which covers the results of Nagaev \cite{Nag},
   Asadullin and Nagaev \cite[Theorem 1]{AsaNag} and Foster and Williamson \cite[part (i) of Theorem on   
   page 227]{FosWil} as special cases.

Concerning the deterministic approximation of a sequence of nearly critical 
 GWII processes (the sequence of offspring means tends to $1$),
 Rahimov and Sharipov have written several papers, 
see, for example, Rahimov \cite{Rah1} and Rahimov and Sharipov \cite{RahSha}.
Rahimov \cite{Rah1} considered the case when, for each $n\in\NN$,
the random variables $\{ Y^{(n)}_k : k\in\ZZ_+\}$ representing immigration
are independent, but not identically distributed. 
More precisely, Rahimov \cite[Theorem 1]{Rah1} investigated the asymptotic behavior 
of appropriately normalized sequence of such
 nearly critical GWII processes under the assumption (among others) that the sequence
of the immigration means and variances can be approximated by regularly varying 
sequences with non-negative exponents.
It was shown that the limit stochastic process is in fact deterministic, 
and it depends on the rate of the convergence of the sequence of the offspring means towards 1.
The proof of Theorem 1 in Rahimov \cite{Rah1} is based on a limit theorem which gives sufficient conditions
under which the partial sum processes of a triangular array of dependent 
random variables converges weakly to a deterministic continuous function,
see, e.g., Theorem A in Rahimov \cite{Rah1}.
Rahimov and Sharipov \cite[Theorem 1]{RahSha} have recently shown that the above recalled
result of Rahimov \cite[Theorem 1]{Rah1} remains true when, for each $n\in\NN$,  
the random variables $\{ Y^{(n)}_k : k\in\ZZ_+\}$ can be dependent as well
(under a slightly stronger condition on the immigration variances compared to the one 
in Rahimov \cite{Rah1}). The method of proof in Rahimov and Sharipov \cite[Theorem 1]{RahSha}
is the same as in Rahimov \cite{Rah1}. 
For further references on the extensions of the result of Rahimov \cite{Rah1},
see the Introduction of Rahimov and Sharipov \cite{RahSha}.
Here we only mention Guo and Zhang \cite[Theorem 2.1]{GuoZha}, who considered the case
when there exists an $N\in\NN\setminus\{1\}$ such that $Y_i$ and $Y_j$ are independent 
for any $i,j\in\ZZ_+$ satisfying $\vert i-j\vert\geq N$.

Very recently, Sharipov \cite{Sha} has derived a functional limit theorem for a critical GWII process 
  $(X_k)_{k\in\ZZ_+}$ such that $(Y_k)_{k\in\ZZ_+}$ is strictly stationary and ergodic 
  satisfying an additional (one might call a Maxwell-Woodroofe type) condition
  under finite second-order moment assumptions on the offspring and immigration distributions. 
  It has been shown that $(n^{-1}X_{\nt})_{t\in\RR_+}$ converges weakly towards a squared Bessel process
  as $n\to\infty$, which extends the results of Wei and Winnicki \cite[Theorem 2.1]{WW}
  concerning a critical GWI process.

We call the attention of the reader to the fact that functional limit theorems for 
 branching processes can be interesting on their own rights, 
but they are also very useful in deriving the asymptotic behavior 
of some estimators of the offspring and immigration means, 
see, e.g., Isp\'any et al.\ \cite[Section 3]{IspPapZui} (for the offspring mean
in case of a critical GWI process) and Rahimov and Sharipov \cite[Section 3]{RahSha2} 
(for the offspring mean based on partial observations in case of 
a critical branching process with generation-dependent immigrations).

The paper is organized as follows.
In Section \ref{Section_GWII_model} first we give some examples of 
 GWII processes that appear in the literature
 (e.g., Rahimov and Sharipov \cite{RahSha}, Sagitov \cite{Sag},
 Barczy et al.\ \cite{BarBezPap2}, and Barczy and Bezd\'any \cite{BarBez}),
introduce the main hypotheses of our paper,
 and provide some examples of stochastic processes which satisfy our hypotheses.
We close Section \ref{Section_GWII_model} with some remarks and
 auxiliary results concerning the hypotheses.
Section \ref{Section_conv_results} contains our main results on the asymptotic behavior
 of sequences of GWII processes, see Theorems \ref{main_1}, \ref{main_2}, and \ref{exp_thm}.
In all these theorems, we assume a growth condition on
 the sequence of expected value functions of the immigration processes (see Hypothesis \ref{H2}).
In Theorems \ref{main_1} and \ref{main_2}, under the assumption that the properly scaled
 sequence of immigration processes converges weakly to some appropriate stochastic process $\cY$ (see Hypothesis \ref{H3})
 and growth conditions on the sequence of variances of the offspring distributions (see Hypotheses \ref{H4} and \ref{H5}),
 we establish functional limit theorems for a sequence of GWII processes
 in the case when the offspring means converge to a number less than $1$ or converge to $1$, respectively.
In the case of Theorem \ref{main_1}, the limit stochastic process is a constant multiple of $\cY$,
 and in the case of Theorem \ref{main_2}, the limit stochastic process
 is an integral process of a function of $\cY$.
In Theorem \ref{exp_thm}, we establish the locally uniform convergence of the expected value functions
 of the sequence of GWII processes.
We end Section \ref{Section_conv_results} with an application of Theorem \ref{main_2} to prove
 a functional limit theorem for some $2$-type Galton--Watson processes with immigration,
 see Corollary \ref{main_4}.
In Section \ref{Prelims}, we introduce some notation and present some lemmas which will be used in the proofs.
In Section \ref{Section_prop_proofs} we prove Propositions \ref{Pro_Y_jump} and \ref{Pro_H2_UI}
 concerning our hypotheses.
Sections \ref{main_proof_1}--\ref{main_proof_4} contain the proofs of our main results,
 which are based on careful applications of a version of the continuous mapping theorem
 and moment estimates of certain random variables.
We emphasize that our method of proof is different from the ones in Rahimov \cite{Rah1}
 and Rahimov and Sharipov \cite{RahSha}, who used a theorem about the convergence of
 a sequence of martingale differences towards a diffusion process.
We close the paper with two appendices.
Appendix \ref{GWII_moments} contains some results on the moments of some
 random variables in the paper, and in Appendix \ref{CMT} we present 
 an equivalent characterization of convergence in the Skorokhod $J_1$ topology,
 and we recall some facts about the convergence of c\`ad\`ag functions,
 a version of the continuous mapping theorem,
 and some lemmas establishing the Borel measurability or continuity of certain mappings
 (in particular, taking composition of c\`adl\`ag functions).

\section{Galton--Watson processes with inhomogeneous immigration}\label{Section_GWII_model}

In this section, first we introduce some notations,
 give some known examples from the literature of GWII processes,
 then introduce five hypotheses that are used throughout the paper,
 give two examples of stochastic processes which satisfy our hypotheses about the immigration processes,
 and, in the end, we present two results which highlight the role of
 the third hypothesis in question.

Let $\ZZ_+$, $\NN$, $\QQ$, $\RR$, $\RR_+$ and $\RR_{++}$ denote the set
of non-negative integers, positive integers, rational numbers, real numbers, non-negative real
numbers and positive real numbers, respectively.
The maximum of two real numbers $x,y\in\RR$ is denoted by $x\vee y$.
The positive part of a real number $x\in\RR$ is denoted by $x^+$.
The $d\times d$ identity matrix is denoted by $\bI_d$.
The null vector in $\RR^d$ is denoted by $\bzero_d$,
 and if the dimension is clear from context, we simply write $\bzero$.
The restriction of a function $f:\RR\to\RR$ on a subset $A$ of $\RR$ is denoted by $f\vert_A$.
We use the convention $\sum_{j=1}^0:=\sum_{j=0}^{-1}:=0$.

Every random variable will be defined on a fixed probability space $(\Omega, \cA, \PP)$.
Convergence in probability, convergence in $L_2$, equality in distribution and almost sure equality
 is denoted by $\stoch$, $\qmean$, $\distre$ and $\ase$, respectively.
A function $f : \RR_+ \to \RR^d$ is called c\`adl\`ag if it is right continuous with left limits.
For $d\in\NN$, we will use $\DD(\RR_+,\RR^d)$ for the set of c\`adl\`ag functions from $\RR_+$ to $\RR^d$,
and $\CC(\RR_+,\RR^d)$ for the set of continuous functions from $\RR_+$ to $\RR^d$.
For each $d\in\NN$ and $f\in\DD(\RR_+,\RR^d)$, we define
 $\Delta f(0):=\bzero\in\RR^d$ and $\Delta f(t):=f(t)-\lim_{s\uparrow t}f(s)$, $t\in\RR_{++}$.
For each $d\in\NN$ and all $t\in\RR_+$, we define $\pi_t:\DD(\RR_+,\RR^d)\to\RR^d$,
$\pi_t(f):=f(t)$, $f\in\DD(\RR_+,\RR^d)$, that is, $\pi_t$ is the natural projection onto $t$.
For $d\in\NN$, the weak convergence of the finite dimensional distributions of $\RR^d$-valued stochastic processes with sample paths in $\DD(\RR_+, \RR^d)$
and the weak convergence of $\RR^d$-valued stochastic processes with sample paths in $\DD(\RR_+, \RR^d)$
are denoted by $\distrf$ and $\distr$, respectively.
For functions $f:\RR_+\to\RR^d$ and $f_n:\RR_+\to\RR^d$, $n \in \NN$,
we write $f_n \lu f$ as $n\to\infty$ if $f_n$ converges to $f$
locally uniformly as $n\to\infty$, i.e., if $\sup_{t \in K\cap\RR_+} \|f_n(t) - f(t)\| \to 0$
as $n \to \infty$ for all compact sets $K$ of $\RR$,
 or, equivalently, if $\sup_{t \in [0,T]} \|f_n(t) - f(t)\| \to 0$ as $n\to\infty$ for all $T\in\RR_+$
 (for some properties of locally uniform convergence, see Appendix \ref{CMT}).
For functions $f\in\DD(\RR_+,\RR^d)$ and $f_n\in\DD(\RR_+,\RR^d)$, $n \in \NN$,
 we write $f_n\Jtoo{d}f$ as $n\to\infty$ if $f_n$ converges to $f$
 as $n\to\infty$ in the Skorokhod $J_1$ topology,
 and if the dimension is clear from context, simply $f_n \Jto f$ as $n\to\infty$
 (for details about Skorokhod $J_1$ topology, see Appendix \ref{CMT}).

In the following remark we list some known models in the literature
which can be considered as special cases of the GWII process in \eqref{X_def}.
\begin{Rem}\label{model_remark}
\begin{enumerate}[label=(\roman*)]
\item
	If $(Y_k)_{k\in\ZZ_+}$ is a sequence of identically distributed, independent $\ZZ_+$-valued random variables,
	then \eqref{X_def} recovers the classical single-type Galton--Watson process with immigration.
\item
	The model described in \eqref{X_def} was considered, for example, in Rahimov and Sharipov \cite{RahSha}.
\item
	Suppose that $p\in\NN$ and $(\bX_k)_{k\in\ZZ_+}$ is a classical $p$-type Galton--Watson process with immigration,
	whose offspring mean matrix is triangular.
	Such processes were considered, for example, in Sagitov \cite[Theorem 4]{Sag} (without immigration), and in Barczy et al.\ \cite{BarBezPap2} and Barczy and Bezd\'any \cite{BarBez}
	in the $p=2$ and $p=3$ cases, respectively (with immigration).
	Let $i\in\{1,\dots,p\}$ be fixed, and assume without loss of generality that the offspring mean matrix is lower triangular.
	Then, for the $i$-th coordinate of $\bX_k$, $k\in\NN$, using the notation and assumptions of Barczy and Bezd\'any \cite[Section 2]{BarBez}, we can write
	\begin{equation*}
	X_{k,i}
		=\sum_{j=1}^{i}\sum_{\ell=1}^{X_{k-1,j}}\xi_{k,\ell,j,i}+\vare_{k,i}
		=\sum_{\ell=1}^{X_{k-1,i}}\xi_{k,\ell,i,i}+\left(\sum_{j=1}^{i-1}\sum_{\ell=1}^{X_{k-1,j}}\xi_{k,\ell,j,i}+\vare_{k,i}\right),
		\quad k\in\NN,
	\end{equation*}
	 where $X_{0,i}=0$,
	 and by $\xi_{k,\ell,j,i}$ we denote the number of type $i$ offsprings produced by the $\ell$-th
	 individual who is of type $j$ belonging to the $(k-1)$-th generation,
	 and the number of type $i$ immigrants in the $k$-th generation is denoted by $\vare_{k,i}$.
	Denoting the expression in the parentheses on the right hand side of the equation above by $Y_{k-1,i}$ for each $k\in\NN$,
	we have that $(X_{k,i})_{k\in\ZZ_+}$ is a GWII process in the sense of \eqref{X_def} with immigration process $(Y_{k,i})_{k\in\ZZ_+}$.
	Indeed, $\{Y_{h,i}:h\in\ZZ_+\}$ and  $\{\xi_{k,\ell,i,i}: k,\ell\in\NN\}$ are independent
	 for each $i\in\{1,\dots,p\}$,
	 as it is easy to see that, for each $i\in\{1,\dots,p\}$, the stochastic process $(Y_{h,i})_{h\in\ZZ_+}$
	 can be expressed as a function of the random variables
	 $\{\xi_{h,\ell, j, r}, \vare_{h,r}: h\in\NN, \ell\in\NN, j\in\{1,\dots,i-1\}, r\in\{1,\dots,i\}\}$,
	 which are assumed to be independent of $\{\xi_{k,\ell,i,i}: k,\ell\in\NN\}$
	 (due to the definition of $p$-type GWI processes).
	For example, in cases of $h=1$ and $h=2$, we have that
	\[
	Y_{1,i}
		=\sum_{j=1}^{i-1}\sum_{\ell=1}^{X_{1,j}}\xi_{2,\ell,j,i}+\vare_{2,i}
		=\sum_{j=1}^{i-1}\sum_{\ell=1}^{\vare_{1,j}}\xi_{2,\ell,j,i}+\vare_{2,i},
	\]
	and
	\[
	Y_{2,i}
		=\sum_{j=1}^{i-1}\sum_{\ell=1}^{X_{2,j}}\xi_{3,\ell,j,i}+\vare_{3,i},
	\]
	where
	\[
	X_{2,j}
		=\sum_{r=1}^{j}\sum_{\ell=1}^{X_{1,r}}\xi_{2,\ell,r,j}+\vare_{2,j}
		=\sum_{r=1}^{j}\sum_{\ell=1}^{\vare_{1,r}}\xi_{2,\ell,r,j}+\vare_{2,j}, \qquad j\in\{1,\dots,i-1\},
	\]
	which shows that $X_{2,j}$ is independent of $\{\xi_{k,\ell,i,i}: k,\ell\in\NN\}$ for each $j\in\{1,\dots,i-1\}$,
	yielding that $Y_{2,i}$ is also independent of $\{\xi_{k,\ell,i,i}: k,\ell\in\NN\}$.
\proofend
\end{enumerate}
\end{Rem}

In this paper, we consider a sequence of GWII processes.
For each $n\in\NN$, let $(X_k^{(n)})_{k\in\ZZ_+}$ be a GWII process in the sense of \eqref{X_def}
with immigration process $(Y_k^{(n)})_{k\in\ZZ_+}$, that is,
\begin{equation}\label{Xn_def}
X_{k}^{(n)}=\sum_{j=1}^{X_{k-1}^{(n)}}\xi_{k,j}^{(n)}+Y_{k-1}^{(n)}, \qquad \text{$k\in\NN$ \ with \ $X_0^{(n)}=0$},\quad n\in\NN.
\end{equation}

For the rest of this paper,
\begin{equation*}
\text{
we assume that \ $m_n\in\RR_{++}$, $n\in\NN$, \ is a fixed sequence of positive real numbers.
}\end{equation*}

Next, we introduce five hypotheses.
Whenever we suppose that some of these hypotheses are satisfied, we will state them explicitly.
\begin{enumerate}[label=(\textbf{H\arabic*})]
\item\label{H1}
	We have that $\EE\big(\xi_{1,1}^{(n)}\big)<\infty$, $n\in\NN$.
\item\label{H2}
	We have that $\EE\big(Y_k^{(n)}\big)<\infty$, $k\in\ZZ_+$, $n\in\NN$,
	and there exists $L\in\DD(\RR_+,\RR)$ such that
	\begin{equation*}\label{E_conv}
	\limsup_{n\to\infty}\sup_{s\in[0,t]}\EE\big(m_n^{-1}Y_{\ns}^{(n)}\big)
	\leq L(t),	\qquad t\in\RR_+.
	\end{equation*}
\item\label{H3}
	There exists a $\RR_+$-valued stochastic process $(\cY_t)_{t\in\RR_+}$ with c\`adl\`ag sample paths such that
	\begin{equation*}
	\left(m_n^{-1}Y_\nt^{(n)}\right)_{t\in\RR_+}
	\distr
	\left(\cY_t\right)_{t\in\RR_+}
	\qquad \text{as $n\to\infty$.}
	\end{equation*}
\item\label{H4}
	We have that $\EE\big((\xi_{1,1}^{(n)})^2\big)<\infty$, $n\in\NN$, and $m_n^{-1}n\var\big(\xi_{1,1}^{(n)}\big)\to0$ as $n\to\infty$.
\item\label{H5}
	We have that $\EE\big((\xi_{1,1}^{(n)})^2\big)<\infty$, $n\in\NN$, and $m_n^{-1}\var\big(\xi_{1,1}^{(n)}\big)\to0$ as $n\to\infty$.
\end{enumerate}

Note that hypothesis \ref{H4} obviously implies that hypothesis \ref{H5} holds.
The reason for introducing both hypotheses \ref{H4} and \ref{H5} is that \ref{H4} is needed for the proof of Theorem \ref{main_1}, while \ref{H5} is needed for the proof of Theorem \ref{main_2}.

In the following, we give two examples for sequences of processes $(Y_k^{(n)})_{k\in\ZZ_+}$, $n\in\NN$, which satisfy hypotheses \ref{H2} and \ref{H3}.
\begin{Ex}\label{Ex_GWI}
Let $(Y_k^{(n)})_{k\in\ZZ_+}:=(Y_k)_{k\in\ZZ_+}$, $n\in\NN$,
 where $(Y_k)_{k\in\ZZ_+}$ is a Galton-Watson process with immigration, that is, $Y_0=0$ and
\[
	Y_k=\sum_{j=1}^{Y_{k-1}}\vare_{k,j}+\widetilde{\vare}_{k}, \qquad k\in\NN,
\]
 where $\{\vare_{k,j}, \widetilde{\vare}_k: k,j\in\ZZ_+\}$ are independent random variables
 such that $\{\vare_{k,j}: k,j\in\ZZ_+\}$ and $\{\widetilde{\vare}_k: k\in\ZZ_+\}$ both consist of i.i.d.\ $\ZZ_+$-valued random variables.
Furthermore, suppose that $\EE(\vare_{1,1}^2)<\infty$, $\EE(\vare_{1,1})=1$ (i.e., the critical case), $\EE(\widetilde{\vare}_1^2)<\infty$, and $m_n=n$, $n\in\NN$.
Next, we check that the sequence of stochastic processes $(Y_k^{(n)})_{k\in\ZZ_+}$, $n\in\NN$,
satisfy hypotheses \ref{H2} and \ref{H3}.
Since $\EE(Y_k)=k\EE(\widetilde{\vare}_1)$, $k\in\ZZ_+$, we have that
\[
\sup_{s\in[0,t]}\EE(n^{-1}Y_\ns^{(n)})=\sup_{s\in[0,t]}n^{-1}\ns\EE(\widetilde{\vare}_1)\leq t \EE(\widetilde{\vare}_1), \qquad t\in\RR_+,
\]
and thus hypothesis \ref{H2} is satisfied with $L(t):=t\EE(\widetilde{\vare}_1)$, $t\in\RR_+$.
Additionally, by Wei and Winnicki \cite[Theorem 2.1]{WW}, we have that
\[
\left(n^{-1}Y_\nt^{(n)}\right)_{t\in\RR_+}\distr\left(\cY_t\right)_{t\in\RR_+} \qquad \text{as $n\to\infty$,}
\]
where $(\cY_t)_{t\in\RR_+}$ is the pathwise unique strong solution of the stochastic differential equation (SDE)
 \begin{equation*}
  \dd \cY_t
  = \EE(\widetilde{\vare}_1) \, \dd t
    + \sqrt{\var(\vare_{1,1}) \, \cY_t^+} \, \dd \cW_t , \qquad t\in\RR_+,
  \qquad \cY_0 = 0 ,
 \end{equation*}
 where $(\cW_t)_{t\in\RR_+}$ is a standard Wiener process.
Consequently, since the stochastic process $(\cY_t)_{t\in\RR_+}$ has continuous (and thus c\`adl\`ag) sample paths,
 hypothesis \ref{H3} is also satisfied.
\end{Ex}

\begin{Ex}\label{Ex_Markov}
Let $h\in\NN$, and let $S\subset\ZZ_+$ be a finite set having $h$ elements.
Let $\bQ:=(q_{i,j})_{i,j\in S}\in\RR^{h\times h}$ be such that
 $q_{i,j}\geq0$, $i\neq j$, $i,j\in S$, and $q_{i,i}=-\sum_{j\in S\setminus\{i\}}q_{i,j}$, $i\in S$
 (i.e., a conservative rate matrix).
Let $\eta$ and $\eta_n$, $n\in\NN$, be $S$-valued random variables such that $\eta_n\distr\eta$ as $n\to\infty$. 
Let $n_0\in\NN$ be the smallest natural number such that
 $\bP_n:=\bI_h+n^{-1}\bQ\in\RR_+^{h\times h}$ holds for each $n\geq n_0$, $n\in\NN$.
Note that for $n\geq n_0$, $n\in\NN$, $\bP_n$ is a transition matrix,
 since $\bP_n\in\RR_{+}^{h\times h}$ and each row-sum of $\bP_n$ is equal to $1$
 (following from the fact that each row-sum of $\bQ$ is equal to $0$).
For each $n\in\{1,\dots,n_0-1\}$, let $(Y_k^{(n)})_{k\in\ZZ_+}$ be a constant $0$ process,
 and for each $n\geq n_0$, $n\in\NN$, let $(Y_k^{(n)})_{k\in\ZZ_+}$ be a Markov chain on $S$ with transition matrix $\bP_n$ and $Y_0^{(n)}\distre\eta_n$.
Supposing that $m_n:=1$, $n\in\NN$, the stochastic processes $(Y_k^{(n)})_{k\in\ZZ_+}$, $n\in\NN$,
 clearly satisfy \ref{H2}, since $\EE(Y_\nt^{(n)})\leq \max_{s\in S}s<\infty$, $t\in\RR_+$, $n\in\NN$,
 thus $L(t):=\max_{s\in S}s$, $t\in\RR_+$, is a suitable choice of the function $L$ in \ref{H2}.
Furthermore, we verify that \ref{H3} also holds with a continuous time Markov chain $(\cY_t)_{t\in\RR_+}$
 with rate matrix $\bQ$ and $\cY_0\distre\eta$.
For sufficiently large $n\in\NN$, we have that 
 	$\|\bP_n-\bI_h\|=n^{-1}\|\bQ_n\|<1$,
 	where $\|\cdot\|$ denotes the spectral norm on $\RR^{h\times h}$,
 	thus we may assume without loss of generality that this holds for all $n\in\NN$.
 Consequently, by Hall \cite[Theorem 2.7]{Hal}, the matrix logarithms
 \[
 \log(\bP_n)
 	:=\sum_{k=1}^\infty(-1)^{k+1}\frac{(\bP_n-\bI_h)^k}{k}
 	=\sum_{k=1}^\infty(-1)^{k+1}\frac{\bQ_n^k}{kn^k}, \qquad n\in\NN,
 \]
  are well-defined and satisfy
  $\ee^{\log\bP_n}=\bP_n$, $n\in\NN$,
  where the matrix exponential is defined in the usual way (see, e.g., Hall \cite[Proposition 2.1]{Hal}).
Furthermore, by Hall \cite[Proposition 2.8]{Hal}, we have that
  \[
  \log(\bP_n)=\log(\bI_h+n^{-1}\bQ)=n^{-1}\bQ+\bD_n, \qquad n\in\NN,
  \]
  for some $\bD_n\in\RR^{h\times h}$, $n\in\NN$, where $n\|\bD_n\|\to0$ as $n\to\infty$.
 Now it is easy to see that for $s<t$, $s,t\in\RR_+$, and $x_1,x_2\in S$,
 by the properties of the matrix exponential
 	(see, e.g., Hall \cite[Proposition 2.1 and part 5 of Proposition 2.3]{Hal}) and the
 	 fact that the transition probability function of $(\cY_t)_{t\in\RR_+}$
 takes the form $\ee^{t\bQ}$, $t\in\RR_+$
 (following from the finiteness of $S$ and the Kolmogorov differential equations), we have that
\begin{align*}
&\PP(Y_\nt^{(n)}=x_2\mid Y_\ns^{(n)}=x_1)
	=(\bP_n^{\nt-\ns})_{x_1,x_2}
	=\left(\left(\ee^{\log\bP_n}\right)^{\nt-\ns}\right)_{x_1,x_2}\\
&\qquad\qquad
	=\left( \ee^{(\nt-\ns)\log\bP_n}\right)_{x_1,x_2}
	= \left(\ee^{\frac{\nt-\ns}{n}(\bQ+n\bD_n)}\right)_{x_1,x_2}\\
&\qquad\qquad
	\to(\ee^{\bQ(t-s)})_{x_1,x_2}=\PP(\cY_t=x_2\mid\cY_s=x_1) \qquad \text{as $n\to\infty$.}
\end{align*}
Using the above convergence,
 the multi-step multiplication formula (for the Markov processes
$(Y_k^{(n)})_{k\in\ZZ_+}$, $n\in\NN$, and $(\cY_t)_{t\in\RR_+}$),
 and the fact that $Y_0^{(n)}\distr\cY_0$ as $n\to\infty$,
 we can derive that
\[
\left(Y_\nt^{(n)}\right)_{t\in\RR_+}\distrf\left(\cY_t\right)_{t\in\RR_+} \qquad \text{as $n\to\infty$.}
\]
To show that \ref{H3} holds, by Ethier and Kurtz \cite[Chapter 3, Theorems 2.2 and 7.8]{EthKur},
it is enough to verify that $(Y_\nt^{(n)})_{t\in\RR_+}$, $n\in\NN$, is tight.
We show that the conditions of Billingsley \cite[Theorem 16.10]{Bil} (Aldous's tightness criterion) are satisfied.
First, it is clear that
\[
\lim_{K\to\infty}\limsup_{n\to\infty}\PP\bigg(\sup_{t\in[0,T]}|Y_\nt^{(n)}|\geq K\bigg)=0,\qquad\text{$T\in\RR_+$,}
\]
since $|Y_\nt^{(n)}|\leq \max_{s\in S}s<\infty$, $t\in\RR_+$, $n\in\NN$, almost surely.
Next, let us fix $T\in\RR_{++}$ and $n\in\NN$  such that $n>n_0$, and let $\tau$ be a discrete stopping time
 for $(Y_\nt^{(n)})_{t\in\RR_+}$ (see Billingsley \cite[page 176]{Bil}) such that $\tau\leq T$.
Let $\cG_r:=\sigma(Y_0^{(n)},\dots,Y_r^{(n)})$, $r\in\ZZ_+$,
 where, since $\tau$ is a stopping time, we have that $\{\tau=t\}\subset\cG_{\nt}$, $t\in\RR_+$.
Since $(Y_r^{(n)})_{r\in\ZZ_+}$ has the Markov property, for each $\ell,k\in\NN$, $\ell<k$,
 we have that
	\begin{align*}
		\EE\left(\bbone_{\{Y_k^{(n)}\ne Y_{k-1}^{(n)}\}}\mid\cG_\ell\right)
		&=\EE\left(\EE\left(\bbone_{\{Y_k^{(n)}\ne Y_{k-1}^{(n)}\}}\mid\cG_{k-1}\right)\mid\cG_\ell\right)
		=\EE\left(\PP\left(Y_k^{(n)}\ne Y_{k-1}^{(n)}\mid Y_{k-1}^{(n)}\right)\mid\cG_\ell\right)\\
		&=\EE\left(1-(\bP_n)_{Y_{k-1}^{(n)},Y_{k-1}^{(n)}}\mid\cG_\ell\right)
		=\EE\left(1-(1-n^{-1}q_{Y_{k-1}^{(n)},Y_{k-1}^{(n)}})\mid\cG_\ell\right)\\
		&\leq\EE\left(n^{-1}\max_{j\in S}|q_{j,j}|\mid\cG_\ell\right)
		=n^{-1}\max_{j\in S}|q_{j,j}|.
	\end{align*}
For each $\ell\in\NN$ and all $s\in[0,T]$,
 we have that
\begin{align}
\begin{split}\label{Pro_help}
&\PP\Big(Y_{\lfloor n\tau\rfloor+\ell}^{(n)}\neq Y_{\lfloor n\tau\rfloor+\ell-1}^{(n)},\tau=s\Big)
	=\EE\left(\EE\left(\bbone_{\{Y_{\lfloor ns\rfloor+\ell}^{(n)}\neq Y_{\lfloor ns\rfloor+\ell-1}^{(n)}\}}\bbone_{\{\tau=s\}}\mid\cG_{\ns+\ell-1}\right)\right)\\
&\qquad\qquad\qquad
	=\EE\left(\EE\left(\bbone_{\{Y_{\lfloor ns\rfloor+\ell}^{(n)}\neq Y_{\lfloor ns\rfloor+\ell-1}^{(n)}\}}\mid\cG_{\ns+\ell-1}\right)\bbone_{\{\tau=s\}}\right)
	\leq n^{-1}\max_{j\in S}|q_{j,j}|\PP(\tau=s).
\end{split}
\end{align}
Since $\tau$ is discrete, there exists a countable set $K\subset[0,T]$
 such that $\PP(\tau\in K)=1$ and $\PP(\tau=s)>0$, $s\in K$, and thus for all $\delta,\vare\in\RR_{++}$,
 by \eqref{Pro_help}, we have that
\begin{align*}
 &\PP(|Y_{\lfloor n(\tau+\delta)\rfloor}^{(n)}-Y_{\lfloor n\tau\rfloor}^{(n)}|\geq\vare)
 	\leq\PP(Y_{\lfloor n(\tau+\delta)\rfloor}^{(n)}\neq Y_{\lfloor n\tau\rfloor}^{(n)})
 	\leq\PP(\text{$(Y_\nt^{(n)})_{t\in\RR_+}$ jumps in $(\tau,\tau+\delta]$})\\
 &\qquad\qquad
 	\leq\sum_{j=1}^{\lfloor n\delta\rfloor+1}
 		\PP\left(Y_{\lfloor n\tau\rfloor +j}^{(n)}\ne Y_{\lfloor n\tau\rfloor +j-1}^{(n)}\right)
 		=\sum_{j=1}^{\lfloor n\delta\rfloor+1}
 		\PP\left(Y_{\lfloor n\tau\rfloor +j}^{(n)}\ne Y_{\lfloor n\tau\rfloor +j-1}^{(n)}, \bigcup_{s\in K}\{\tau=s\}\right)\\
 &\qquad\qquad
 	=\sum_{j=1}^{\lfloor n\delta\rfloor+1}
 		\sum_{s\in K}\PP\left(Y_{\lfloor n\tau\rfloor +j}^{(n)}\ne Y_{\lfloor n\tau\rfloor +j-1}^{(n)}, \tau=s\right)
 	\leq\sum_{j=1}^{\lfloor n\delta\rfloor+1}
 		\sum_{s\in K}n^{-1}\max_{j\in S}|q_{j,j}|\PP(\tau=s)\\
 &\qquad\qquad
 			=\frac{\lfloor n\delta\rfloor+1}{n}\max_{j\in S}|q_{j,j}|
 			\leq\max_{j\in S}|q_{j,j}|(\delta+n^{-1}).
\end{align*}
This shows that all the conditions of Billingsley \cite[Theorem 16.10]{Bil} hold.
Consequently, we have that $(Y_\nt^{(n)})_{t\in\RR_+}\distr(\cY_t)_{t\in\RR_+}$, i.e., \ref{H3} holds.
\end{Ex}

In the next remark, we formulate some simple observations regarding our hypotheses.

\begin{Rem}\label{hyp_remark}
\textup{(i)}
	Hypothesis \textup{\ref{H2}} holds if $\EE\big(Y_k^{(n)}\big)<\infty$, $k\in\ZZ_+$, $n\in\NN$,
	 and there exists a function $\beta\in\DD(\RR_+,\RR)$ such that
	\begin{equation}\label{H2_beta}
	\left(\EE\big(m_n^{-1}Y_\nt^{(n)}\big)\right)_{t\in\RR_+}\Jto\left(\beta(t)\right)_{t\in\RR_+} \qquad \text{as $n\to\infty$.}
	\end{equation}
	In this case, we may choose $L(t):=\sup_{s\in[0,t+1]}\beta(s)$, $t\in\RR_+$.
	Indeed, if there exists such a function $\beta$, then $\beta$ is non-negative, and, since the mapping
	 $\DD(\RR_+,\RR)\ni f\mapsto \left(\sup_{s\in[0,t]}f(s)\right)_{t\in\RR_+}\in\DD(\RR_+,\RR)$ is continuous
	 (see, e.g., Whitt \cite[Theorem 6.1]{Whitt} or Ethier and Kurtz \cite[Problem 3.11.25]{EthKur}),
	 we have that
	\[
	\left(\sup_{s\in[0,t]}\EE\big(m_n^{-1}Y_\ns^{(n)}\big)\right)_{t\in\RR_+}
	\Jto \left(\sup_{s\in[0,t]}\beta(s)\right)_{t\in\RR_+} \qquad \text{as $n\to\infty$.}
	\]
	Since $(\sup_{s\in[0,t]}\beta(s))_{t\in\RR_+}$ is monotone increasing,
	 it has at most countably many discontinuity points.
	Therefore, for all $t\in\RR_+$, there exists $t_0\in[t,t+1]$ such that
	  the function $(\sup_{s\in[0,t]}\beta(s))_{t\in\RR_+}$ is continuous at $t_0$.
	Then, by Billingsley \cite[Theorem 16.6 or page 124]{Bil}, we have that
	  $\pi_{t_0}$ is continuous at $(\sup_{s\in[0,t]}\beta(s))_{t\in\RR_+}$, and thus
	\[
	\sup_{s\in[0,t]}\EE(m_n^{-1}Y_{\ns}^{(n)})
		\leq\sup_{s\in[0,t_0]}\EE(m_n^{-1}Y_{\ns}^{(n)})
		\to\sup_{s\in[0,t_0]}\beta(s)
		\leq\sup_{s\in[0,t+1]}\beta(s) \qquad \text{as $n\to\infty$},
	\]
	 and hence \ref{H2} is satisfied with the function $L$ defined above.
	Note also that if $\beta$ is continuous, then, by part \textup{(i)} of Lemma \ref{Jto_basic},
	 \eqref{H2_beta} is equivalent to
	\begin{equation}\label{H2_beta_cont}
	\left(\EE(m_n^{-1}Y_\nt^{(n)})\right)_{t\in\RR_+}\lu\left(\beta(t)\right)_{t\in\RR_+} \qquad \text{as $n\to\infty$.}
	\end{equation}

\textup{(ii)}
	Assume that $\EE(Y_k^{(n)})<\infty$, $k\in\ZZ_+$, $n\in\NN$,
	 the hypothesis \ref{H3} holds, and \eqref{H2_beta} is satisfied with some $\beta\in\DD(\RR_+,\RR)$.
	Then, in general, it is not true that $\beta(t)=\EE(\cY_t)$, $t\in\RR_+$.
	Indeed, consider the random variables $U_n$, $n\in\NN$,
	such that $\PP(U_n=n)=n^{-1}$ and $\PP(U_n=0)=1-n^{-1}$, $n\in\NN$,
	and define $Y_k^{(n)}:=U_n$, $k\in\ZZ_+$, $n\in\NN$.
	We have that $U_n\stoch 0$ as $n\to\infty$, yielding that 
	$U_n\distr 0$ as $n\to\infty$, and thus
	$(Y_\nt^{(n)})_{t\in\RR_+}\distr(0)_{t\in\RR_+}$ as $n\to\infty$.
	Hence \ref{H3} holds with $m_n:=1$, $n\in\NN$, and $\cY_t:=0$, $t\in\RR_+$.
	At the same time, we have that $\EE(U_n)=1$, $n\in\NN$, yielding that
	$(\EE(Y_\nt^{(n)}))_{t\in\RR_+}\lu(1)_{t\in\RR_+}$ as $n\to\infty$,
	and therefore \eqref{H2_beta} holds with $m_n=1$, $n\in\NN$,
	and $\beta(t):=1$, $t\in\RR_+$.
	However, $\EE(\cY_t)=0$, $t\in\RR_+$, and thus
	$(\beta(t))_{t\in\RR_+}$ does not coincide with $(\EE(\cY_t))_{t\in\RR_+}$.
	For another more involved counter-example,
	see Jacod and Shiryaev \cite[Chapter VII, Remark 3.19]{JacShi}.

\textup{(iii)}
	Neither hypothesis \textup{\ref{H4}} nor \textup{\ref{H5}}
	necessarily implies that $m_n\to\infty$ as $n\to\infty$
	 (e.g., if $m_n=1$ and $\var(\xi_{1,1}^{(n)})=n^{-2}$, $n\in\NN$).
	However, if either \ref{H4} or \ref{H5} holds and $\limsup_{n\to\infty}m_n<\infty$,
	 then we have that $\lim_{n\to\infty}n\var(\xi_{1,1}^{(n)})=0$ or $\lim_{n\to\infty}\var(\xi_{1,1}^{(n)})=0$, respectively.
	Indeed, if $\limsup_{n\to\infty}m_n<\infty$, then there exist $M^*>0$ and $n_0\in\NN$
	 such that $m_n<M^* +1$ for all $n\geq n_0$, yielding that $m_n^{-1}n\var(\xi_{1,1}^{(n)})\geq (M^*+1)^{-1}n\var(\xi_{1,1}^{(n)})\geq0$
	and $m_n^{-1}\var(\xi_{1,1}^{(n)})\geq (M^*+1)^{-1}\var(\xi_{1,1}^{(n)})\geq 0$ for all $n\geq n_0$.
	The desired convergences follow by the sandwich theorem.
	In both cases $\lim_{n\to\infty}\var(\xi_{1,1}^{(n)}) = 0$, yielding that 
	 $\xi_{1,1}^{(n)} - \EE(\xi_{1,1}^{(n)}) \qmean 0$ as $n\to\infty$.
	If, in addition, we assume that there exists an $a\in\RR_+$ such that $\EE(\xi_{1,1}^{(n)})\to a$ as $n\to\infty$, 
	then in both cases we have that $\xi_{1,1}^{(n)}\qmean a$ as $n\to\infty$.
	
	\textup{(iv)}
		If \textup{\ref{H4}} holds and there exists $a\in(0,1)$ such that $\EE(\xi_{1,1}^{(n)})\to a$ as $n\to\infty$,
		then $m_n/n\to\infty$ as $n\to\infty$.
		Indeed, we may assume without loss of generality that $\EE(\xi_{1,1}^{(n)})\in(0,1)$, $n\in\NN$.
		Since the random variables $\xi_{1,1}^{(n)}$, $n\in\NN$,
		are $\ZZ_+$-valued, we have that
		\begin{align*}
		\var(\xi_{1,1}^{(n)})
			&=\EE((\xi_{1,1}^{(n)})^2) - (\EE(\xi_{1,1}^{(n)}))^2
				\geq\EE(\xi_{1,1}^{(n)}) - (\EE(\xi_{1,1}^{(n)}))^2
				\to a-a^2\in\RR_{++} \qquad \text{as $n\to\infty$,}
		\end{align*}
		which, together with \textup{\ref{H4}}, yields that
		$n/m_n\to0$ as $n\to\infty$, and hence $m_n/n\to\infty$ as $n\to\infty$.

\textup{(v)}
	It is worth noting a special case when we have a single GWII process \eqref{X_def}
	 instead of a sequence of GWII processes.
	In this case, hypothesis \ref{H4} means that $\EE(\xi_{1,1}^2)<\infty$,
	 and at least one of the conditions $\var(\xi_{1,1})=0$ or $\lim_{n\to\infty}n^{-1}m_n=\infty$ is satisfied.
	Similarly, hypothesis \ref{H5} means that $\EE(\xi_{1,1}^2)<\infty$,
	  and at least one of the conditions $\var(\xi_{1,1})=0$ or $\lim_{n\to\infty}m_n=\infty$ is satisfied.
\proofend
\end{Rem}

We end this section with two propositions.
The first one shows that in the case when \ref{H3} holds
 and $m_n$ does not converge to $\infty$ as $n\to\infty$
 (e.g., in the case of Example \ref{Ex_Markov}),
 the stochastic process $(\cY_t)_{t\in\RR_+}$ is a pure-jump process,
 i.e., $\cY_t=\cY_0+\sum_{s\leq t}\Delta\cY_s$, $t\in\RR_+$, almost surely,
 where the sum on the right hand side of the equality is absolutely convergent for all $t\in\RR_+$ almost surely
(for more on pure-jump processes, see, e.g., {\c{C}}{\i}nlar \cite[page 317]{Cin}).

\begin{Pro}\label{Pro_Y_jump}
Suppose that $\liminf_{n\to\infty}m_n<\infty$ and hypothesis \textup{\ref{H3}} holds.
Then $(\cY_t)_{t\in\RR_+}$ is a pure-jump process.
In addition, if $(\cY_t)_{t\in\RR_+}$ has continuous sample paths, then $\PP\left(\cY_t=\cY_0, t\in\RR_+\right)=1$.
\end{Pro}

Finally, we give sufficient conditions under which \eqref{H2_beta}
 holds with $\beta(t):=\EE(\cY_t)$, $t\in\RR_+$
(note that this is not always the case, see part \textup{(ii)} of Remark \ref{hyp_remark}).
For a similar result on the convergence of the expectation functions
of a sequence of c\`adl\`ag processes with independent increments,
see Jacod and Shiryaev \cite[Chapter VII, Proposition 3.18]{JacShi}.

\begin{Pro}\label{Pro_H2_UI}
Suppose that the random variables
 $\sup_{t\in[0,T]}m_n^{-1}Y_\nt^{(n)}$, $n\in\NN$, are uniformly integrable for all $T\in\RR_{++}$, that is,
 \begin{equation}\label{Y_ui}
	\lim_{K\to\infty}\sup_{n\in\NN}\EE\left(\bigg(\sup_{t\in[0,T]}m_n^{-1}Y_\nt^{(n)}\bigg)\bone_{\left\{\sup_{t\in[0,T]}m_n^{-1}Y_\nt^{(n)}>K\right\}}\right)=0 \qquad \text{for all $T\in\RR_{++}$,}
 \end{equation}
and that \textup{\ref{H3}} holds such that $(\cY_t)_{t\in\RR_+}$ has continuous sample paths.
Then $\EE(\cY_t)<\infty$, $t\in\RR_+$,
 and \eqref{H2_beta} holds with the continuous function $\RR_+\ni t \mapsto \beta(t):=\EE(\cY_t)$.
\end{Pro}

\section{Main results}\label{Section_conv_results}

This section contains our main results.

\begin{Thm}\label{main_1}
For each $n\in\NN$, let $(X_k^{(n)})_{k\in\ZZ_+}$ be given by \eqref{Xn_def},
and assume that \textup{\ref{H1}--\ref{H4}} hold such that $(\cY_t)_{t\in\RR_+}$ has continuous sample paths and $\cY_0=0$.
Furthermore, assume that there exists an $a\in[0,1)$ such that $\EE(\xi_{1,1}^{(n)})\to a$ as $n\to\infty$.
Then
\begin{equation}\label{main_1_X_conv}
\left(m_n^{-1}\begin{bmatrix}
	Y_\nt^{(n)}\\[1mm]
	X_\nt^{(n)}
	\end{bmatrix}\right)_{t\in\RR_+}
\distr	\left(\begin{bmatrix}
	\cY_t\\
	\frac{1}{1-a}\cY_t
	\end{bmatrix}\right)_{t\in\RR_+}
 \qquad \text{as $n\to\infty$.}
\end{equation}
\end{Thm}

\begin{Rem}
The condition $\cY_0=0$ in Theorem \ref{main_1} is not a technical requirement.
Since $X_0^{(n)}=0$, $n\in\NN$, we have that $m_n^{-1}X_0^{(n)}\distr 0$ as $n\to\infty$.
If \eqref{main_1_X_conv} holds, then it also holds that
 $m_n^{-1}X_0^{(n)}\distr\frac{1}{1-a}\cY_0$ as $n\to\infty$.
Consequently, we have that $\frac{1}{1-a}\cY_0\distre0$, and, since $a\in[0,1)$,
 this implies that $\cY_0=0$.
\proofend
\end{Rem}

\begin{Thm}\label{main_2}
For each $n\in\NN$, let $(X_k^{(n)})_{k\in\ZZ_+}$ be given by \eqref{Xn_def},
and assume that \textup{\ref{H1}--\ref{H3}} and \textup{\ref{H5}} hold.
Furthermore, assume that there exist $\gamma\in\RR$ and $\gamma_n\in\RR$, $n\in\NN$,
such that $\gamma_n\to\gamma$ as $n\to\infty$ and $\EE(\xi_{1,1}^{(n)})=1+\frac{\gamma_n}{n}$, $n\in\NN$.
Then
\begin{equation}\label{main_2_X_conv}
\left(\begin{bmatrix}
	m_n^{-1}Y_\nt^{(n)}\\[1mm]
	(nm_n)^{-1}X_\nt^{(n)}
	\end{bmatrix}\right)_{t\in\RR_+}
\distr	\left(\begin{bmatrix}
	\cY_t\\
	\int_0^t\ee^{\gamma(t-s)}\cY_s\,\dd s
	\end{bmatrix}\right)_{t\in\RR_+}
 \qquad \text{as $n\to\infty$.}
\end{equation}
\end{Thm}

\begin{Rem}\label{cr_rem}
\textup{(i)}
Denote by $(\cX_t)_{t\in\RR_+}$ the second coordinate process of the limit process in \eqref{main_2_X_conv},
that is,
\[
\cX_t
	:=\int_0^t\ee^{\gamma(t-s)}\cY_s\,\dd s
	=\ee^{\gamma t}\int_0^t\ee^{-\gamma s}\cY_s\,\dd s, \qquad t\in\RR_+.
\]
Then $(\cX_t)_{t\in\RR_+}$ satisfies ($\PP$-almost surely) the differential equation
\begin{equation*}
\dd\cX_t=\left(\gamma\cX_t+\cY_t\right)\,\dd t, \qquad t\in\RR_+, \qquad \cX_0=0.
\end{equation*}

\textup{(ii)}
The condition on $\EE(\xi_{1,1}^{(n)})$, $n\in\NN$, in Theorem \ref{main_2}
implies that $\EE(\xi_{1,1}^{(n)})\to1$ as $n\to\infty$.
Furthermore, if $\EE(\xi_{1,1}^{(n)})=1$, $n\in\NN$,
 then $\gamma_n=0$, $n\in\NN$, $\gamma=0$, and
 the limit stochastic process in \eqref{main_2_X_conv} takes the form $([\cY_t,\int_0^t\cY_s\,\dd s]^\top)_{t\in\RR_+}$.
If $(Y_k^{(n)})_{k\in\ZZ_+}$ is a sequence of i.i.d.\ $\ZZ_+$-valued random variables for each $n\in\NN$,
then this means that the corresponding sequence of GWI processes $(X_k^{(n)})_{k\in\ZZ_+}$, $n\in\NN$,
is nearly critical (see, e.g., Isp\'any \cite[Definition 1.1]{Isp}).
Similarly, if $\EE(\xi_{1,1}^{(n)})=1$, $n\in\NN$, then the corresponding
 sequence of GWI processes is critical (see, e.g., Athreya and Ney \cite[V.3]{AN}).
\proofend
\end{Rem}

The next result establishes limit theorems for the expected value functions of the GWII processes in Theorems \ref{main_1} and \ref{main_2}.
\begin{Thm}\label{exp_thm}
For each $n\in\NN$, let $(X_k^{(n)})_{k\in\ZZ_+}$ be given by \eqref{Xn_def}.
Assume that \textup{\ref{H1}} holds, $\EE(Y_k^{(n)})<\infty$, $k\in\ZZ_+$, $n\in\NN$,
 and there exists a function $\beta\in\DD(\RR_+,\RR)$ such that \eqref{H2_beta} holds.
\begin{enumerate}[label=(\roman*)]
\item
	If there exists an $a\in[0,1)$ such that $\EE(\xi_{1,1}^{(n)})\to a$ as $n\to\infty$,
	 and $\beta$ is continuous with $\beta(0)=0$, then
	\begin{equation*}\label{sc_exp}
	\left(m_n^{-1}\EE\left(
	X_\nt^{(n)}
	\right)\right)_{t\in\RR_+}
	\lu
	\left(	
	\frac{1}{1-a}\beta(t)\right)_{t\in\RR_+}
	 \qquad \text{as $n\to\infty$.}
	\end{equation*}
\item
	If there exist
	 $\gamma\in\RR$ and $\gamma_n\in\RR$, $n\in\NN$,
	 such that $\gamma_n\to\gamma$ as $n\to\infty$
	 and $\EE(\xi_{1,1}^{(n)})=1+\frac{\gamma_n}{n}$, $n\in\NN$, then
	\begin{equation*}\label{cr_exp}
	\left(
	(nm_n)^{-1}\EE\left(
	X_\nt^{(n)}\right)\right)_{t\in\RR_+}
		\lu
	\left(
	\int_0^t\ee^{\gamma(t-s)}\beta(s)\,\dd s\right)_{t\in\RR_+}
	 \qquad \text{as $n\to\infty$.}
	\end{equation*}
\end{enumerate}
\end{Thm}

In the next remark, we point out the fact that an application
of Theorem \ref{main_2} for a critical GWI process having non-degenerate
(i.e., not identically one) offspring distribution results in a $2$-dimensional
identically zero limit process.

\begin{Rem}\label{GWI_1_rem}
Let $(X_k)_{k\in\ZZ_+}$ be a critical GWI process such that
$\EE(\xi_{1,1}^2)<\infty$, $\var(\xi_{1,1})>0$, and $\EE(Y_0)<\infty$.
Let $(X_k^{(n)})_{k\in\ZZ_+}:=(X_k)_{k\in\ZZ_+}$, $n\in\NN$,
and $m_n:=n$, $n\in\NN$.
We check that the assumptions of Theorem \ref{main_2} hold with the choices
$L(t):=0$, $t\in\RR_+$, $\cY_t:=0$, $t\in\RR_+$, $\gamma_n:=0$, $n\in\NN$,
and $\gamma:=0$.
Since $\EE(\xi_{1,1})=1$, \ref{H1} readily holds.
Since
\[
\sup_{s\in[0,t]} \EE\big(m_n^{-1} Y^{(n)}_{\ns}\big) = n^{-1}\EE(Y_0)\to 0 \qquad \text{as $n\to\infty$ for all $t\in\RR_+$,}
\]
we have that \ref{H2} holds with the given function $L$.
Next, we check that
 \begin{align}\label{help_05}
\sup_{t\in[0,T]}m_n^{-1} Y^{(n)}_{\nt}\stoch 0 \qquad \text{as $n\to\infty$ for all $T\in\RR_+$.}
 \end{align}
In what follows, let $T\in\RR_+$ be fixed.
Then, for all $\vare\in\RR_{++}$ and $n\in\NN$, we have that
\begin{align*}
\PP\Big( \sup_{t\in[0,T]}m_n^{-1} Y^{(n)}_{\nt} \geq \vare\Big)
& = \PP\Big(\max_{k\in\{0,\ldots,\nT\}} Y_k \geq m_n\vare\Big)
= 1 - (1 - \PP(Y_0 \geq m_n\vare))^{\nT+1} \\
&=1 - \left(\left(1 - \frac{n\PP( Y_0 \geq m_n\vare)}{n}\right)^{n}\right)^{\frac{\nT+1}{n}}
 \to 1-(\ee^{0})^T = 0 \qquad \text{as $n\to\infty$,}
\end{align*}
since $\lim_{n\to\infty} n\PP(Y_0 \geq m_n\vare) =\lim_{n\to\infty} n\PP(Y_0\geq n\vare) =0$ for all $\vare\in\RR_{++}$, which can be checked as follows.
Since $\PP(Y_0\in\ZZ_+)=1$, we have $\sum_{m=1}^\infty \PP(Y_0\geq m)= \EE(Y_0)<\infty$, where
$\PP(Y_0\geq m)\downarrow 0$ as $m\to\infty$.
Using Cauchy condensation test, it implies that  $\sum_{m=0}^\infty 2^m \PP(Y_0\geq 2^m)<\infty$,
and hence $2^m \PP(Y_0\geq 2^m)\to 0$ as $m\to\infty$.
If $n\in\NN$ is such that $2^m \leq n < 2^{m+1}$ with some $m\in\ZZ_+$, then
$n\PP(Y_0\geq n) \leq 2^{m+1}\PP(Y_0 \geq 2^m)$, yielding that $n\PP(Y_0 \geq n\vare) \to 0$ as $n\to\infty$
as desired.
Using \eqref{help_05} and that $(0)_{t\in\RR_+}\distr(0)_{t\in\RR_+}$ as $n\to\infty$,
 by Jacod and Shiryaev \cite[Chapter VI, Lemma 3.31]{JacShi}, we have that
\[
\left(m_n^{-1}Y_\nt^{(n)}\right)_{t\in\RR_+}
	=(0)_{t\in\RR_+}+\left(m_n^{-1}Y_\nt^{(n)}\right)_{t\in\RR_+}	
	\distr(0)_{t\in\RR_+} \qquad \text{as $n\to\infty$.}
\]
Consequently, we get that \ref{H3} holds with the given $\cY$.
The hypothesis  \ref{H5} readily holds, since $m_n\to\infty$ as $n\to\infty$.
Then Theorem \ref{main_2} implies that
\begin{equation*}
\left(\begin{bmatrix}
m_n^{-1}Y_\nt^{(n)}\\[1mm]
(nm_n)^{-1}X_\nt^{(n)}
\end{bmatrix}\right)_{t\in\RR_+}
=
\left(\begin{bmatrix}
n^{-1}Y_\nt^{(n)}\\[1mm]
n^{-2}X_\nt^{(n)}
\end{bmatrix}\right)_{t\in\RR_+}
\distr	\left(\begin{bmatrix}
0\\
0
\end{bmatrix}\right)_{t\in\RR_+}
\qquad \text{as $n\to\infty$.}
\end{equation*}

We mention that the previous argument remains true with
 other choices of $m_n\in\RR_{++}$, $n\in\NN$, different from $m_n=n$, $n\in\NN$, which
 was used only to show that $n\PP(Y_0\geq m_n\vare)\to0$ as $n\to\infty$.
Suppose that there exists $\delta\in\RR_+$ such that $\EE(Y_0^{1+\delta})<\infty$,
 and $m_n\in\RR_{++}$, $n\in\NN$, is such that $\limsup_{n\to\infty}\frac{n}{m_n^{1+\delta}}=:K<\infty$.
Then we have that $m_n\to\infty$ as $n\to\infty$ and $n\leq (K+1)m_n^{1+\delta}$ for sufficiently large $n\in\NN$.
Consequently, we get that
\begin{align*}
\limsup_{n\to\infty}n\PP(Y_0\geq m_n\vare)
	&=\limsup_{n\to\infty}n\PP(Y_0^{1+\delta}\geq m_n^{1+\delta}\vare^{1+\delta})\\
	&\leq (K+1)\limsup_{n\to\infty}m_n^{1+\delta}\PP(Y_0^{1+\delta}\geq m_n^{1+\delta}\vare^{1+\delta}),
\end{align*}
where we have that
\begin{align*}
\limsup_{n\to\infty}m_n^{1+\delta}
		\PP\left(Y_0^{1+\delta}\geq\lfloor m_n^{1+\delta}\rfloor\vare^{1+\delta}\right)
	&\leq\limsup_{n\to\infty}m_n^{1+\delta}
		\PP\left( \lfloor Y_0^{1+\delta}\rfloor+1 \geq\lfloor m_n^{1+\delta}\rfloor\vare^{1+\delta}\right)\\
	&\leq\limsup_{n\to\infty}(\lfloor m_n^{1+\delta}\rfloor+1)
		\PP( \lfloor Y_0^{1+\delta}\rfloor+1 \geq\lfloor m_n^{1+\delta}\rfloor\vare^{1+\delta})
	=0,
\end{align*}
where the equality follows from $\lim_{\ell\to\infty} \ell\PP( \lfloor Y_0^{1+\delta}\rfloor+1\geq \ell )=0$,
 since $\PP(\lfloor Y_0^{1+\delta}\rfloor+1\in\ZZ_+)=1$ and $\EE(\lfloor Y_0^{1+\delta}\rfloor+1)\leq\EE(Y_0^{1+\delta})+1<\infty$.
\proofend
\end{Rem}

Next, we specialize Theorem \ref{main_2} for $2$-type GWI processes to obtain a
 particular case of Theorem \textup{2.2} of Barczy et al.\ \cite{BarBezPap2}.
Now we recall the notion of $2$-type GWI processes.
Let \ $\big\{\bxi_{k,j,1}, \bxi_{k,j,2}, \, \bvare_k: k, j \in \NN \big\}$
 \ be independent $\ZZ_+^2$-valued random vectors
such that $\big\{\bxi_{k,j,i} : k, j \in \NN\big\}$ \ for each
 \ $i \in \{1, 2\}$, \ and \ $\{\bvare_k : k \in \NN\}$ \ consist of identically distributed random vectors.
We use the notation
\[
    \bxi_{k,j,1} := \begin{bmatrix}
                    \xi_{k,j,1,1} \\
                    \xi_{k,j,1,2}
                   \end{bmatrix} , \qquad
    \bxi_{k,j,2} := \begin{bmatrix}
                    \xi_{k,j,2,1} \\
                    \xi_{k,j,2,2}
                   \end{bmatrix} , \qquad
   \bvare_k := \begin{bmatrix}
                \vare_{k,1} \\
                \vare_{k,2}
               \end{bmatrix}, \qquad k,j\in\NN.
\]
Let $\bX_0:=[X_{0,1},X_{0,2}]^\top:=[0,0]^\top$ and
\begin{align}\label{2X_def}
   \bX_k:=\begin{bmatrix} X_{k,1}\\ X_{k,2}\end{bmatrix}
   = \sum_{j=1}^{X_{k-1,1}}
      \begin{bmatrix} \xi_{k,j,1,1} \\ \xi_{k,j,1,2}\end{bmatrix}
     + \sum_{j=1}^{X_{k-1,2}}
        \begin{bmatrix} \xi_{k,j,2,1} \\ \xi_{k,j,2,2}\end{bmatrix}
     + \begin{bmatrix} \vare_{k,1}\\ \vare_{k,2} \end{bmatrix} , \qquad
   k \in \NN.
 \end{align}
We say that $(\bX_k)_{k\in\ZZ_+}$ is a $2$-type GWI process.

\begin{Cor}\label{main_4}
Let $(\bX_k)_{k\in\ZZ_+}$ be a $2$-type GWI process defined by \eqref{2X_def}.
Suppose that the following conditions hold:
\begin{itemize}
\item
	$\EE(\xi_{1,1,1,1}^2)<\infty$, $\EE(\xi_{1,1,2,2}^2)<\infty$, $\EE(\vare_{1,1}^2)<\infty$, and $\EE(\vare_{1,2})<\infty$,
\item
	$\EE(\xi_{1,1,1,1})=1$, $\EE(\xi_{1,1,2,2})=1$, and $\xi_{1,1,2,1}\ase0$,
\item
	there exists $a_{2,1}\in\RR_+$ such that $\xi_{1,1,1,2}\ase a_{2,1}$.
\end{itemize}
Then
	\begin{equation}\label{2X_conv}
		\left(	\begin{bmatrix}
			n^{-1}X_{\nt,1}\\
			n^{-2}X_{\nt,2}\\
		\end{bmatrix}	\right)_{t\in\RR_+}
		\distr
		\left(	\begin{bmatrix}
			\cX_{t,1}\\
			a_{2,1}\int_0^t\cX_{s,1}\,\dd s
		\end{bmatrix}\right)_{t\in\RR_+} \qquad \text{as $n\to\infty$,}
	\end{equation}
where the stochastic process $(\cX_{t,1})_{t\in\RR_+}$
 is the pathwise unique strong solution of the SDE
\begin{equation}\label{CRI_SDE}
\dd\cX_{t,1}=\EE(\vare_{1,1})\,\dd t+\sqrt{\var(\xi_{1,1,1,1})\cX_{t,1}^+}\,\dd\cW_{t,1}, \qquad t\in\RR_+,\qquad \cX_{0,1}=0,
\end{equation}
where $(\cW_{t,1})_{t\in\RR_+}$ is a standard Wiener process.
\end{Cor}

We call attention to the fact that,
 under second order moment assumption on $\vare_{1,2}$,
 and supposing only that $\EE(\xi_{1,1,1,2})\in\RR_{++}$ and $\EE(\xi_{1,1,1,2}^2)<\infty$
 (instead of $\xi_{1,1,1,2}\ase a_{2,1}$),
 the conclusion of Corollary \ref{main_4} was already proved in Barczy et al.\ \cite[Theorem 2.2]{BarBezPap2}.

\section{Preliminaries for the proofs}\label{Prelims}

In this section, we present some auxiliary results which will be useful in the proofs.
To improve readability, we will state the relevant facts
 only for the GWII process defined in \eqref{X_def},
 although they hold true for each GWII process defined in \eqref{Xn_def} as well
 after a few natural changes (e.g., writing $a_n$ instead of $a$, where the notations $a$ and $a_n$, $n\in\NN$, are introduced below in \eqref{Notation}).
For the same reason, we will only explicitly define
 the martingale differences \eqref{M_def} for the GWII process given in \eqref{X_def}.
In the proofs, we will denote by $(M_k^{(n)})_{k\in\ZZ_+}$, $n\in\NN$,
 the martingale differences corresponding to the GWII processes defined in \eqref{Xn_def},
defined analogously to \eqref{M_def}.

Let $(X_k)_{k\in\ZZ_+}$ be the GWII process defined in \eqref{X_def}, with $(Y_k)_{k\in\ZZ_+}$ its immigration process,
such that $\EE(\xi_{1,1}^2)<\infty$ and $\EE(Y_k)<\infty$, $k\in\ZZ_+$.
We will often make use of the natural filtration of the process $([Y_k,X_k])_{k\in\ZZ_+}$, which we denote by $(\cF_k)_{k\in\ZZ_+}$, where
\begin{equation}\label{filtration}
\cF_k:=\sigma\left(Y_0,X_0,\dots,Y_k,X_k\right), \qquad k\in\ZZ_+.
\end{equation}
Introduce the notations
\begin{equation}\label{Notation}
a:=\EE\left(\xi_{1,1}\right)\in\RR_+,
	\qquad
v:=\var\left(\xi_{1,1}\right)\in\RR_+,
\end{equation}
provided that the corresponding quantities exist and are finite.
When dealing with the GWII processes defined in \eqref{Xn_def}, for each $n\in\NN$,
we denote the corresponding filtration defined in \eqref{filtration} by $(\cF_k^{(n)})_{k\in\ZZ_+}$,
and the corresponding quantities defined in \eqref{Notation} by $a_n$ and $v_n$, respectively.

First, we introduce the martingale differences given by
\begin{equation}\label{M_def}
M_k:=
	X_k-\EE(X_k\mid\cF_{k-1}), \qquad k\in\NN,
\end{equation}
 with $M_0:=0$.
By \eqref{X_def}, using the tower rule, we get that
\begin{equation*}
    \EE(X_k\mid\cF_{k-1})
	=\EE\left(\sum_{j=1}^{X_{k-1}}\xi_{k,j}+Y_{k-1}\,\bigg|\,\cF_{k-1}\right)
	= \EE(\xi_{1,1}) X_{k-1}+Y_{k-1}, \qquad k\in\NN.
\end{equation*}
Hence, using \eqref{M_def} and \eqref{Notation}, we have that
\begin{equation*}\label{M_X}
M_k
	=X_k-aX_{k-1}-Y_{k-1}, \qquad k\in\NN.
\end{equation*}
Consequently, using $X_0=0$ and \eqref{M_def}, we obtain that
\begin{align}\label{X_alt}
\begin{split}
X_k
	&=aX_{k-1}+(M_k+Y_{k-1})
	=a^k X_0+\sum_{j=1}^ka^{k-j}\left(M_j+Y_{j-1}\right)\\
	&=\sum_{j=1}^ka^{k-j}Y_{j-1}+\sum_{j=1}^ka^{k-j}M_j
	=\sum_{j=0}^{k-1}a^{k-1-j}Y_{j}+\sum_{j=1}^ka^{k-j}M_j,
	\qquad k\in\NN.
\end{split}
\end{align}
It is easy to see that these martingale differences are uncorrelated,
since for each $i,j\in\NN$, $i<j$, we have
\begin{equation}\label{M_uncorr}
\EE(M_iM_j)
	=\EE(\EE(M_iM_j\mid\cF_{j-1}))
	=\EE(M_i\EE(M_j\mid\cF_{j-1}))
	=\EE(M_i0)=0.
\end{equation}
Some additional facts about the martingale differences $(M_k)_{k\in\ZZ_+}$,
such as moment estimates, can be found in Appendix \ref{GWII_moments}.

Let us define $S_0:=0$ and
\begin{equation}\label{S_def}
S_k:=\sum_{j=1}^ka^{k-j}M_j,\qquad k\in\NN.
\end{equation}
Note that
\begin{equation}\label{S_rec}
S_k=aS_{k-1}+M_k, \qquad k\in\NN,
\end{equation}
thus $(S_k)_{k\in\ZZ_+}$ is a heteroscedastic autoregressive process of order $1$.

The following lemma will be useful when dealing with stochastic processes
 whose structure is similar to the structure of $(S_k)_{k\in\ZZ_+}$.
\begin{Lem}\label{lemma_submartingale}
Suppose that $q\in\RR_{++}$ and $(Z_k)_{k\in\ZZ_+}$ is a stochastic process adapted to a filtration $(\cH_k)_{k\in\ZZ_+}$
such that $\EE(|Z_k|)<\infty$, $k\in\ZZ_+$, and
\begin{equation*}
\EE(Z_k\mid\cH_{k-1})=qZ_{k-1}, \qquad k\in\NN.
\end{equation*}
Then $((1\vee q^{-k})|Z_k|)_{k\in\ZZ_+}$ is a submartingale with respect to the filtration $(\cH_k)_{k\in\ZZ_+}$.
\end{Lem}
\noindent\textbf{Proof.}
By assumption, we have that $(1\vee q^{-k})|Z_k|$ is integrable for each $k\in\ZZ_+$.
For all $q\in\RR_{++}$, by the conditional Jensen inequality, we have
\begin{align*}
(1\vee q^{-(k-1)})|Z_{k-1}|
	&=|(1\vee q^{-k+1})Z_{k-1}|
	=|(1\vee q^{-k+1})\EE(q^{-1}Z_{k}\mid\cH_{k-1})|\\
	&\leq(1\vee q^{-k+1})\EE(q^{-1}|Z_{k}|\mid\cH_{k-1})
	=\EE((q^{-1}\vee q^{-k})|Z_{k}|\mid\cH_{k-1})\\
	&\leq\EE\big((1\vee q^{-k})|Z_k|\mid\cH_{k-1}\big), \qquad k\in\NN,
\end{align*}
where at the last inequality we used that if $q\leq1$, then $1\leq q^{-1}\leq q^{-k}$ for each $k\in\NN$,
and if $q\geq1$, then $q^{-k}\leq q^{-1}\leq1$ for each $k\in\NN$.
\proofend

Note that, if the conditions of Lemma \ref{lemma_submartingale} are satisfied,
then $(q^{-k}|Z_k|)_{k\in\ZZ_+}$ is a submartingale with respect to the filtration $(\cH_k)_{k\in\ZZ_+}$ as well (following from the proof of Lemma \ref{lemma_submartingale}).

In order to facilitate the application of Lemma \ref{lemma_submartingale} in the proof of Theorem \ref{main_2},
 the following remark establishes that its conditions are satisfied by the stochastic process defined in \eqref{S_def}.
\begin{Rem}\label{remark_submartingale}
If $a=\EE(\xi_{1,1})>0$, then the stochastic process $(S_k)_{k\in\ZZ_+}$ defined in \eqref{S_def} satisfies the conditions of Lemma \ref{lemma_submartingale}
with $q:=a$ and the filtration $(\cF_k)_{k\in\ZZ_+}$ defined in \eqref{filtration}.
Indeed, it is clearly adapted to the filtration $(\cF_k)_{k\in\ZZ_+}$, and, by \eqref{S_rec},
\[
\EE(S_k\mid\cF_{k-1})
	=\EE(aS_{k-1}+M_k\mid\cF_{k-1})
	=aS_{k-1}+\EE(M_k\mid\cF_{k-1})
	=aS_{k-1},\qquad k\in\NN.
\]
Furthermore, by the Jensen inequality and Lemma \ref{A_1}, using that $X_0=0$, we have
\begin{align*}
\left(\EE(|M_k|)\right)^2
	&\leq\EE(M_k^2)
	=\var(M_k)
	=v\EE(X_{k-1})
	=v \sum_{\ell=1}^{k-1}a^{k-1-\ell}\EE(Y_{\ell-1})<\infty, \qquad k\in\NN,
\end{align*}
by assumption, and thus
$\EE(|S_k|)\leq\sum_{j=1}^k a^{k-j}\EE(|M_j|)<\infty$, $k\in\NN$, as desired.
\proofend
\end{Rem}

We will make use of the following lemma several times.
\begin{Lem}\label{sum_int}
For all $f:\ZZ_+\to\RR$, $k\in\{-1\}\cup\ZZ_+$, and $n\in\NN$, we have
\begin{equation*}\label{int_eq}
\sum_{j=0}^kf(j)
	=n\int_{0}^{\frac{k+1}{n}}f(\ns)\,\dd s.
\end{equation*}
\end{Lem}

\noindent\textbf{Proof.}
Let $f:\ZZ_+\to\RR$ and $n\in\NN$.
If $k=-1$, then both sides of the equality to be proved are equal to $0$ (due to the convention $\sum_{j=0}^{-1}:=0$).
If $k\in\NN$, then,
 by the properties of integrals, we have
\begin{align*}
\sum_{j=0}^kf(j)
	&=n\sum_{j=0}^kf(j)\frac{1}{n}
	=n\sum_{j=0}^kf(j)\int_{\frac{j}{n}}^{\frac{j+1}{n}}1\,\dd s\\
	&=n\sum_{j=0}^k\int_{\frac{j}{n}}^{\frac{j+1}{n}}f(\ns)\,\dd s
	=n\int_{0}^{\frac{k+1}{n}}f(\ns)\,\dd s,
\end{align*}
 as desired.
\proofend

In the proofs, we will use the fact that a certain mapping
 is Borel measurable with respect to the $\sigma$-algebra on the set of c\`adl\`ag functions
 (defined in Appendix \ref{CMT}).
For each $d,n\in\NN$ and all $q\in\RR$, we define $\cG_{n,q}:\DD(\RR_+,\RR^d)\to\DD(\RR_+,\RR^d)$,
\begin{equation}\label{cG_def}
\left(\cG_{n,q}(f)\right)(t):=\sum_{j=0}^{\nt-1} q^{\nt-1-j}f(j/n), \qquad t\in\RR_+, \quad f\in\DD(\RR_+,\RR^d).
\end{equation}
Note that the above mapping indeed maps c\`adl\`ag functions to c\`adl\`ag functions,
which follows from the fact that $(\nt)_{t\in\RR_+}$ is c\`adl\`ag, and the product of c\`adl\`ag functions is c\`adl\`ag.
In fact, it holds that $\cG_{n,q}(f)$ is c\`adl\`ag for all (not necessarily c\`adl\`ag functions) $f:\RR_+\to\RR^d$.

\begin{Lem}\label{cG_measurable}
For each $d,n\in\NN$ and all $q\in\RR$, the mapping $\cG_{n,q}$ defined in \eqref{cG_def} is Borel measurable.
\end{Lem}

\noindent\textbf{Proof.}
Let $d,n\in\NN$ and $q\in\RR$ be fixed.
Using that the finite dimensional sets in $\DD(\RR_+,\RR^d)$ generate the Borel $\sigma$-algebra on $\DD(\RR_+,\RR^d)$
(see, e.g., Jacod and Shiryaev \cite[Chapter VI, Theorem 1.14, part c)]{JacShi}), to check the Borel measurability of $\cG_{n,q}$
 it is enough to verify that the mapping $\pi_t\circ \cG_{n,q}:\DD(\RR_+,\RR^d)\to\RR^d$
 is Borel measurable for all $t\in\RR_+$, where $\pi_t$ is the natural projection onto $t$.
Note that $\pi_t$ is Borel measurable for all $t\in\RR_+$
(see, e.g., Billingsley \cite[part (ii) of Theorem 16.6]{Bil}).
Let $t\in\RR_+$ be fixed. 
If $\nt=0$, then, by definition, $(\pi_t\circ\cG_{n,q})(f)=\bzero$ for all $f\in\DD(\RR_+,\RR^d)$,
 and hence, in this case, $\pi_t\circ\cG_{n,q}$ is clearly a Borel measurable mapping.
If $\nt>0$, then, for all $f\in\DD(\RR_+,\RR^d)$, we have that
\begin{align*}
\left(\pi_t\circ\cG_{n,q}\right)(f)
	&=\sum_{j=0}^{\nt-1} q^{\nt-1-j}f(j/n)
	=\sum_{j=0}^{\nt-1} q^{\nt-1-j}\pi_{j/n}(f).
\end{align*}
The mappings $\pi_{j/n}$, $j\in\{0,1,\ldots,\nt-1\}$, are Borel measurable,
 and using that any linear combination of Borel measurable mappings is Borel measurable,
 we have that $\pi_t\circ\cG_{n,q}$ is Borel measurable.
This implies the Borel measurability of $\cG_{n,q}$.
\proofend

\section{Proofs of Propositions \ref{Pro_Y_jump} and \ref{Pro_H2_UI}}\label{Section_prop_proofs}
\noindent\textbf{Proof of Proposition \ref{Pro_Y_jump}.}
Suppose that $\liminf_{n\to\infty}m_n<\infty$ and hypothesis \ref{H3} holds.
Then we have that $\limsup_{n\to\infty}m_n^{-1}>0$, and let $\delta\in(0,1)$
 be such that $\delta<\limsup_{n\to\infty}m_n^{-1}$.
Let $(n_\ell)_{\ell\in\NN}$ be a subsequence in $\NN$
 such that $n_\ell\to\infty$ as $\ell\to\infty$ and $\delta<m_{n_\ell}^{-1}$, $\ell\in\NN$.
By the Portmanteau theorem (see, e.g., Billingsley \cite[Chapter 1, Theorem 2.1]{Bil}),
the convergence in \ref{H3} holds along each subsequence,
therefore we have that
\begin{equation}\label{Y_subs_conv}
(m_{n_\ell}^{-1}Y_{\lfloor n_\ell t\rfloor}^{(n_\ell)})_{t\in\RR_+}\distr(\cY_t)_{t\in\RR_+}
	\qquad \text{as $\ell\to\infty$.}
\end{equation}
Let $g:\RR\to\RR$ be the continuous function
\[
g(x):=\begin{cases}
	0, &\text{if $|x|<\delta/2$,}\\
	2x-\delta, &\text{if $x\in[\delta/2,\delta]$,}\\
	\delta+2x, &\text{if $x\in[-\delta,-\delta/2]$,}\\
	x, &\text{if $|x|>\delta$.}
	\end{cases}
\]
Note that $g$ is a continuous function on $\RR$ vanishing in a neighborhood of $0$.
Consequently, by Jacod and Shiryaev \cite[Chapter VI, Proposition 3.16]{JacShi},
we have that
\begin{equation}\label{Y_jump_conv1}
\left(\begin{bmatrix}
m_{n_\ell}^{-1}Y_{\lfloor n_\ell t\rfloor}^{(n_\ell)}\\
\sum_{s\leq t}g(\Delta (m_{n_\ell}^{-1}Y_{\lfloor n_\ell s\rfloor}^{(n_\ell)}))
\end{bmatrix}\right)_{t\in\RR_+}
	\distr
\left(\begin{bmatrix}
\cY_t\\
\sum_{s\leq t}g(\Delta \cY_s)
\end{bmatrix}\right)_{t\in\RR_+} \qquad \text{as $\ell\to\infty$.}
\end{equation}
Note that the sums in the convergence above are well-defined finite sums,
 since $g$ is identically $0$ on $(-\delta/2,\delta/2)$,
 and every c\`adl\`ag function can have only finitely many jumps
 having absolute value at least $\delta/2$ on any finite interval
 (see, e.g., Billingsley \cite[page 122]{Bil}).
Moreover, since $(Y_k^{(n)})_{k\in\ZZ_+}$ is $\ZZ_+$-valued for each $n\in\NN$,
 the absolute value of each jump of $(m_n^{-1}Y_\nt^{(n)})_{t\in\RR_+}$ is at least $m_n^{-1}$.
Thus, for each $\ell\in\NN$, each jump of
 $(m_{n_\ell}^{-1}Y_{\lfloor n_\ell s \rfloor}^{(n_\ell)})_{s\in\RR_+}$
 has absolute value at least $\delta$.
Furthermore, we check that each jump of $(\cY_t)_{t\in\RR_+}$ has absolute value at least $\delta$ as well.
Let
\[
F_\delta:=\{f\in\DD(\RR_+,\RR): |\Delta f(t)|\in\{0\}\cup[\delta,\infty), t\in\RR_+\},
\]
which is a closed set in $\DD(\RR_+,\RR)$.
Indeed, let $f_n\in F_\delta$, $n\in\NN$, so that there exists $f\in\DD(\RR_+,\RR)$
 such that $f_n\Jto f$ as $n\to\infty$.
If $f$ is continuous, then $\Delta f(t)=0$, $t\in\RR_+$, yielding that
 $f\in F_\delta$, and otherwise, let $t\in\RR_+$
 be such that $|\Delta f(t)|>0$.
By Jacod and Shiryaev \cite[Chapter VI, part a) of Proposition 2.1]{JacShi},
 there exists a sequence $t_n\in\RR_+$, $n\in\NN$
 such that $t_n\to t$ as $n\to\infty$ and $\Delta f_n(t_n)\to\Delta f(t)$ as $n\to\infty$.
But since $|\Delta f_n(t_n)|\in\{0\}\cup[\delta,\infty)$, $n\in\NN$, and $|\Delta f(t)|>0$,
 we have that $|\Delta f(t)|=\lim_{n\to\infty}|\Delta f_n(t_n)|\geq\delta$, hence $f\in F_\delta$, as desired.
Since $\PP((m_{n_\ell}^{-1}Y_{\lfloor n_\ell t\rfloor}^{(n_\ell)})_{t\in\RR_+}\in F_\delta)=1$, $\ell\in\NN$,
 and $F_\delta$ is closed, by the Portmanteau theorem (see, e.g., Billingsley \cite[Chapter 2, Theorem 2.1]{Bil})
 and \eqref{Y_subs_conv}, we have that
\[
\PP((\cY_t)_{t\in\RR_+}\in F_\delta)
	\geq\limsup_{\ell\to\infty}\PP((m_{n_\ell}^{-1}Y_{\lfloor n_\ell t\rfloor}^{(n_\ell)})_{t\in\RR_+}\in F_\delta)=1.
\]
Hence $\PP((\cY_t)_{t\in\RR_+}\in F_\delta)=1$, yielding that the absolute value of each jump of $(\cY_t)_{t\in\RR_+}$ is at least $\delta$.
Consequently, since $g(0)=0$ and $g(x)=x$ for all $|x|\geq\delta$, we have that
\[
\sum_{s\leq t}g(\Delta (m_{n_\ell}^{-1}Y_{\lfloor n_\ell s\rfloor}^{(n_\ell)}))
	=\sum_{s\leq t}\Delta (m_{n_\ell}^{-1}Y_{\lfloor n_\ell s\rfloor}^{(n_\ell)})
	=m_{n_\ell}^{-1}Y_{\lfloor n_\ell t\rfloor}^{(n_\ell)}-m_{n_\ell}^{-1}Y_0^{(n_\ell)}, \qquad t\in\RR_+, \quad \ell\in\NN,
\]
 and $g(\Delta\cY_t)=\Delta\cY_t$, $t\in\RR_+$,
 which together with \eqref{Y_jump_conv1} yield
\begin{equation}\label{jump_conv}
\left(\begin{bmatrix}
m_{n_\ell}^{-1}Y_{\lfloor n_\ell t\rfloor}^{(n_\ell)}\\
m_{n_\ell}^{-1}Y_{\lfloor n_\ell t\rfloor}^{(n_\ell)}-m_{n_\ell}^{-1}Y_{0}^{(n_\ell)}
\end{bmatrix}\right)_{t\in\RR_+}
	\distr
\left(\begin{bmatrix}
\cY_t\\
\sum_{s\leq t}\Delta \cY_s
\end{bmatrix}\right)_{t\in\RR_+} \qquad \text{as $\ell\to\infty$.}
\end{equation}
Since we have shown that the absolute values of all jumps of $(\cY_t)_{t\in\RR_+}$ are at least $\delta$,
 the sum $\sum_{s\leq t}\Delta \cY_s$ has finitely many terms for all $t\in\RR_+$ almost surely
 (and at this point of the proof, we do not know whether this sum equals $\cY_t-\cY_0$ or not).
Let $\widetilde\pi_0:\DD(\RR_+,\RR^2)\to\DD(\RR_+,\RR^4)$,
 $(\widetilde\pi_0(f))(t):=[f(t),f(0)]^\top$, $t\in\RR_+$, $f\in\DD(\RR_+,\RR^2)$.
We check that $\widetilde\pi_0$ is continuous.
Let $f\in\DD(\RR_+,\RR^2)$ and $f_n\in\DD(\RR_+,\RR^2)$, $n\in\NN$,
 be such that $f_n\Jto f$ as $n\to\infty$.
By Jacod and Shiryaev \cite[Chapter VI, part \textup{a)} of Theorem 1.14]{JacShi},
we need to check that there is a sequence of strictly increasing functions
$\lambda_n:\RR_+\to\RR_+$, $n\in\NN$,
 such that $\lambda_n(0)=0$, $\lim_{t\to\infty}\lambda_n(t)=\infty$, $n\in\NN$,
 $\sup_{t\in\RR_+}|\lambda_n(t)-t|\to0$ as $n\to\infty$,
and
\begin{equation}\label{pi0_conv}
 \sup_{t\in[0,K]}\|(\widetilde\pi_0(f_n)\circ\lambda_n)(t)-(\widetilde\pi_0(f))(t)\|\to0 \qquad \text{as $n\to\infty$ for all $K\in\RR_{++}$.}
\end{equation}
Let $K\in\RR_{++}$ be fixed, and let $\lambda_n$, $n\in\NN$, be strictly increasing functions
which satisfy our assumption other than \eqref{pi0_conv}, and
\begin{equation}\label{fn_lambda}
 \sup_{t\in[0,K]}\|(f_n\circ\lambda_n)(t)-f(t)\|\to0 \qquad \text{as $n\to\infty$ for all $K\in\RR_{++}$.}
\end{equation}
Since $f_n\Jto f$ as $n\to\infty$, by Jacod and Shiryaev
 \cite[Chapter VI, part \textup{a)} of Theorem 1.14]{JacShi},
such functions $\lambda_n$, $n\in\NN$, exist.
Then, by the triangle inequality, we have that
\[
 \sup_{t\in[0,K]}\|(\widetilde\pi_0(f_n)\circ\lambda_n)(t)-(\widetilde\pi_0(f))(t)\|
	\leq \sup_{t\in[0,K]}\|(f_n\circ\lambda_n)(t)-f(t)\|+\|f_n(0)-f(0)\|, \qquad n\in\NN.
\]
By \eqref{fn_lambda}, the first term on the right hand side of the above inequality tends to $0$ as $n\to\infty$.
The second term converges to $0$ as $n\to\infty$ since $f_n\Jto f$ as $n\to\infty$
 and the natural projection onto $0$ is continuous
  (see, e.g., Billingsley \cite[Chapter 3, part (i) of Theorem 16.6]{Bil}).
Thus we have that $\widetilde\pi_0$ is continuous.
Since the mapping $F:\RR^4\to\RR$,
 $F([x_1,x_2,x_3,x_4]^\top):=x_1-(x_2+x_3)$, $x_1,x_2,x_3,x_4\in\RR$,
 is continuous, by Lemma \ref{cont_cont}, we get that the mapping
 $\widetilde F:\DD(\RR_+,\RR^4)\to\DD(\RR_+,\RR)$,
 $\widetilde F(f)=F\circ f$, $f\in\DD(\RR_+,\RR^4)$,
 is continuous as well.
Thus, by \eqref{jump_conv}, the continuity of $\widetilde F\circ\widetilde\pi_0$, and
the continuous mapping theorem (see, e.g., Billingsley \cite[Chapter 1, Theorem 2.7]{Bil}),
 we have that
\begin{align}
\begin{split}\label{Y_jump_last}
&\left(0\right)_{t\in\RR_+}
=\widetilde F\left(
\left(\begin{bmatrix}
m_{n_\ell}^{-1}Y_{\lfloor n_\ell t\rfloor}^{(n_\ell)}\\[1mm]
m_{n_\ell}^{-1}Y_{\lfloor n_\ell t\rfloor}^{(n_\ell)}-m_{n_\ell}^{-1}Y_{0}^{(n_\ell)}\\
m_{n_\ell}^{-1}Y_{0}^{(n_\ell)}\\
0
\end{bmatrix}\right)_{t\in\RR_+}\right)\\
&=(\widetilde F\circ\widetilde\pi_0)\left(
\left(\begin{bmatrix}
m_{n_\ell}^{-1}Y_{\lfloor n_\ell t\rfloor}^{(n_\ell)}\\[1mm]
m_{n_\ell}^{-1}Y_{\lfloor n_\ell t\rfloor}^{(n_\ell)}-m_{n_\ell}^{-1}Y_{0}^{(n_\ell)}
\end{bmatrix}\right)_{t\in\RR_+}\right)
\distr
(\widetilde F\circ\widetilde\pi_0)\left(
\left(\begin{bmatrix}
\cY_t\\
\sum_{s\leq t}\Delta \cY_s
\end{bmatrix}\right)_{t\in\RR_+}\right)
\end{split}
\end{align}
 as $\ell\to\infty$.
Since $(0)_{t\in\RR_+}\distr(0)_{t\in\RR_+}$ as $\ell\to\infty$, this implies that
\begin{align}
\begin{split}\label{Y_jump_conv_0}
\left(
\cY_t-\sum_{s\leq t}\Delta \cY_s-\cY_0
\right)_{t\in\RR_+}
&=\widetilde F\left(\left(\begin{bmatrix}
\cY_t\\
\sum_{s\leq t}\Delta \cY_s\\
\cY_0\\
0
\end{bmatrix}\right)_{t\in\RR_+}\right)\\
&=(\widetilde F\circ\widetilde\pi_0)\left(\left(\begin{bmatrix}
\cY_t\\
\sum_{s\leq t}\Delta \cY_s
\end{bmatrix}\right)_{t\in\RR_+}\right)
\distre\left(0\right)_{t\in\RR_+}.
\end{split}
\end{align}
It is easy to see that the singleton $\{(0)_{t\in\RR_+}\}$ is a closed set in $\DD(\RR_+,\RR)$,
 thus it is also a Borel set.
Consequently, by \eqref{Y_jump_conv_0}, we have that
\begin{align*}
1&=\PP\left((0)_{t\in\RR_+}\in\{(0)_{t\in\RR_+}\}\right)
	=\PP\left(\left(
	\cY_t-\sum_{s\leq t}\Delta \cY_s-\cY_0
	\right)_{t\in\RR_+}\in\{(0)_{t\in\RR_+}\}\right)\\
	&=\PP\left(
	\cY_t-\sum_{s\leq t}\Delta \cY_s-\cY_0=0, t\in\RR_+\right)
	=\PP\left(
	\cY_t=\sum_{s\leq t}\Delta \cY_s+\cY_0, t\in\RR_+\right),
\end{align*}
as desired.
The second statement of Proposition \ref{Pro_Y_jump} follows from the fact that if $(\cY_t)_{t\in\RR_+}$
has continuous sample paths, then $\sum_{s\leq t}\Delta\cY_s+\cY_0=\cY_0$, $t\in\RR_+$, almost surely.
\proofend

\noindent\textbf{ Second proof of Proposition \ref{Pro_Y_jump} in the case when $(\cY_t)_{t\in\RR_+}$ has continuous sample paths.}
Suppose that $\liminf_{n\to\infty}m_n<\infty$ and hypothesis \textup{\ref{H3}} holds
 such that $(\cY_t)_{t\in\RR_+}$ has continuous sample paths.
Let $M:=\liminf_{n\to\infty}m_n$, and note that
\begin{equation}\label{help_Y_felbontas}
Y_k^{(n)}=Y_0^{(n)}+(Y_k^{(n)}-Y_0^{(n)}), \qquad k\in\ZZ_+, \quad n\in\NN,
\end{equation}
and, as a consequence of \ref{H3} and the continuity of $\pi_0$
(see, e.g., Billingsley \cite[Chapter 3, part (i) of Theorem 16.6]{Bil}),
we have that $m_n^{-1}Y_0^{(n)} \distr \cY_0$ as $n\to\infty$.
Let $(n_\ell)_{\ell\in\NN}$
be a subsequence in $\NN$ such that $m_{n_\ell}<M+2$ for each $\ell\in\NN$.
Then, since $(Y_k^{(n)})_{k\in\ZZ_+}$ is a $\ZZ_+$-valued stochastic process for each $n\in\NN$, we have that
\begin{align}
\begin{split}\label{Y_const_calc}
\PP\left(m_{n_\ell}^{-1}\sup_{t\in[0,T]}|Y_{\lfloor n_\ell t\rfloor}^{(n_\ell)}-Y_0^{(n_\ell)}|>0\right)
	&\leq\PP\left(m_{n_\ell}^{-1}\sup_{t\in[0,T+n_\ell^{-1}]}|Y_{\lfloor n_\ell t\rfloor}^{(n_\ell)}-Y_0^{(n_\ell)}|>0\right)\\
    &= \PP\left(\sup_{t\in[0,T+n_\ell^{-1}]}|Y_{\lfloor n_\ell t\rfloor}^{(n_\ell)}-Y_0^{(n_\ell)}|>0\right) \\
    &= \PP\Big(\Big\{\exists \, j\in\{0,1,\ldots,\lfloor n_\ell T\rfloor\} : |Y_{j+1}^{(n_\ell)}-Y_0^{(n_\ell)}|>0 \Big\} \Big) \\
    &\leq \PP\Big(\Big\{\exists \, j\in\{0,1,\ldots,\lfloor n_\ell T\rfloor\} : |Y_{j+1}^{(n_\ell)}-Y_j^{(n_\ell)}|>0 \Big\} \Big) \\
    &= \PP\Big(\Big\{\exists \; j\in\{0,1,\ldots,\lfloor n_\ell T\rfloor\} : |Y_{j+1}^{(n_\ell)}-Y_j^{(n_\ell)}|\geq 1 \Big \} \Big) \\
    &\leq \PP\Big( \sup_{t\in[0,T]} |Y_{\lfloor n_\ell t\rfloor +1}^{(n_\ell)}-Y_{\lfloor n_\ell t\rfloor}^{(n_\ell)}|\geq 1 \} \Big) \\
	&\leq \PP\left(m_{n_\ell}^{-1}\sup_{t\in[0,T]}|Y_{{\lfloor n_\ell t\rfloor}+1}^{(n_\ell)}-Y_{\lfloor n_\ell t\rfloor}^{(n_\ell)}|>(M+2)^{-1}\right)
\end{split}
\end{align}
for each $\ell\in\NN$ and all $T\in\RR_{++}$.
Since $(\cY_t)_{t\in\RR_+}$ has continuous sample paths,
 \textup{\ref{H3}} implies that $(m_n^{-1}Y^{(n)}_\nt)_{t\in\RR_+}$, $n\in\NN$, is \textit{C}-tight
 (in the sense of Jacod and Shiryaev \cite[Chapter VI, Definition 3.25]{JacShi}),
 and hence, by Jacod and Shiryaev \cite[Chapter VI, Proposition 3.26]{JacShi}, we have that
\begin{equation*}
\lim_{n\to\infty}\PP\left(m_n^{-1}\sup_{t\in[0,T]}|Y_{{\lfloor nt\rfloor}+1}^{(n)}-Y_{\lfloor nt\rfloor}^{(n)}|>(M+2)^{-1}\right)=0 \qquad \text{for all $T\in\RR_{++}$.}
\end{equation*}
Therefore, using \eqref{Y_const_calc}, we get that
 \begin{equation*}
  \lim_{\ell\to\infty}\PP\left(m_{n_\ell}^{-1}\sup_{t\in[0,T]}|Y_{\lfloor n_\ell t\rfloor}^{(n_\ell)}-Y_0^{(n_\ell)}|>0\right)= 0 \qquad \text{for all $T\in\RR_{++}$.}
  \end{equation*}
Consequently, for all $\vare,T\in\RR_{++}$, we have that
 \begin{equation}\label{Y_const_stoch}
  \PP\left(m_{n_\ell}^{-1}\sup_{t\in[0,T]}|Y_{\lfloor n_\ell t\rfloor}^{(n_\ell)}-Y_0^{(n_\ell)}|\geq \vare\right)
     \leq \PP\left(m_{n_\ell}^{-1}\sup_{t\in[0,T]}|Y_{\lfloor n_\ell t\rfloor}^{(n_\ell)}-Y_0^{(n_\ell)}|>0\right)
      \to 0\quad \text{as $\ell\to\infty$.}
 \end{equation}
Using Jacod and Shiryaev \cite[Chapter VI, Lemma 3.31]{JacShi}, \eqref{help_Y_felbontas}
 and that $(m_n^{-1}Y_0^{(n)})_{t\in\RR_+}\distr(\cY_0)_{t\in\RR_+}$ as $n\to\infty$,
 this implies that $(m_{n_\ell}^{-1}Y_{\lfloor n_\ell t\rfloor}^{(n_\ell)})_{t\in\RR_+}\distr(\cY_0)_{t\in\RR_+}$ as $\ell\to\infty$.
By the Portmanteau theorem (see, e.g., Billingsley \cite[Chapter 1, Theorem 2.1]{Bil}),
the convergence in \textup{\ref{H3}} holds along each subsequence,
therefore we have that $(m_{n_\ell}^{-1}Y_{\lfloor n_\ell t\rfloor}^{(n_\ell)})_{t\in\RR_+}\distr(\cY_t)_{t\in\RR_+}$ as $\ell\to\infty$.
Consequently, we have that $(\cY_t)_{t\in\RR_+}\distre(\cY_0)_{t\in\RR_+}$.

Next, we show that $\PP(\cY_t=\cY_0, t\in\RR_+)=1$.
First, note that $(\cY_t)_{t\in\RR_+}$ and $(\cY_0)_{t\in\RR_+}$ both have continuous sample paths.
By Billingsley \cite[Chapter 3, part (ii) of Theorem 16.6]{Bil},
 $\pi_t-\pi_0$, $t\in\RR_+$, are Borel measurable mappings, yielding that
 $(\pi_t-\pi_0)^{-1}\left(\{0\}\right)$, $t\in\RR_+$, are Borel measurable subsets of $\DD(\RR_+,\RR)$.
Then, since $(\cY_t)_{t\in\RR_+}\distre(\cY_0)_{t\in\RR_+}$, we have
\begin{align}\label{Y_const_help1}
	\begin{split}
	\PP(\cY_t=\cY_0)
		=&\PP(\cY_t-\cY_0=0)
		=\PP\left((\cY_s)_{s\in\RR_+}\in (\pi_t-\pi_0)^{-1}(\{0\})\right)\\
		=&\PP\left((\cY_0)_{s\in\RR_+}\in (\pi_t-\pi_0)^{-1}(\{0\})\right)
		=\PP(\cY_0=\cY_0)
		=1, \qquad t\in\RR_+.
	\end{split}
\end{align}
Since $(\cY_t)_{t\in\RR_+}$ has continuous sample paths, we have that $\PP((\cY_t)_{t\in\RR_+}\in\CC(\RR_+,\RR))=1$, and
 hence we can write
\begin{align*}
\PP\left(\cY_t=\cY_0,t\in\RR_+\right)
	&=\PP\left(\{(\cY_t)_{t\in\RR_+}\in\CC(\RR_+,\RR)\}\cap\{\cY_r=\cY_0,r\in\RR_+\}\right)\\
	&=\PP\left(\{(\cY_t)_{t\in\RR_+}\in\CC(\RR_+,\RR)\}\cap\{\cY_r=\cY_0,r\in\QQ\cap\RR_+\}\right)\\
	&=\PP\left(\cY_r=\cY_0,r\in\QQ\cap\RR_+\right)=1,
\end{align*}
 where at the second equality we used that two functions in $\CC(\RR_+,\RR)$ are equal
 if and only if they coincide on a dense subset of $\RR_+$ (in particular, on $\RR_+\cap\QQ$),
 and at the last equality, \eqref{Y_const_help1} and the fact that
 the intersection of countably many events with probability $1$ is an event with probability $1$ as well.
\proofend

\noindent\textbf{Proof of Proposition \ref{Pro_H2_UI}.}
Let $(\cY_t)_{t\in\RR_+}$ and $(Y_k^{(n)})_{k\in\ZZ_+}$, $n\in\NN$, be given such that \eqref{Y_ui} and \ref{H3} hold.
Note that \eqref{Y_ui} trivially implies that $\EE(Y_k^{(n)})<\infty$, $k\in\ZZ_+$, $n\in\NN$.
Let $\cY_t^{(n)}:=m_n^{-1}Y_\nt^{(n)}$, $t\in\RR_+$, $n\in\NN$.
Then \ref{H3} means that
\begin{equation}\label{H3_rephrase}
(\cY_t^{(n)})_{t\in\RR_+}\distr(\cY_t)_{t\in\RR_+} \qquad \text{as $n\to\infty$,}
\end{equation}
where $(\cY_t)_{t\in\RR_+}$ has continuous sample paths due to our assumption.

First, we show that $\EE(\cY_t)<\infty$, $t\in\RR_+$.
Note that, for all $t\in\RR_+$, $\pi_t$ is continuous on $\CC(\RR_+,\RR)$
(see, e.g., Billingsley \cite[part (i) of Theorem 16.6]{Bil}),
thus \eqref{H3_rephrase} implies that $\cY_t^{(n)}\distr\cY_t$ as $n\to\infty$.
Furthermore, for all $t\in\RR_+$, \eqref{Y_ui} implies that $\cY_t^{(n)}$, $n\in\NN$,
are uniformly integrable. Therefore, by Billingsley \cite[Theorem 3.5]{Bil},
we have that $\EE(\cY_t)<\infty$ and $\EE(\cY_t^{(n)})\to\EE(\cY_t)$ as $n\to\infty$.

Next, we show that the convergence in \eqref{H2_beta} holds with the function $\RR_+\ni t\mapsto\beta(t):=\EE(\cY_t)\in\RR_+$.
Since $\DD(\RR_+,\RR)$ is a separable metric space,
we necessarily have that the support of the distribution of $(\cY_t)_{t\in\RR_+}$ is separable.
Therefore, by the Skorokhod representation theorem (see, e.g., Billingsley \cite[Chapter 1, Theorem 6.7]{Bil}),
there exist stochastic processes $(\cU_t)_{t\in\RR_+}$, $(\cU_t^{(n)})_{t\in\RR_+}$, $n\in\NN$,
defined on an appropriate probability space
(which we assume to be the original one, $(\Omega,\cA,\PP)$ for simplicity)
such that $(\cU_t)_{t\in\RR_+}\distre(\cY_t)_{t\in\RR_+}$
and $(\cU^{(n)}_t)_{t\in\RR_+}\distre(\cY^{(n)}_t)_{t\in\RR_+}$, $n\in\NN$,
and, for each $\omega\in\Omega$,
\begin{equation}\label{U_skor_conv}
(\cU^{(n)}_t(\omega))_{t\in\RR_+}\Jto(\cU_t(\omega))_{t\in\RR_+} \qquad \text{as $n\to\infty$.}
\end{equation}
Since $(\cU_t)_{t\in\RR_+}$ has continuous sample paths almost surely (following from the assumption that $(\cY_t)_{t\in\RR_+}$ possesses this property),
by Jacod and Shiryaev \cite[Chapter VI, part b) of Proposition 1.17]{JacShi},
this implies that
\begin{equation}\label{as_conv}
\sup_{t\in[0,T]}|\cU^{(n)}_t-\cU_t|\as0 \qquad \text{as $n\to\infty$, for all $T\in\RR_{++}$.}
\end{equation}
For the rest of the proof, let $T\in\RR_{++}$ be fixed.
Recall that we have shown that $\EE(\cY_t)<\infty$, $t\in\RR_+$.
Thus, since $\cU_t\distre\cY_t$ and $\cU_t^{(n)}\distre\cY_t^{(n)}$, $n\in\NN$, $t\in\RR_+$, we have that
\begin{equation}\label{YU_diff}
\sup_{t\in[0,T]}|\EE(\cY_t^{(n)})-\EE(\cY_t)|
	=\sup_{t\in[0,T]}|\EE(\cU_t^{(n)})-\EE(\cU_t)|, \qquad n\in\NN.
\end{equation}
By the monotonicity of the expected value, we have that
\begin{equation}\label{ui_bound}
\sup_{t\in[0,T]}\left|\EE(\cU_t^{(n)})-\EE(\cU_t)\right|
	\leq	\EE\left(\sup_{t\in[0,T]}|\cU_t^{(n)}-\cU_t|\right), \qquad n\in\NN.
\end{equation}

We check that $\sup_{t\in[0,T]}|\cU_t^{(n)}-\cU_t|$, $n\in\NN$, are uniformly integrable.
Note that
\begin{equation}\label{U_dominated}
\sup_{t\in[0,T]}|\cU_t^{(n)}-\cU_t|\leq\sup_{t\in[0,T]}\cU_t^{(n)}+\sup_{t\in[0,T]}\cU_t, \qquad n\in\NN,
\end{equation}
where $\sup_{t\in[0,T]}\cU^{(n)}_t$, $n\in\NN$, are uniformly integrable by the hypothesis \eqref{Y_ui}.
Since $(\cU_t)_{t\in\RR_+}$ has almost surely continuous sample paths,
and thus, in particular, the sample paths are continuous at the point $T$ almost surely,
by Jacod and Shiryaev \cite[Chapter VI, Proposition 2.4]{JacShi} and \eqref{U_skor_conv},
we have that $\sup_{t\in[0,T]}\cU^{(n)}_t\as\sup_{t\in[0,T]}\cU_t$ as $n\to\infty$.
Thus, using again the uniform integrability of the random variables
$\sup_{t\in[0,T]}\cU_t^{(n)}$, $n\in\NN$,
by Billingsley \cite[Theorem 3.5]{Bil}, we have that
$\EE(\sup_{t\in[0,T]}\cU_t)<\infty$ and $\EE(\sup_{t\in[0,T]}\cU^{(n)}_t)\to\EE(\sup_{t\in[0,T]}\cU_t)$ as $n\to\infty$.
Therefore, as $\sup_{t\in[0,T]}\cU_t$ is integrable, we get that $\sup_{t\in[0,T]}\cU^{(n)}_t+\sup_{t\in[0,T]}\cU_t$, $n\in\NN$, are uniformly integrable
(see, e.g., Kallenberg \cite[Exercise 3.10]{Kal}).
Thus, by \eqref{U_dominated}, $\sup_{t\in[0,T]}|\cU^{(n)}_t-\cU_t|$, $n\in\NN$, are uniformly integrable, as desired.

Hence, by \eqref{as_conv} and Billingsley \cite[Theorem 3.5]{Bil},
the right hand side of \eqref{ui_bound} converges to $0$ as $n\to\infty$.
Consequently, by \eqref{YU_diff} and \eqref{ui_bound}, we have that
\begin{equation*}
\sup_{t\in[0,T]}|\EE(\cY_t^{(n)})-\EE(\cY_t)|\to0 \qquad \text{as $n\to\infty$,}
\end{equation*}
as desired.

Finally, we prove that the function $\RR_+\ni t\mapsto\beta(t)=\EE(\cY_t)\in\RR_+$, is continuous.
Using that $\cU_t\distre\cY_t$, $t\in[0,T]$, for each $t\in[0,T]$, we have that
\begin{equation}\label{beta_U}
\sup_{\substack{s\in[0,T]\\|t-s|\leq\delta}}|\beta(t)-\beta(s)|
	=\sup_{\substack{s\in[0,T]\\|t-s|\leq\delta}}|\EE(\cU_t)-\EE(\cU_s)|
	\leq\EE\left(\sup_{\substack{s\in[0,T]\\|t-s|\leq\delta}}|\cU_t-\cU_s|\right), \qquad \delta\in\RR_{++}.
\end{equation}
Recall that $(\cU_t)_{t\in\RR_+}$ has continuous sample paths almost surely, therefore,
\begin{equation}\label{U_cont}
\sup_{\substack{s\in[0,T]\\|t-s|\leq\delta}}|\cU_t-\cU_s|\as0 \qquad \text{as $\delta\downarrow0$ for all $t\in[0,T]$.}
\end{equation}
Further, since
\begin{equation*}
\sup_{\substack{s\in[0,T]\\|t-s|\leq\delta}}|\cU_t-\cU_s|\leq2\sup_{s\in[0,T]}\cU_s, \qquad \delta\in\RR_{++},\quad t\in[0,T],
\end{equation*}
where the right hand side is integrable (see the previous part of the proof),
we have that, for each $t\in[0,T]$, the family of random variables
\begin{equation*}
\sup_{\substack{s\in[0,T]\\|t-s|\leq\delta}}|\cU_t-\cU_s|, \qquad \delta\in\RR_{++},
\end{equation*}
is uniformly integrable.
Consequently, using \eqref{U_cont} and Billingsley \cite[Theorem 3.5]{Bil},
we have that 
\begin{equation*}
\EE\left(\sup_{\substack{s\in[0,T]\\|t-s|\leq\delta}}|\cU_t-\cU_s|\right)\to0 \qquad \text{as $\delta\downarrow0$ for all $t\in[0,T]$.}
\end{equation*}
Using \eqref{beta_U}, this yields that $\beta$ is continuous, as desired.
\proofend

\section{Proof of Theorem \ref{main_1}}\label{main_proof_1}
 
First, we prove an auxiliary lemma.
\begin{Lem}\label{sc_lemma}
	Let $d\in\NN$, and let $q\in[0,1)$, $q_n\in\RR_+$, $n\in\NN$, be such that $q_n\to q$ as $n\to\infty$.
	Then, for all $f\in\CC(\RR_+,\RR^d)$ and $f_n\in\DD(\RR_+,\RR^d)$, $n\in\NN$,
	with $f_n\lu f$ as $n\to\infty$, we have that
	\begin{equation*}
	\left(\sum_{j=0}^{\nt-1} q_n^{\nt-1-j}(f_n(j/n)-f_n(0))\right)_{t\in\RR_+}
	\lu
	\left(\frac{1}{1-q}(f(t)-f(0)) \right)_{t\in\RR_+} \qquad \text{as $n\to\infty$.}
	\end{equation*}
\end{Lem}

\noindent\textbf{Proof.}
Let $f\in\CC(\RR_+,\RR^d)$, $f_n\in\DD(\RR_+,\RR^d)$, $n\in\NN$,
be fixed such that $f_n\lu f$ as $n\to\infty$,
and let $q\in[0,1)$, $q_n\in\RR_+$, $n\in\NN$, be fixed such that $q_n\to q$ as $n\to\infty$.
By definition, it is enough to show that
\begin{equation*}
\sup_{t\in[0,T]}\left\|\sum_{j=0}^{\nt-1} q_n^{\nt-1-j}(f_n(j/n)-f_n(0))-\frac{1}{1-q}\left(f(t)-f(0)\right)\right\|\to0 \qquad \text{as $n\to\infty$}
\end{equation*}
for all $T\in\RR_{++}$. For the rest of this proof, let $T\in\RR_{++}$ be fixed.
Since $q\in[0,1)$ and $q_n\to q$ as $n\to\infty$,
 we may assume without loss of generality that $q_n<1$, $n\in\NN$.
Note that $q<1$, $q_n<1$, $n\in\NN$, and $q_n\to q\in[0,1)$ as $n\to\infty$
 implies that $(1-q_n)^{-1}\to(1-q)^{-1}\in[1,\infty)$ as $n\to\infty$.
Since $f$ is continuous, it is uniformly continuous on $[0,T]$.
Let $g:=(1-q)^{-1}(f-f(0))$, and, for each $n\in\NN$ and all $t\in\RR_+$, let
\begin{equation*}
g_n(t):=\sum_{j=0}^{\nt-1}q_n^{\nt-1-j}(f_n(j/n)-f_n(0)) \qquad \text{and} \qquad h_n(t):=\sum_{j=0}^{\nt-1}q_n^{\nt-1-j}(f(j/n)-f(0)).
\end{equation*}
We first show that $\sup_{t\in[0,T]}\left\|g_n(t)-h_n(t)\right\|\to0$ as $n\to\infty$.
By the triangle inequality, we have that
\begin{align*}
&\sup_{t\in[0,T]}\left\|g_n(t)-h_n(t)\right\|
	=\sup_{t\in[0,T]}\left\|\sum_{j=0}^{\nt-1} q_n^{\nt-1-j}\biggl(f_n(j/n)-f_n(0)-f(j/n)+f(0)\biggr)\right\|\\
&\qquad\qquad
	\leq\sup_{t\in[0,T]}\sum_{j=0}^{\nt-1} q_n^{\nt-1-j}\biggl(\left\|f_n(j/n)-f(j/n)\right\|+\left\|f_n(0)-f(0)\right\|\biggr)\\
&\qquad\qquad
	\leq\sup_{t\in[0,T]}\sum_{j=0}^{\nt-1} q_n^{\nt-1-j}2\sup_{s\in[0,T]}\|f_n(s)-f(s)\|\\
&\qquad\qquad =2\sup_{t\in[0,T]}\frac{1-q_n^\nt}{1-q_n}\sup_{s\in[0,T]}\|f_n(s)-f(s)\|\\
&\qquad\qquad\leq\frac{2}{1-q_n}\sup_{t\in[0,T]}\|f_n(t)-f(t)\|\to0 \qquad \text{as $n\to\infty$,}
\end{align*}
 where the convergence follows from the fact that $(1-q_n)^{-1}\to(1-q)^{-1}\in[1,\infty)$ and $f_n\lu f$ as $n\to\infty$.
Thus, since
 \begin{align*}
   &\sup_{t\in[0,T]}\left\|g_n(t)-g(t)\right\|
  	\leq\sup_{t\in[0,T]}\left\|g_n(t)-h_n(t)\right\|
	+\sup_{t\in[0,T]}\left\|h_n(t)-g(t)\right\|, \qquad n\in\NN,
 \end{align*}
 it is enough to verify that $\sup_{t\in[0,T]}\|h_n(t)-g(t)\|\to0$ as $n\to\infty$.
We have that
 \begin{align*}
 h_n(t)
	=&\sum_{j=0}^{\nt-1}q_n^{\nt-1-j}\big(f(j/n)-f(0)\big)
	=\sum_{j=0}^{\nt-1} q_n^{\nt-1-j}\left(\sum_{\ell=1}^jf(\ell/n)-f((\ell-1)/n)\right)\\
	=&\sum_{j=0}^{\nt-1}\sum_{\ell=1}^jq_n^{\nt-1-j}\big(f(\ell/n)-f((\ell-1)/n)\big)
	=\sum_{\ell=1}^{\nt-1}\sum_{j=\ell}^{\nt-1} q_n^{\nt-1-j}\big(f(\ell/n)-f((\ell-1)/n)\big)\\
	=&\sum_{\ell=1}^{\nt-1} \frac{1-q_n^{\nt-\ell}}{1-q_n}\big(f(\ell/n)-f((\ell-1)/n)\big)\\
	=&\frac{1}{1-q_n}\sum_{\ell=1}^{\nt-1}\big(f(\ell/n)-f((\ell-1)/n)\big)-\frac{1}{1-q_n}\sum_{\ell=1}^{\nt-1} q_n^{\nt-\ell}\big(f(\ell/n)-f((\ell-1)/n)\big),
\end{align*}
and thus
\begin{align}\label{hn_decomp}
h_n(t)	=\frac{1}{1-q_n}\left(f\left(0\vee\frac{\nt-1}{n}\right)-f(0)\right)-\frac{1}{1-q_n}\sum_{\ell=1}^{\nt-1} q_n^{\nt-\ell}\big(f(\ell/n)-f((\ell-1)/n)\big)
 \end{align}
 for each $n\in\NN$ and all $t\in[0,T]$.
It is easy to see that $\left((0\vee(\nt-1)/n)\right)_{t\in\RR_+}$ is c\`adl\`ag, non-negative, monotone increasing for each $n\in\NN$, and
\[
\left(0\vee\frac{\nt-1}{n}\right)_{t\in\RR_+}\lu(t)_{t\in\RR_+} \qquad \text{as $n\to\infty$,}
\]
since $\sup_{t\in[0,K]}|(0\vee(\nt-1)/n)-t|\leq2/n\to0$ as $n\to\infty$ for all $K\in\RR_{++}$.
Thus, since $f$ is continuous and $(t)_{t\in\RR_+}$ is non-negative,
 continuous, and strictly monotone increasing,
 part \textup{(ii)} of Lemma \ref{timechange_cont2} imply that
 $(f(0\vee(\nt-1)/n))_{t\in\RR_+}\Jto f$ as $n\to\infty$.
Hence, since $f$ is continuous, by part \textup{(i)} of Lemma \ref{Jto_basic},
 we have that  $(f(0\vee(\nt-1)/n))_{t\in\RR_+}\lu f$ as $n\to\infty$.
Consequently, since $(1-q_n)^{-1}\to(1-q)^{-1}\in[1,\infty)$ as $n\to\infty$,
 part \textup{(ii)} of Lemma \ref{Jto_basic} yields
 that the first term on the right hand side of \eqref{hn_decomp}
 converges locally uniformly to $g$ as $n\to\infty$.
Using again part \textup{(ii)} of Lemma \ref{Jto_basic},
 it remains to show that the second term on the right hand side of \eqref{hn_decomp}
 converges to the $0$ function locally uniformly as $n\to\infty$. 
We have that
\begin{align*}
&\sup_{t\in[0,T]}\left\|\sum_{\ell=1}^{\nt-1} q_n^{\nt-\ell}\left(f(\ell/n)-f((\ell-1)/n)\right)\right\|
	\leq\sup_{t\in[0,T]}\sum_{\ell=1}^{\nt-1} q_n^{\nt-\ell}\|f(\ell/n)-f((\ell-1)/n)\|\\
&\qquad\qquad
	\leq\sup_{t\in[0,T]}\sum_{\ell=1}^{\nt-1} q_n^{\nt-\ell}\sup_{\left\{s,s'\in[0,T]: |s-s'|\leq n^{-1}\right\}}\|f(s)-f(s')\|\\
&\qquad\qquad
 	=\sup_{t\in[0,T]}\frac{q_n-q_n^{\nt}}{1-q_n}\sup_{\left\{s,s'\in[0,T]: |s-s'|\leq n^{-1}\right\}}\|f(s)-f(s')\|\\
&\qquad\qquad
	\leq\frac{q_n}{1-q_n}\sup_{\left\{s,s'\in[0,T]: |s-s'|\leq n^{-1}\right\}}\|f(s)-f(s')\|
	\to\frac{q}{1-q}\cdot 0=0 \qquad \text{as $n\to\infty$,}
\end{align*}
where the convergence holds due to the fact that $q\in[0,1)$, which implies that $q_n(1-q_n)^{-1}\to q(1-q)^{-1}\in\RR_+$ as $n\to\infty$,
 and the fact that $f$ is uniformly continuous on $[0,T]$.
Therefore, by part \textup{(ii)} of Lemma \ref{Jto_basic},
 we get that $h_n\lu g$ as $n\to\infty$, and consequently,
that $g_n\lu g$ as $n\to\infty$, as desired.
\proofend

Now we move on to the proof of Theorem \ref{main_1}.
Recall the notations $a_n=\EE(\xi_{1,1}^{(n)})$ and $v_n=\var(\xi_{1,1}^{(n)})$, $n\in\NN$ from Section \ref{Prelims}.
Note that since $a\in[0,1)$ and $a_n\to a$ as $n\to\infty$, we may assume without loss of generality that $a_n<1$, $n\in\NN$.
By \eqref{X_alt} and $X_0^{(n)}=0$, $n\in\NN$, we have
\begin{equation}\label{sc_X_decomp}
 X_\nt^{(n)}
	= \sum_{j=0}^{\nt-1}a_n^{\nt-1-j}Y_{j}^{(n)}
	+\sum_{j=1}^{\nt}a_n^{\nt-j}M_j^{(n)},
	 \qquad t\in\RR_+,\quad n\in\NN.
\end{equation}
First, we will show that
\begin{align}\label{sc_Y_conv}
\left(m_n^{-1}
	\begin{bmatrix}
	Y_\nt^{(n)}\\[1mm]
	\sum_{j=0}^{\nt-1}a_n^{\nt-1-j}Y_j^{(n)}
	\end{bmatrix}\right)_{t\in\RR_+}
	\distr
	\left(\begin{bmatrix}
	\cY_t\\
	\frac{1}{1-a}\cY_t
	\end{bmatrix}\right)_{t\in\RR_+} \qquad \text{as $n\to\infty$.}
\end{align}
Note that for all $t\in\RR_+$ and each $n\in\NN$, we have that
\begin{align}\label{sc_Y_decomp}
m_n^{-1}\sum_{j=0}^{\nt-1}a_n^{\nt-1-j}Y_{j}^{(n)}
	= m_n^{-1}\sum_{j=0}^{\nt-1}a_n^{\nt-1-j}\left(Y_{\lfloor nj/n\rfloor}^{(n)}-Y_0^{(n)}\right)
	  + m_n^{-1}\sum_{j=0}^{\nt-1}a_n^{\nt-1-j}Y_0^{(n)},
\end{align}
where
\begin{equation*}
\left|m_n^{-1}\sum_{j=0}^{\nt-1}a_n^{\nt-1-j}Y_0^{(n)}\right|
	=\sum_{j=0}^{\nt-1}a_n^{\nt-1-j}m_n^{-1}Y_0^{(n)}
	\leq\frac{1}{1-a_n}m_n^{-1}Y_0^{(n)}, \qquad t\in\RR_+, \quad n\in\NN.
\end{equation*}
Here, since $a\in[0,1)$ and $a_n\to a$ as $n\to\infty$,
we have that $(1-a_n)^{-1}\to(1-a)^{-1}\in[1,\infty)$ as $n\to\infty$.
Further, we have that $m_n^{-1}Y_0^{(n)}\distr\cY_0$ as $n\to\infty$
 by \ref{H3} and the continuity of $\pi_0$ (see, e.g., Billingsley \cite[part (i) of Theorem 16.6]{Bil}),
where $\cY_0=0$ by assumption, and thus $m_n^{-1}Y_0^{(n)}\stoch0$ as $n\to\infty$, yielding that
\begin{equation*}
\sup_{t\in[0,T]}\left|m_n^{-1}\sum_{j=0}^{\nt-1}a_n^{\nt-1-j}Y_0^{(n)}\right|\stoch0 \qquad \text{as $n\to\infty$ for all $T\in\RR_{++}$.}
\end{equation*}
Consequently, by \eqref{sc_Y_decomp} and Jacod and Shiryaev \cite[Chapter VI, Lemma 3.31]{JacShi},
to show \eqref{sc_Y_conv} it is enough to verify that
\begin{align}\label{sc_Y_conv2}
\left(m_n^{-1}
\begin{bmatrix}
Y_\nt^{(n)}\\[2mm]
\sum_{j=0}^{\nt-1}a_n^{\nt-1-j}\left(Y_{\lfloor nj/n\rfloor}^{(n)}-Y_0^{(n)}\right)
\end{bmatrix}
\right)_{t\in\RR_+}
	\distr
	\left(
	\begin{bmatrix}
	\cY_t\\
	\frac{1}{1-a}\cY_t
	\end{bmatrix}\right)_{t\in\RR_+} \qquad \text{as $n\to\infty$.}
\end{align}
Using the notation \eqref{cG_def}, we have that
\begin{align*}
\left(m_n^{-1}\sum_{j=0}^{\nt-1}a_n^{\nt-1-j}\left(Y^{(n)}_{\lfloor nj/n\rfloor}-Y_0^{(n)}\right)\right)_{t\in\RR_+}
	=\cG_{n,a_n}\left(\left(m_n^{-1}Y_\ns^{(n)}-m_n^{-1}Y_0^{(n)}\right)_{s\in\RR_+}\right)
\end{align*}
for each $n\in\NN$.

Next, we apply part \textup{(i)} of Lemma \ref{Conv2Funct} with $d=q=1$ and the following choices:
\begin{itemize}
\item
	$\Phi(f):=(1-a)^{-1}(f-f(0))$, $f\in\DD(\RR_+,\RR)$,
\item
	$\Phi_n(f):=\cG_{n,a_n}(f-f(0))$, $f\in\DD(\RR_+,\RR)$, $n\in\NN$,
\item
	$C:=\CC(\RR_+,\RR)\in\cB(\DD(\RR_+,\RR))$,
\item 
	$(\cU_t)_{t\in\RR_+}:=(\cY_t)_{t\in\RR_+}$,
\item 
	$(\cU_t^{(n)})_{t\in\RR_+}:=(m_n^{-1}Y_\nt^{(n)})_{t\in\RR_+}$, $n\in\NN$.
\end{itemize}
Since the mapping $\DD(\RR_+,\RR)\ni f\mapsto f-f(0)\in\DD(\RR_+,\RR)$ is Borel measurable,
by Lemmas \ref{cG_measurable} and \ref{cont_cont},
we have that the mappings $\Phi$, $\Phi_n$, $n\in\NN$, are Borel measurable.
By Lemma \ref{sc_lemma}, since $a_n\to a\in[0,1)$ as $n\to\infty$, the mappings $\Phi$, $\Phi_n$, $n\in\NN$,
satisfy the conditions of part \textup{(i)} of Lemma \ref{Conv2Funct}.
The stochastic processes $\cU^{(n)}$, $n\in\NN$, are $\RR_+$-valued and have c\`adl\`ag paths,
 and, by \ref{H3} and the assumption that $(\cY_t)_{t\in\RR_+}$ has continuous sample paths,
 we have that $\cU^{(n)}\distr\cU$ as $n\to\infty$ with $\PP(\cU\in C)=1$.
Thus, we can apply part \textup{(i)} of Lemma \ref{Conv2Funct}, which yields \eqref{sc_Y_conv2}, since $\Phi(\cU)=(1-a)^{-1}(\cY-\cY_0)=(1-a)^{-1}\cY$ due to the assumption that $\cY_0=0$.

Having proved \eqref{sc_Y_conv}, by taking into account \eqref{sc_X_decomp}
 and Jacod and Shiryaev \cite[Chapter VI, Lemma 3.31]{JacShi},
 to finish the proof it is enough to show that for all $T\in\RR_{++}$, we have
\begin{equation}\label{help_04}
 m_n^{-1}\sup_{t\in[0,T]}\left|\sum_{j=1}^{\nt}a_n^{\nt-j}M_{j}^{(n)}\right|\stoch0\qquad \text{as $n\to\infty$.}
\end{equation}
Let $T\in\RR_{++}$ be fixed for the rest of this proof.
We have
\begin{align*}
&\sup_{t\in[0,T]}\left|\sum_{j=1}^{\nt}a_n^{\nt-j}M_{j}^{(n)}\right|
	\leq\sup_{t\in[0,T]}\sum_{j=1}^{\nt}a_n^{\nt-j}\left|M_{j}^{(n)}\right|
	\leq\sup_{t\in[0,T]}\sum_{j=1}^{\nt}a_n^{\nt-j}\sup_{s\in[0,T]}\left|M_{\ns}^{(n)}\right|\\
&\qquad\qquad
	=\sup_{t\in[0,T]}\frac{1-a_n^\nt}{1-a_n}\sup_{s\in[0,T]}\left|M_{\ns}^{(n)}\right|
	\leq\frac{1}{1-a_n}\sup_{t\in[0,T]}\left|M_{\nt}^{(n)}\right|, \qquad n\in\NN.
\end{align*}
Consequently, since $a\in[0,1)$ and $a_n\to a$ as $n\to\infty$
 imply that $(1-a_n)^{-1}\to(1-a)^{-1}\in[1,\infty)$ as $n\to\infty$,
 it is enough to show that
\begin{equation*}
m_n^{-1}\sup_{t\in[0,T]}\left|M_{\nt}^{(n)}\right|\stoch0\qquad \text{as $n\to\infty$.}
\end{equation*}
We show the stronger statement that
\begin{equation}\label{sc_M}
\EE\left(m_n^{-2}\sup_{t\in[0,T]} (M_{\nt}^{(n)})^2 \right)\to0\qquad \text{as $n\to\infty$.}
\end{equation}
By simple estimation, together with $M_0^{(n)}=0$, $n\in\NN$, we have that
\begin{equation*}
\EE\left(m_n^{-2}\sup_{t\in[0,T]} (M_{\nt}^{(n)})^2\right)
	=m_n^{-2}\EE\left(\max_{k\in\{0,\dots,\nT\}} (M_{k}^{(n)})^2\right)
	\leq m_n^{-2}\sum_{j=1}^\nT\EE\left( (M_{j}^{(n)})^2\right), \qquad n\in\NN.
\end{equation*}
Note that, by \ref{H2}, we can assume without loss of generality that
\[
\sup_{t\in[0,T]}\EE(m_n^{-1}Y_\nt^{(n)})<L(T)+1, \qquad n\in\NN.
\]
By Lemma \ref{A_1}, we get that
\begin{align*}
m_n^{-2}\sum_{j=1}^\nT\EE\left(\left(M_{j}^{(n)}\right)^2\right)
	=&m_n^{-2}\sum_{j=1}^\nT v_n\EE\left(X_{j-1}^{(n)}\right)
	\leq\frac{v_n\nT}{m_n^2}\max_{j\in\{0,\dots,\nT\}}\EE(X_j^{(n)}),\qquad n\in\NN.
\end{align*}
Using again Lemma \ref{A_1} and $a_n<1$, $n\in\NN$, we have that
\[
\EE(X_k^{(n)})
	=\sum_{\ell=1}^ka_n^{k-\ell}\EE(Y_{\ell-1}^{(n)})
	\leq\frac{1-a_n^k}{1-a_n}\max_{j\in\{0,\dots,k\}}\EE(Y_j^{(n)})
	\leq\frac{1}{1-a_n}\max_{j\in\{0,\dots,k\}}\EE(Y_j^{(n)}), \qquad k\in\NN.
\]
Consequently, we have that
\begin{align*}
&m_n^{-2}\sum_{j=1}^\nT\EE\left(\left(M_{j}^{(n)}\right)^2\right)
	\leq\frac{v_n\nT}{m_n^2}\max_{j\in\{0,\dots,\nT\}}\EE(X_j^{(n)})
	\leq\frac{v_n\nT}{(1-a_n)m_n^2}\max_{j\in\{0,\dots,\nT\}}\EE(Y_j^{(n)})\\
&\qquad\qquad
	\leq\frac{nv_n}{m_n}\cdot\frac{T}{1-a_n}\cdot\sup_{t\in[0,T]}\EE(m_n^{-1}Y_\nt^{(n)})
	\leq\frac{nv_n}{m_n}\cdot\frac{T}{1-a_n}(L(T)+1), \qquad n\in\NN.
\end{align*}
The right hand side of the above inequality converges to $0$ due to \ref{H4},
 since $L(T)<\infty$, and due to $a<1$, we have $(1-a_n)^{-1}\to(1-a)^{-1}\in[1,\infty)$ as $n\to\infty$.
Therefore, we have that \eqref{sc_M} holds, yielding \eqref{help_04}, as desired.
Thus, since \eqref{sc_Y_conv} and \eqref{help_04} together imply \eqref{main_1_X_conv}, we finished the proof of Theorem \ref{main_1}.
\proofend

\section{Proof of Theorem \ref{main_2}}\label{main_proof_2}

First, we prove an auxiliary lemma.
\begin{Lem}\label{cr_lemma}
Let $d\in\NN$, and let $\lambda\in\RR$, $\lambda_n\in\RR$, $n\in\NN$, be such that $\lambda_n\to \lambda$ as $n\to\infty$.
Set $q_n:=1+\frac{\lambda_n}{n}$, $n\in\NN$.
	Then, for all $f\in\DD(\RR_+,\RR^d)$ and $f_n\in\DD(\RR_+,\RR^d)$, $n\in\NN$,
	with $f_n\Jto f$ as $n\to\infty$, we have that
	\begin{equation*}
	\left(n^{-1}\sum_{j=0}^{\nt-1} q_n^{\nt-1-j}f_n(j/n)\right)_{t\in\RR_+}
	\lu
	\left(\int_0^t\ee^{\lambda(t-s)}f(s)\,\dd s\right)_{t\in\RR_+} \quad \text{as $n\to\infty$.}
	\end{equation*}
\end{Lem}

\noindent\textbf{Proof.}
Let $\lambda\in\RR$, $\lambda_n\in\RR$, $n\in\NN$,
 be fixed such that $\lambda_n\to \lambda$ as $n\to\infty$,
and let $f\in\DD(\RR_+,\RR^d)$, $f_n\in\DD(\RR_+,\RR^d)$, $n\in\NN$, be fixed such that $f_n\Jto f$ as $n\to\infty$.
Set $q_n:=1+\frac{\lambda_n}{n}$, $n\in\NN$, and note that, since $q_n\to1$ as $n\to\infty$, we may assume without loss of generality that $q_n>0$, $n\in\NN$.

First, we check that
 \begin{align}\label{help_03}
\left(q_n^\nt\right)_{t\in\RR_+}\lu(\ee^{\lambda t})_{t\in\RR_+}
\qquad \text{and} \qquad 
\left(q_n^{-\nt}\right)_{t\in\RR_+}\lu(\ee^{-\lambda t})_{t\in\RR_+}
 \qquad \text{as $n\to\infty$.}
\end{align}
Note that $q_n^n=(1+\frac{\lambda_n}{n})^n\to\ee^\lambda$ as $n\to\infty$,
and thus $n\ln(q_n)\to\lambda$ as $n\to\infty$.
Let us write
\begin{equation*}
q_n^\nt=\ee^{\ln(q_n)\nt}=\ee^{n\ln(q_n)\frac{\nt}{n}}, \qquad t\in\RR_+, \quad n\in\NN.
\end{equation*}
Since $\sup_{t\in[0,T]}\vert t-\nt/n\vert\leq n^{-1}$, $T\in\RR_+$, $n\in\NN$,
 we have $(\nt/n)_{t\in\RR_+}\lu(t)_{t\in\RR_+}$ as $n\to\infty$,
 and thus, by part \textup{(ii)} of Lemma \ref{Jto_basic}, we have that
\begin{equation*}
\left(\ln(q_n)\nt\right)_{t\in\RR_+}
	=\left(n\ln(q_n)\frac{\nt}{n}\right)_{t\in\RR_+}
	\lu\left(\lambda t\right)_{t\in\RR_+} \qquad \text{as $n\to\infty$.}
\end{equation*}
Consequently, since the exponential function is continuous,
 by part \textup{(i)} of Lemma \ref{Jto_basic} and Lemma \ref{cont_cont}, we get that
\begin{align*}
\left(q_n^\nt\right)_{t\in\RR_+}=\left(\ee^{\ln(q_n)\nt}\right)_{t\in\RR_+}\lu\left(\ee^{\lambda t}\right)_{t\in\RR_+}\qquad \text{as $n\to\infty$.}
\end{align*}
In a similar way, one can show that
\begin{align*}
\left(q_n^{-\nt}\right)_{t\in\RR_+}=\left(\ee^{-\ln(q_n)\nt}\right)_{t\in\RR_+}\lu\left(\ee^{-\lambda t}\right)_{t\in\RR_+}\qquad \text{as $n\to\infty$,}
\end{align*}
as desired.

Note that, using Lemma \ref{sum_int}, for all $t\in\RR_+$ and each $n\in\NN$ we may write
\begin{equation}\label{cr_lemma_int}
n^{-1}\sum_{j=0}^{\nt-1} q_n^{\nt-1-j}f_n(j/n)
	=q_n^{\nt}\int_{0}^{\frac{\nt}{n}}q_n^{-\ns}q_n^{-1}f_n(\ns/n)\,\dd s.
\end{equation}
It is easy to see that $(\nt/n)_{t\in\RR_+}$ is c\`adl\`ag, non-negative, monotone increasing for each $n\in\NN$, and
$(\nt/n)_{t\in\RR_+}\lu(t)_{t\in\RR_+}$ as $n\to\infty$,
 where $(t)_{t\in\RR_+}$ is non-negative, continuous, and strictly monotone increasing 
 such that it takes value $0$ at $0$ and tends to $\infty$ as $t\to\infty$.
Consequently, by part \textup{(ii)} of Lemma \ref{timechange_cont2},
we get that
\begin{equation*}
\left(f_n(\nt/n)\right)_{t\in\RR_+}\Jto f \qquad \text{as $n\to\infty$.}
\end{equation*}
Hence, since $q_n^{-1}\to1$ as $n\to\infty$, by \eqref{help_03} and
 parts \textup{(i)} and \textup{(iii)} of Lemma \ref{Jto_basic}, we have that
\begin{equation*}
\left(q_n^{-\nt}q_n^{-1}f_n(\nt/n)\right)_{t\in\RR_+}\Jto \left(\ee^{-\lambda t}f(t)\right)_{t\in\RR_+} \qquad \text{as $n\to\infty$.}
\end{equation*}
Using again that $(\nt/n)_{t\in\RR_+}\lu(t)_{t\in\RR_+}$ as $n\to\infty$,
 by part \textup{(i)} of Lemma \ref{Jto_basic}, part \textup{(ii)} of Lemma \ref{timechange_cont2},
 and Lemma \ref{int_cont}, we have that
\begin{equation*}
\left(\int_0^\frac{\nt}{n}q_n^{-\ns}q_n^{-1}f_n(\ns/n)\,\dd s\right)_{t\in\RR_+}\lu\left(\int_0^t\ee^{-\lambda s}f(s)\,\dd s\right)_{t\in\RR_+} \qquad \text{as $n\to\infty$.}
\end{equation*}
Finally, since $(q_n^{\nt})_{t\in\RR_+}\lu(\ee^{\lambda t})_{t\in\RR_+}$ as $n\to\infty$ (see \eqref{help_03}),
 where the exponential function is continuous, part \textup{(ii)} of Lemma \ref{Jto_basic} and \eqref{cr_lemma_int} yield the assertion.
\proofend

Now we move on to the proof of Theorem \ref{main_2}.
Recall the notations $a_n=\EE(\xi_{1,1}^{(n)})$ and $v_n=\var(\xi_{1,1}^{(n)})$, $n\in\NN$ from Section \ref{Prelims}.
Note that since $a_n\to1$ as $n\to\infty$, we may assume without loss of generality that $a_n>0$ for all $n\in\NN$.
Recall that, by \eqref{X_alt} and $X^{(n)}_0=0$, $n\in\NN$, we have
\begin{equation}\label{cr_X_decomp}
X_\nt^{(n)}=\sum_{j=0}^{\nt-1} a_n^{\nt-1-j}Y_j^{(n)}+\sum_{j=1}^\nt a_n^{\nt-j}M_{j}^{(n)}, \qquad n\in\NN, t\in\RR_+.
\end{equation}
As a first step, we show that
\begin{equation}\label{cr_Y_conv}
\left(\begin{bmatrix}
m_n^{-1}Y_\nt^{(n)}\\
(nm_n)^{-1}\sum_{j=0}^{\nt-1} a_n^{\nt-1-j}Y_j^{(n)}
\end{bmatrix}\right)_{t\in\RR_+}
\distr
\left(\begin{bmatrix}
\cY_t\\
\ee^{\gamma t}\int_0^t\ee^{-\gamma s}\cY_s\,\dd s
\end{bmatrix}\right)_{t\in\RR_+} \qquad \text{as $n\to\infty$.}
\end{equation}
Note that for all $t\in\RR_+$ and each $n\in\NN$, using the notation \eqref{cG_def}, we have that
\begin{equation*}
(nm_n)^{-1}\sum_{j=0}^{\nt-1} a_n^{\nt-1-j}Y_j^{(n)}
	=\left(n^{-1}\cG_{n,a_n}\left(\left(m_n^{-1}Y_\ns^{(n)}\right)_{s\in\RR_+}\right)\right)(t).
\end{equation*}

Next, we apply part \textup{(ii)} of Lemma \ref{Conv2Funct} with $d=q=1$ and the following choices:
\begin{itemize}
\item
	$(\Phi(f))(t):=\ee^{\gamma t}\int_0^t\ee^{-\gamma s}f(s)\,\dd s$, $t\in\RR_+$, $f\in\DD(\RR_+,\RR)$,
\item
	$\Phi_n(f):=n^{-1}\cG_{n,a_n}(f)$, $f\in\DD(\RR_+,\RR)$, $n\in\NN$,
\item
	$C:=\DD(\RR_+,\RR)\in\cB(\DD(\RR_+,\RR))$,
\item 
	$(\cU_t)_{t\in\RR_+}:=(\cY_t)_{t\in\RR_+}$,
\item 
	$(\cU_t^{(n)})_{t\in\RR_+}:=(m_n^{-1}Y_\nt^{(n)})_{t\in\RR_+}$, $n\in\NN$.
\end{itemize}
By Lemma \ref{cG_measurable}, we have that the mappings $\Phi_n$, $n\in\NN$, are Borel measurable.
Now we verify that the mapping $\Phi$ is also Borel measurable.
The mapping $F_1:\DD(\RR_+,\RR)\to\DD(\RR_+,\RR)$, $F_1(f):=(\ee^{-\gamma t}f(t))_{t\in\RR_+}$, $f\in\DD(\RR_+,\RR)$,
 is Borel measurable.
Indeed, using that the finite dimensional sets in $\DD(\RR_+,\RR)$ generate the Borel $\sigma$-algebra on $\DD(\RR_+,\RR)$
 (see, e.g., Jacod and Shiryaev \cite[Chapter VI, Theorem 1.14, part c)]{JacShi}),
  it is enough to verify that for each $t\in\RR_+$, the mapping
 \[
   \DD(\RR_+,\RR)\ni f\mapsto (\pi_t\circ F_1)(f)=\ee^{-\gamma t}f(t) = \ee^{-\gamma t}\pi_t(f)
 \]
  is Borel measurable, which holds since it is 
  the composition of the Borel measurable mappings $\pi_t$
  and multiplication by the constant $\ee^{-\gamma t}$
  (whose Borel measurability follows, e.g., from Lemma \ref{cont_cont}).
Let us also consider the mapping $F_2:\DD(\RR_+,\RR)\to\DD(\RR_+,\RR)$,
 $F_2(f):=(\ee^{\gamma t}f(t))_{t\in\RR_+}$, $f\in\DD(\RR_+,\RR)$,
 which is also Borel measurable by a similar argument.
Then $\Phi=F_2 \circ \cI^{(1)} \circ F_1$, where $\cI^{(1)}$  (defined in Lemma \ref{int_cont})
 is continuous and thus Borel measurable by Lemma \ref{int_cont},
 yielding that $\Phi$ is Borel measurable.
It is easy to see that $\Phi(\DD(\RR_+,\RR))\subset\CC(\RR_+,\RR)$,
 and thus $C=\DD(\RR_+,\RR)=\Phi^{-1}\left(\CC(\RR_+,\RR)\right)$.
By Lemma \ref{cr_lemma} and part \textup{(i)} of Lemma \ref{Jto_basic},
 since $a_n=1+\gamma_n/n$, $n\in\NN$, where $\gamma_n\to\gamma$ as $n\to\infty$,
 the set $C$ and the mappings $\Phi$, $\Phi_n$, $n\in\NN$,
 satisfy the conditions of part \textup{(ii)} of Lemma \ref{Conv2Funct}.
The stochastic processes $\cU^{(n)}$, $n\in\NN$, are $\RR_+$-valued and have c\`adl\`ag paths,
 and, by \ref{H3}, we have $\cU^{(n)}\distr\cU$ as $n\to\infty$ with $\PP(\cU\in C)=1$.
Thus, we can apply part \textup{(ii)} of Lemma \ref{Conv2Funct}, which yields that \eqref{cr_Y_conv} holds.

Having proved \eqref{cr_Y_conv}, by taking into account \eqref{cr_X_decomp}
 and Jacod and Shiryaev \cite[Chapter VI, Lemma 3.31]{JacShi},
 to finish the proof it is now enough
 to show that for all $T\in\RR_{++}$, we have
\begin{gather}\label{cr_proof_M_stoch}
\begin{split}
(nm_n)^{-1}\sup_{t\in[0,T]}\left|\sum_{j=1}^\nt a_n^{\nt-j}M_{j}^{(n)}\right|
	\stoch0 \qquad \text{as $n\to\infty$.}
\end{split}
\end{gather}
Let $T\in\RR_{++}$ be fixed for the rest of this proof.
We show the stronger statement that
\begin{equation}\label{cr_M}
(nm_n)^{-2}\EE\left(\sup_{t\in[0,T]}\bigg|\sum_{j=1}^\nt a_n^{\nt-j}M_{j}^{(n)}\bigg|^2\right)\to0
	 \qquad \text{as $n\to\infty$.}
\end{equation}
Since $a_n>0$, $n\in\NN$,
 by Lemma \ref{lemma_submartingale} and Remark \ref{remark_submartingale},
 we have that
\begin{equation*}
\left((1\vee a_n^{-k})\bigg|\sum_{j=1}^{k} a_n^{k-j}M_{j}^{(n)}\bigg|\right)_{k\in\ZZ_+}
\end{equation*}
is a submartingale with respect to the filtration $(\cF_k^{(n)})_{k\in\ZZ_+}$ for each $n\in\NN$.
Using that, for $q>0$,
\[
\max_{j\in\{0,\dots,k\}}(1\vee q^{-j})^{-1}
	=\max_{j\in\{0,\dots,k\}}(1\wedge q^{j})\leq1, \qquad k\in\NN,
\]
by Doob's maximal inequality, we get that
\begin{align*}
\EE\left(\sup_{t\in[0,T]}\bigg|\sum_{j=1}^{\nt} a_n^{\nt-j}M_{j}^{(n)}\bigg|^2\right)
	=&\EE\left(\max_{k\in\{0,\dots,\nT\}}\bigg|\frac{1\vee a_n^{-k}}{1\vee a_n^{-k}}\sum_{j=1}^{k} a_n^{k-j}M_{j}^{(n)}\bigg|^2\right)\\
	\leq&\EE\left(\max_{k\in\{0,\dots,\nT\}}\bigg((1\vee a_n^{-k})\bigg|\sum_{j=1}^{k} a_n^{k-j}M_{j}^{(n)}\bigg|\bigg)^2\right)\\
	\leq&\,4\EE\left((1\vee a_n^{-2\nT})\bigg(\sum_{j=1}^{\nT} a_n^{\nT-j}M_{j}^{(n)}\bigg)^2\right).
\end{align*}
Using \eqref{M_uncorr} and \eqref{A_M_var} we get that
\begin{align*}
\EE\left(\bigg(\sum_{j=1}^{\nT} a_n^{\nT-j}M_{j}^{(n)}\bigg)^2\right)
	=&\sum_{j=1}^\nT a_n^{2(\nT-j)}\var\left(M_{j}^{(n)}\right)
	=\sum_{j=1}^\nT a_n^{2(\nT-j)}v_n\EE(X_{j-1}^{(n)}).
\end{align*}
Consequently, we have that
\begin{align}\label{help05}
\EE\left(\sup_{t\in[0,T]}\bigg|\sum_{j=1}^{\nt} a_n^{\nt-j}M_{j}^{(n)}\bigg|^2\right)
	&\leq 4v_n\left(1\vee a_n^{-2\nT}\right)\sum_{j=1}^{\nT} a_n^{2(\nT-j)}\EE(X^{(n)}_{j-1}),
	\qquad n\in\NN.
\end{align}
It is easy to see that for all $q\in\RR_+$ and each $k\in\NN$,
 we have $\max_{j\in\{0,\dots,k\}}q^{k-j}\leq 1\vee q^k$,
 and $a_n^{-2\nT}\to\ee^{-2\gamma T}<\infty$ as $n\to\infty$.
Note that, by \ref{H2}, we may assume without loss of generality that
\[
\sup_{t\in[0,T]}\EE(m_n^{-1}Y_\nt^{(n)})<L(T)+1, \qquad n\in\NN.
\]
By Lemma \ref{A_1}, for each $k,n\in\NN$, we have that
\[
\EE(X_k^{(n)})
	=\sum_{\ell=1}^ka_n^{k-\ell}\EE(Y_{\ell-1}^{(n)})
	\leq k \max_{\ell\in\{1,\dots,k\}}(a_n^{k-\ell}\EE(Y_{\ell-1}^{(n)}))
	\leq k \max_{\ell\in\{0,\dots,k\}}a_n^{k-\ell}\max_{\ell\in\{0,\dots,k\}}\EE(Y_\ell^{(n)}),
\]
and consequently, for each $k,n\in\NN$, we get that
\begin{align*}
\max_{j\in\{0,\dots,k\}}\EE(X_j^{(n)})
	&\leq\max_{j\in\{0,\dots,k\}}\left(j\max_{\ell\in\{0,\dots,j\}}a_n^{j-\ell}\max_{\ell\in\{0,\dots,j\}}\EE(Y_\ell^{(n)})\right)\\
	&\leq\max_{j\in\{0,\dots,k\}}\left(j(1\vee a_n^j)\max_{\ell\in\{0,\dots,j\}}\EE(Y_\ell^{(n)})\right)
	\leq k (1\vee a_n^k)\max_{\ell\in\{0,\dots,k\}}\EE(Y_\ell^{(n)}).
\end{align*}
Therefore, using also \eqref{help05} and that $\max_{j\in\{0,\dots,k\}}q^{k-j}\leq 1\vee q^k$, $q\in\RR_+$, $k\in\NN$, we have that
\begin{align*}
&(nm_n)^{-2}\EE\left(\sup_{t\in[0,T]}\bigg|\sum_{j=1}^\nt a_n^{\nt-j}M_{j}^{(n)}\bigg|^2\right)
	\leq \frac{4v_n}{(nm_n)^2}\left(1\vee a_n^{-2\nT}\right)\sum_{j=1}^{\nT} a_n^{2(\nT-j)}\EE(X^{(n)}_{j-1})\\
&\qquad\qquad
	\leq\frac{4v_n}{(nm_n)^2}\left(1\vee a_n^{-2\nT}\right)\nT(1\vee a_n^{2\nT})\max_{j\in\{0\dots,\nT\}}\EE(X^{(n)}_{j})\\
&\qquad\qquad
	\leq\frac{4v_n}{m_n}\frac{\nT^2}{n^2}\left(1\vee a_n^{-2\nT}\right)(1\vee a_n^{2\nT})(1\vee a_n^\nT)\sup_{t\in[0,T]}\EE(m_n^{-1}Y^{(n)}_{\nt})\\
&\qquad\qquad
	\leq\frac{4v_n}{m_n}T^2\left(a_n^{-2\nT}\vee a_n^{3\nT}\right)(L(T)+1)\\
&\qquad\qquad
	\to 0\cdot T^2(\ee^{-2\gamma T}\vee\ee^{3\gamma T})(L(T)+1)=0 \qquad \text{as $n\to\infty$,}
\end{align*}
where the convergence holds since $m_n^{-1}v_n\to0$ as $n\to\infty$ by the assumption \ref{H5},
yielding \eqref{cr_M}.
Finally, since \eqref{cr_Y_conv} and \eqref{cr_M} together imply \eqref{main_2_X_conv},
we have finished the proof of Theorem \ref{main_2}.
\proofend

\section{Proof of Theorem \ref{exp_thm}}\label{main_proof_3}
Recall the notations $a_n=\EE(\xi_{1,1}^{(n)})$ and $v_n=\var(\xi_{1,1}^{(n)})$, $n\in\NN$ from Section \ref{Prelims}.

First, we prove part \textup{(i)}.
Note that, since $\beta$ is continuous, by part \textup{(i)} of Lemma \ref{Jto_basic},
 \eqref{H2_beta} implies \eqref{H2_beta_cont},
 and since $a\in[0,1)$ and $a_n\to a$ as $n\to\infty$,
 we may assume without loss of generality that $a_n<1$, $n\in\NN$.
By Lemma \ref{A_1} and $X_0^{(n)}=0$, $n\in\NN$, we have
\begin{align*}
\EE(X_\nt^{(n)})
     &  =  \sum_{j=0}^{\nt-1}a_n^{\nt-1-j} \EE(Y^{(n)}_j)
          = \sum_{j=0}^{\nt-1}a_n^{\nt-1-j}\left(\EE(Y^{(n)}_j)-\EE(Y^{(n)}_0)+\EE(Y^{(n)}_0)\right) \\
      &  = \sum_{j=0}^{\nt-1} a_n^{\nt-1-j} (\EE(Y^{(n)}_j) - \EE(Y^{(n)}_0)) 
                + \sum_{j=0}^{\nt-1} a_n^{\nt-1-j}\EE(Y^{(n)}_0)  \\
	 & = \sum_{j=0}^{\nt-1}a_n^{\nt-1-j}\left(\EE(Y^{(n)}_{\lfloor n j/n\rfloor})-\EE(Y^{(n)}_0)\right)+\frac{1-a_n^\nt}{1-a_n}\EE(Y^{(n)}_0),
	\qquad t\in\RR_+, \quad n\in\NN.
\end{align*}
By \eqref{H2_beta_cont}
and $\beta(0)=0$, we have that $m_n^{-1}\EE(Y^{(n)}_0)\to\beta(0)=0$ as $n\to\infty$.
Hence, we have
\begin{equation*}
\sup_{t\in\RR_+}\frac{1-a_n^\nt}{1-a_n}m_n^{-1}\EE(Y^{(n)}_0)\leq\frac{1}{1-a_n}m_n^{-1}\EE(Y^{(n)}_0)\to\frac{1}{1-a}\beta(0)=0 \qquad \text{as $n\to\infty$,}
\end{equation*}
since $a_n\in[0,1)$, $n\in\NN$, and $a_n\to a\in[0,1)$ as $n\to\infty$ implies that $(1-a_n)^{-1}\to(1-a)^{-1}\in[1,\infty)$ as $n\to\infty$.
Thus
\[
\left(m_n^{-1}\frac{1-a_n^\nt}{1-a_n}\EE(Y^{(n)}_0)\right)_{t\in\RR_+}\lu(0)_{t\in\RR_+} \qquad \text{as $n\to\infty$,}
\]
and, by part \textup{(ii)} of Lemma \ref{Jto_basic}, it is enough to show that
\begin{equation*}\label{sc_exp_conv}
\left(\sum_{j=0}^{\nt-1}a_n^{\nt-1-j}\left(\EE(m_n^{-1}Y_{\lfloor n j/n\rfloor}^{(n)})-\EE(m_n^{-1}Y_0^{(n)})\right)\right)_{t\in\RR_+}
	\lu\left(\frac{1}{1-a}\beta(t)\right)_{t\in\RR_+} \qquad \text{as $n\to\infty$.}
\end{equation*}
Due to \eqref{H2_beta_cont}, the continuity $\beta$,
 and the assumptions $\beta(0)=0$ and $a_n\to a\in[0,1)$ as $n\to\infty$,
 this convergence follows from Lemma \ref{sc_lemma}, 
 which yields the assertion.

Next, we prove part (ii).
By Lemma \ref{A_1} and $X_0^{(n)}=0$, $n\in\NN$, we have
\begin{equation}\label{exp_thm_proof_help1}
\EE(X_\nt^{(n)})
	=\sum_{j=0}^{\nt-1}a_n^{\nt-1-j}\EE(Y_{j}^{(n)})
	=\sum_{j=0}^{\nt-1}a_n^{\nt-1-j}\EE(Y_{\lfloor n\frac{j}{n}\rfloor}^{(n)}), \qquad t\in\RR_+, \quad n\in\NN.
\end{equation}
By \eqref{H2_beta}, we have $(m_n^{-1}\EE(Y_{\nt}^{(n)}))_{t\in\RR_+}\Jto\beta$ as $n\to\infty$, where $\beta\in\DD(\RR_+,\RR)$,
 and $a_n=1+\frac{\gamma_n}{n}$, $n\in\NN$, where $\gamma_n\to\gamma\in\RR$ as $n\to\infty$.
Therefore, Lemma \ref{cr_lemma} and \eqref{exp_thm_proof_help1} yield
\[
\left((nm_n)^{-1}\EE(X_\nt^{(n)})\right)_{t\in\RR_+}
	\lu
	\left(\int_0^t\ee^{\gamma(t-s)}\beta(s)\,\dd s\right)_{t\in\RR_+} \qquad \text{as $n\to\infty$,}
\]
as desired.
\proofend

\section{Proof of Corollary \ref{main_4}}\label{main_proof_4}
Note that, by the assumptions that $\xi_{1,1,2,1}\ase0$ and there exists $a_{2,1}\in\RR_+$ such that $\xi_{1,1,1,2}\ase a_{2,1}$,
 we have that
\begin{align}\label{2X_form}
   \begin{bmatrix} X_{k,1}\\ X_{k,2}\end{bmatrix}
   = \sum_{j=1}^{X_{k-1,1}}
      \begin{bmatrix} \xi_{k,j,1,1} \\ a_{2,1}\end{bmatrix}
     + \sum_{j=1}^{X_{k-1,2}}
        \begin{bmatrix} 0 \\ \xi_{k,j,2,2}\end{bmatrix}
     + \begin{bmatrix} \vare_{k,1}\\ \vare_{k,2} \end{bmatrix} , \qquad
   k \in \NN.
 \end{align}
Thus, we have that
\[
X_{k,1}=\sum_{j=1}^{X_{k-1,1}}\xi_{k,j,1,1}+\vare_{k,1}, \qquad k\in\NN,
\]
and
\[
X_{k,2}=\sum_{j=1}^{X_{k-1,2}}\xi_{k,j,2,2}+a_{2,1}X_{k-1,1}+\vare_{k,2},\qquad k\in\NN.
\]
This shows that $(X_{k,1})_{k\in\ZZ_+}$ is a GWI process,
 and, by part \textup{(iii)} of Remark \ref{model_remark},
 we have that $(X_{k,2})_{k\in\ZZ_+}$ is a GWII process with immigration process $(Y_k)_{k\in\ZZ_+}$, where
\[
Y_k:=a_{2,1}X_{k,1}+\vare_{k+1,2}, \qquad k\in\ZZ_+.
\]
Since $\EE(\xi_{1,1,1,1})=1$, $(X_{k,1})_{k\in\ZZ_+}$ is a critical GWI process
 whose offspring and immigration distributions have finite second moments,
 and thus, by Wei and Winnicki \cite[Theorem 2.1]{WW}, we have that
\begin{equation*}
(n^{-1}X_{\nt,1})_{t\in\RR_+}\distr\left(\cX_{t,1}\right)_{t\in\RR_+} \qquad \text{as $n\to\infty$,}
\end{equation*}
where $(\cX_{t,1})_{t\in\RR_+}$ is the pathwise unique strong solution of the SDE \eqref{CRI_SDE}.
Note that, by part \textup{(iii)} of Lemma \ref{Jto_basic}, the mapping
 $\DD(\RR_+,\RR)\ni f\mapsto ([f(t),0]^\top)_{t\in\RR_+}\in\DD(\RR_+,\RR^2)$
 is continuous (take $g$ and $g_n$, $n\in\NN$, to be the identically zero functions),
 and hence the continuous mapping theorem (see, e.g., Billingsley \cite[Chapter 1, Theorem 2.7]{Bil}) yields that
\begin{equation}\label{cor_type1_conv}
\left(\begin{bmatrix}
n^{-1}X_{\nt,1}\\
0
\end{bmatrix}\right)_{t\in\RR_+}
\distr\left(
\begin{bmatrix}\cX_{t,1}\\
0\end{bmatrix}\right)_{t\in\RR_+} \qquad \text{as $n\to\infty$,}
\end{equation}
Moreover, the mapping $\RR\ni x\mapsto a_{2,1}x\in\RR$ is continuous, by Lemma \ref{cont_cont}
 and the continuous mapping theorem (see, e.g., Billingsley \cite[Chapter 1, Theorem 2.7]{Bil}),
 we have that
\begin{equation}\label{cor_type1_a_conv}
(n^{-1}a_{2,1}X_{\nt,1})_{t\in\RR_+}\distr\left(a_{2,1}\cX_{t,1}\right)_{t\in\RR_+} \qquad \text{as $n\to\infty$.}
\end{equation}
Furthermore, for all $T,\theta\in\RR_{++}$, we have that
\begin{align}
\begin{split}\label{cor_vare_stoch0}
\PP\left(\sup_{t\in[0,T]}n^{-1}\vare_{\nt+1,2}>\theta\right)
	&\leq\sum_{k=1}^{\nT+1}\PP(\vare_{k,2}>\theta n) = (\nT+1) \PP\left(\vare_{1,2}>\theta n\right)\\
	&\leq (T+1) n\PP(\vare_{1,2}>\theta n)
	\to 0 \qquad \text{as $n\to\infty$,}
\end{split}
\end{align}
 where the convergence is a consequence of the fact that $\EE(\vare_{1,2})<\infty$  
 (for more details, see Remark \ref{GWI_1_rem} where we derived that
 $n\PP(Y_0\geq n\vare)\to 0$ as $n\to\infty$ for all $\vare\in\RR_{++}$).
Hence, taking into account \eqref{cor_type1_a_conv},
 by Jacod and Shiryaev \cite[Chapter VI, Lemma 3.31]{JacShi},
 we have that
\begin{equation}\label{X2_Y_conv}
\left(n^{-1}Y_{\nt}\right)_{t\in\RR_+}\distr\left(a_{2,1}\cX_{t,1}\right)_{t\in\RR_+} \qquad \text{as $n\to\infty$.}
\end{equation}

Next we apply Theorem \ref{main_2} with $m_n:=n$, $n\in\NN$,
 $(X_k^{(n)})_{k\in\ZZ_+}:=(X_{k,2})_{k\in\ZZ_+}$, $n\in\NN$, and
 $(Y_k^{(n)})_{k\in\ZZ_+}:=(Y_k)_{k\in\ZZ_+}$.
By assumption, \ref{H1} is satisfied since $\EE(\xi_{1,1,2,2})=1<\infty$.
Using \eqref{A_exp}, we have that $\EE(Y_k)=a_{2,1}\EE(\vare_{1,1})k+\EE(\vare_{1,2})$, $k\in\ZZ_+$,
 and thus \ref{H2} holds with the function $L(t):=a_{2,1}\EE(\vare_{1,1})t$, $t\in\RR_+$.
The convergence \eqref{X2_Y_conv} shows that
 \ref{H3} holds with the stochastic process $(\cY_t)_{t\in\RR_+}=(a_{2,1}\cX_{t,1})_{t\in\RR_+}$,
 which has c\`adl\`ag (in fact, continuous) sample paths.
Since $\var(\xi_{1,1,2,2})<\infty$ by assumption and $m_n=n$, $n\in\NN$, \ref{H5} obviously holds.
Further, since $\EE(\xi_{1,1,2,2})=1$, the condition on $\EE(\xi_{1,1}^{(n)})$, $n\in\NN$,
 in Theorem \ref{main_2} holds with $\gamma_n:=0$, $n\in\NN$, $\gamma:=0$.
Consequently, Theorem \ref{main_2} yields that
\begin{equation}\label{X2_XY_conv}
\left(
\begin{bmatrix}
n^{-1}Y_\nt\\
n^{-2}X_{\nt,2}
\end{bmatrix}
\right)_{t\in\RR_+}
\distr
\left(
\begin{bmatrix}
a_{2,1}\cX_{t,1}\\
a_{2,1}\int_0^t\cX_{s,1}\,\dd s
\end{bmatrix}
\right)_{t\in\RR_+} \qquad \text{as $n\to\infty$.}
\end{equation}

Now we turn to show that \eqref{2X_conv} holds.
Using \eqref{cor_vare_stoch0}, \eqref{X2_XY_conv},
 and Jacod and Shiryaev \cite[Chapter VI, Lemma 3.31]{JacShi}, we have that
\begin{align}
\begin{split}\label{cor_aX_conv}
\left(
\begin{bmatrix}
n^{-1}a_{2,1}X_{\nt,1}\\
n^{-2}X_{\nt,2}
\end{bmatrix}
\right)_{t\in\RR_+}
&=
\left(
\begin{bmatrix}
n^{-1}Y_\nt-n^{-1}\vare_{\nt+1,2}\\
n^{-2}X_{\nt,2}
\end{bmatrix}
\right)_{t\in\RR_+}\\
&
\distr
\left(
\begin{bmatrix}
a_{2,1}\cX_{t,1}\\
a_{2,1}\int_0^t\cX_{s,1}\,\dd s
\end{bmatrix}
\right)_{t\in\RR_+} \qquad \text{as $n\to\infty$.}
\end{split}
\end{align}
We consider the cases $a_{2,1}=0$ and $a_{2,1}>0$ separately.
First, suppose that $a_{2,1}=0$. Then we have that
\[
\left(
\begin{bmatrix}
n^{-1}X_{\nt,1}\\
n^{-2}X_{\nt,2}
\end{bmatrix}
\right)_{t\in\RR_+}
=
\left(
\begin{bmatrix}
n^{-1}X_{\nt,1}\\
0
\end{bmatrix}
\right)_{t\in\RR_+}
+
\left(
\begin{bmatrix}
n^{-1}a_{2,1}X_{\nt,1}\\
n^{-2}X_{\nt,2}
\end{bmatrix}
\right)_{t\in\RR_+}, \qquad t\in\RR_+, \quad n\in\NN.
\]
Here, by \eqref{cor_aX_conv}, the second term on the right side converges
 weakly to the two-dimensional identically zero process as $n\to\infty$.
Consequently, since the two-dimensional identically zero process has continuous sample paths,
 and the mapping $\DD(\RR_+,\RR^2)\ni f\mapsto \sup_{t\in[0,T]}\|f(t)\|\in\RR$
 is continuous on $\CC(\RR_+,\RR^2)$ for all $T\in\RR_{++}$
 by Jacod and Shiryaev \cite[Chapter VI, Proposition 2.4]{JacShi},
 using the continuous mapping theorem (see, e.g., Billingsley \cite[Chapter 1, Theorem 2.7]{Bil}),
 we have that
\[
\sup_{t\in[0,T]}\left\|\begin{bmatrix}
n^{-1}a_{2,1}X_{\nt,1}\\
n^{-2}X_\nt
\end{bmatrix}\right\|\distr0 \qquad \text{as $n\to\infty$ for all $T\in\RR_{++}$.}
\]
Since convergence in distribution to a constant random variable
 yields convergence in probability, we have
\[
\PP\left(\sup_{t\in[0,T]}\left\|\begin{bmatrix}
n^{-1}a_{2,1}X_{\nt,1}\\
n^{-2}X_\nt
\end{bmatrix}\right\|>\theta\right)\to0 \qquad \text{as $n\to\infty$ for all $T,\theta\in\RR_{++}$.}
\]
Thus, by \eqref{cor_type1_conv} and Jacod and Shiryaev \cite[Chapter VI, Lemma 3.31]{JacShi},
 we have that
\[
\left(
\begin{bmatrix}
n^{-1}X_{\nt,1}\\
n^{-2}X_{\nt,2}
\end{bmatrix}
\right)_{t\in\RR_+}
\distr
\left(
\begin{bmatrix}
\cX_{t,1}\\
0
\end{bmatrix}
\right)_{t\in\RR_+}
=
\left(
\begin{bmatrix}
\cX_{t,1}\\
a_{2,1}\int_0^t\cX_{s,1}\,\dd s
\end{bmatrix}
\right)_{t\in\RR_+} \qquad \text{as $n\to\infty$,}
\]
as desired.
Finally, suppose that $a_{2,1}>0$.
Then, by the continuity of the mapping $\RR^2\ni[x,y]^\top\mapsto[a_{2,1}^{-1}x,y]^\top\in\RR^2$,
 Lemma \ref{cont_cont}, the continuous mapping theorem
 (see, e.g., Billingsley \cite[Chapter 1, Theorem 2.7]{Bil}),
 and \eqref{cor_aX_conv},
 we have that
\[
\left(
\begin{bmatrix}
n^{-1}X_{\nt,1}\\
n^{-2}X_{\nt,2}
\end{bmatrix}
\right)_{t\in\RR_+}
=
\left(
\begin{bmatrix}
a_{2,1}^{-1}(n^{-1}a_{2,1}X_{\nt,1})\\
n^{-2}X_{\nt,2}
\end{bmatrix}
\right)_{t\in\RR_+}
\distr
\left(
\begin{bmatrix}
\cX_{t,1}\\
a_{2,1}\int_0^t\cX_{s,1}\,\dd s
\end{bmatrix}
\right)_{t\in\RR_+}
\qquad \text{as $n\to\infty$,}
\]
as desired.
\proofend
\vspace*{5mm}

\appendix

\vspace*{5mm}

\noindent{\bf\Large Appendices}

\section{Moments}\label{GWII_moments}

In this appendix, we derive some results regarding the moments of some of the random variables in this paper.

\begin{Lem}\label{A_1}
Let $(X_k)_{k\in\ZZ_+}$ be a GWII process defined in \eqref{X_def}.
Assume that $\EE(\xi_{1,1})<\infty$
and $\EE(Y_k)<\infty$, $k\in\ZZ_+$.
Then, for each $k\in\NN$, we have
\begin{align}
\label{A_cond_exp}
&\EE(X_k\mid\cF_{k-1})
	=X_k-M_k
	=\EE(\xi_{1,1})X_{k-1}+Y_{k-1},\\
\label{A_exp}
&\EE(X_k)
	=\EE(\xi_{1,1})\EE(X_{k-1})+\EE(Y_{k-1})
	=\sum_{\ell=1}^k(\EE(\xi_{1,1}))^{k-\ell}\EE(Y_{\ell-1})
	=\sum_{\ell=0}^{k-1}(\EE(\xi_{1,1}))^{k-1-\ell}\EE(Y_{\ell}).
\end{align}
If, in addition, we have that $\EE(\xi_{1,1}^2)<\infty$, then
\begin{align}
\label{A_X_var}
&\var(X_k\mid\cF_{k-1})
	=\var(M_k\mid\cF_{k-1})
	=\EE(M_k^2\mid\cF_{k-1})
	=\var(\xi_{1,1})X_{k-1},
	\qquad\qquad\qquad\qquad\quad\\
\label{A_M_var}
&\var(M_k)
	=\EE(M_k^2)
	=\var(\xi_{1,1})\EE(X_{k-1}).
\end{align}
\end{Lem}

\noindent\textbf{Proof.}
The first equality in \eqref{A_cond_exp} follows from the definition \eqref{M_def} of $M_k$,
and the second equality can be directly checked as follows:
\begin{align*}
\EE(X_k\mid\cF_{k-1})
	=\EE\left(  \sum_{j=1}^{X_{k-1}}\xi_{k,j} + Y_{k-1}  \,\bigg|\,\cF_{k-1}\right)
	=X_{k-1}\EE(\xi_{1,1})+Y_{k-1}, \qquad k\in\NN.
\end{align*}
Taking the expected value of the left and right hand sides of \eqref{A_cond_exp} yields the first equality in \eqref{A_exp}:
\begin{align*}
\EE(X_k)
	=&\EE(\EE(X_k\mid\cF_{k-1}))
	=\EE(\xi_{1,1})\EE(X_{k-1})+\EE(Y_{k-1}), \qquad k\in\NN.
\end{align*}
We check the second equality in \eqref{A_exp} by induction, using the first equality in \eqref{A_exp} and the fact that $X_{0}=0$.
Note that we have
\begin{equation}\label{A_exp_ind_base}
\EE(X_1)=\EE(\xi_{1,1})\EE(X_0)+\EE(Y_0)=\EE(Y_0)=\sum_{\ell=1}^1\left(\EE(\xi_{1,1})\right)^{1-\ell}\EE(Y_{\ell-1}),
\end{equation}
which shows that the second equality in \eqref{A_exp} holds for $k=1$.
Now suppose that $\EE(X_k)=\sum_{\ell=1}^k\left(\EE(\xi_{1,1})\right)^{k-\ell}\EE(Y_{\ell-1})$ for some particular $k\in\NN$.
Then we have that
\begin{align*}
\EE(X_{k+1})
	=&\EE(\xi_{1,1})\EE(X_k)+\EE(Y_k)
	=\EE(\xi_{1,1})\sum_{\ell=1}^k\left(\EE(\xi_{1,1})\right)^{k-\ell}\EE(Y_{\ell-1})+\EE(Y_k)\\
	=&\sum_{\ell=1}^{k+1}\left(\EE(\xi_{1,1})\right)^{k+1-\ell}\EE(Y_{\ell-1}),
\end{align*}
as desired.
The third equality in \eqref{A_exp} follows by a shift of indices.
For the rest of this proof, suppose in addition that $\EE(\xi_{1,1}^2)<\infty$.
Note that, by \eqref{A_cond_exp}, we have that
\begin{align*}
&\var(X_k\mid\cF_{k-1})
	=\EE\left(\left(X_k-\EE(X_k\mid\cF_{k-1})\right)^2\bigg|\cF_{k-1}\right)
	=\EE(M_k^2\mid\cF_{k-1})\\
&\qquad
	=\EE(\left(X_k-\EE(\xi_{1,1})X_{k-1}-Y_{k-1}\right)^2\mid\cF_{k-1})
	=\EE\left(\left(\sum_{j=1}^{X_{k-1}}(\xi_{k,j}-\EE(\xi_{1,1}))\right)^2\,\bigg|\,\cF_{k-1}\right)\\
&\qquad	=\EE\left(\sum_{\substack{i,j=1\\i\neq j}}^{X_{k-1}}(\xi_{k,i}-\EE(\xi_{1,1}))(\xi_{k,j}-\EE(\xi_{1,1}))+\sum_{j=1}^{X_{k-1}}(\xi_{k,j}-\EE(\xi_{1,1}))^2\,\bigg|\,\cF_{k-1}\right),
	\qquad k\in\NN.
\end{align*}
Using the independence of $\xi_{k,j}$ and $\xi_{k,i}$
(and thus $\EE((\xi_{k,i}-\EE(\xi_{1,1}))(\xi_{k,j}-\EE(\xi_{1,1})))=0$) for $i\neq j$, $i,j,k\in\NN$,
we get that
\begin{equation*}
\var(X_k\mid\cF_{k-1})
	=\EE\left(\sum_{j=1}^{X_{k-1}}(\xi_{k,j}-\EE(\xi_{1,1}))^2\,\bigg|\,\cF_{k-1}\right)
	=\var(\xi_{1,1})X_{k-1}, \qquad k\in\NN,
\end{equation*}
yielding \eqref{A_X_var}. Taking expectations of both sides of \eqref{A_X_var} implies \eqref{A_M_var}.
\proofend

\section{Mapping theorems}
\label{CMT}

In this appendix, we recall some results about locally uniform convergence,
 convergence in the Skorokhod $J_1$ topology,
 and establish the Borel measurability of some mappings.

Recall that for $d\in\NN$, $\DD(\RR_+,\RR^d)$ and $\CC(\RR_+,\RR^d)$
 denote the set of $\RR^d$-valued c\`adl\`ag and continuous functions, respectively,
 and, for all $t\in\RR_+$, $\pi_t$ is the natural projection onto $t$.
Let $\cB(\DD(\RR_+, \RR^d))$ denote the Borel $\sigma$-algebra on
 $\DD(\RR_+, \RR^d)$ for the metric defined in Jacod and Shiryaev
 \cite[Chapter VI, (1.26)]{JacShi} (see also Billingsley \cite[Chapter 3, Section 16]{Bil}).
With this metric $\DD(\RR_+, \RR^d)$ is a
 complete and separable metric space and the topology induced by this metric is
 the so-called Skorokhod $J_1$ topology.
Note that $\CC(\RR_+,\RR^d)\in\cB(\DD(\RR_+,\RR^d))$, see, e.g., Ethier and Kurtz \cite[Problem 3.11.25]{EthKur}.

In the following lemma, we collect some simple facts about convergence in the Skorokhod $J_1$ topology
 and locally uniform convergence
 (see, e.g., Jacod and Shiryaev \cite[Chapter VI, Propositions 1.17, 1.23, parts \textup{(a)} and \textup{(b.5)} of 2.1, and part \textup{(ii)} of 2.2]{JacShi},
 Remmert \cite[Chapter 3.1.5, Composition Theorem, page 99]{Rem},
 Apostol \cite[Chapter 9, Exercise 9.4]{Apo}, Thomson et al. \cite[Exercise 9.3.19]{ThomBruck}
 and Whitt \cite[Theorems 4.1 and 4.2]{Whitt}).

\begin{Lem}\label{Jto_basic}
Let $d,q\in\NN$, and let $f\in\DD(\RR_+,\RR^d)$, $f_n\in\DD(\RR_+,\RR^d)$, $n\in\NN$,
 and $g\in\DD(\RR_+,\RR^q)$, $g_n\in\DD(\RR_+,\RR^q)$, $n\in\NN$.
\begin{enumerate}[label=(\roman*)]
\item
	If $f$ is continuous, then $f_n\Jto f$ as $n\to\infty$ if and only if $f_n\lu f$ as $n\to\infty$.
\item
	Suppose that $f_n\lu f$ and $g_n\lu g$ as $n\to\infty$.
	Then $[f_n, g_n]^\top\lu[f,g]$ as $n\to\infty$.
	Further, if $q=d$, then $f_n+g_n\lu f+g$ as $n\to\infty$,
	 and if $q=1$ and $g$ is continuous, then $f_n\cdot g_n\lu f\cdot g$ as $n\to\infty$.
\item
	Suppose that $f_n\Jtoo{d} f$ and $g_n\Jtoo{q} g$ as $n\to\infty$, and $g\in\CC(\RR_+,\RR^q)$.
	Then $[f_n, g_n]^\top\Jtoo{d+q}[f,g]$ as $n\to\infty$.
	Further, if $q=d$, then $f_n+g_n\Jtoo{d}f+g$ as $n\to\infty$,
	 and if $q=1$, then $f_n\cdot g_n\Jtoo{d}f\cdot g$ as $n\to\infty$.
\end{enumerate}
\end{Lem}

For $\RR^d$-valued stochastic processes $(\bcY_t)_{t \in \RR_+}$ and
 $(\bcY^{(n)}_t)_{t \in \RR_+}$, $n \in \NN$, with c\`adl\`ag paths, we write
 $\bcY^{(n)} \distr \bcY$ as $n\to\infty$ if the distribution of $\bcY^{(n)}$ on the
 space $(\DD(\RR_+,  \RR^d), \cB(\DD(\RR_+, \RR^d)))$ converges weakly to the
 distribution of $\bcY$ on the space
 $(\DD(\RR_+,  \RR^d), \cB(\DD(\RR_+, \RR^d)))$ as $n \to \infty$.
If $\xi$ and $\xi_n$, $n \in \NN$, are random elements with values in a metric space $(E,d)$,
 then we denote by $\xi_n \distr \xi$ as $n\to\infty$ the weak convergence of the
 distribution of $\xi_n$ on the space $(E, \cB(E))$ towards the
 distribution of $\xi$ on the space $(E, \cB(E))$ as $n \to \infty$,
 where $\cB(E)$ denotes the Borel $\sigma$-algebra on $E$ induced by
 the given metric $d$.

One can formulate the following versions of
 the continuous mapping theorem, which easily follow from Kallenberg \cite[Theorem 3.27]{Kal}.
Given $d,q\in\NN$, for Borel measurable mappings $\Phi : \DD(\RR_+, \RR^d) \to \DD(\RR_+, \RR^q)$
 and $\Phi_n : \DD(\RR_+, \RR^d) \to \DD(\RR_+, \RR^q)$, $n \in \NN$,
 we will denote by $C_{\Phi, (\Phi_n)_{n \in \NN}}$ the set of all functions
 $f \in \CC(\RR_+, \RR^d)$ for which $\Phi_n(f_n) \lu \Phi(f)$ as $n\to\infty$ whenever
 $f_n \lu f$ as $n\to\infty$ with $f_n \in \DD(\RR_+, \RR^d)$, $n \in \NN$.

\begin{Lem}\label{Conv2Funct}
Let $d,\,q\in\NN$.
Let $\Phi : \DD(\RR_+, \RR^d) \to \DD(\RR_+, \RR^q)$ and
  $\Phi_n : \DD(\RR_+, \RR^d) \to \DD(\RR_+, \RR^q)$, $n \in \NN$, be Borel measurable mappings.
Let $(\bcU_t)_{t \in \RR_+}$ and $(\bcU^{(n)}_t)_{t \in \RR_+}$, $n \in \NN$,
 be $\RR^d$-valued stochastic processes with c\`adl\`ag paths such that
 $\bcU^{(n)} \distr \bcU$ as $n\to\infty$.

\begin{enumerate}[label=(\roman*)]
\item
	Suppose that there exists $C\in\cB(\DD(\RR_+,\RR^d))$
	 such that  $\PP(\bcU \in C)=1$ and $C\subset C_{\Phi, (\Phi_n)_{n \in \NN}}$.
	 Then $[\bcU^{(n)},\Phi_n(\bcU^{(n)})]^\top \distr [\bcU,\Phi(\bcU)]^\top$ as $n\to\infty$.
\item
	Suppose that there exists $C\in\cB(\DD(\RR_+,\RR^d))$
	 such that  $\PP(\bcU \in C)=1$, $C\subset \CC(\RR_+,\RR^d)\cup\Phi^{-1}\left(\CC(\RR_+,\RR^q)\right)$,
	 and for all $f\in C$, we have $\Phi_n(f_n)\Jtoo{q}\Phi(f)$ as $n\to\infty$
	 whenever $f_n\Jtoo{d} f$ as $n\to\infty$ with $f_n\in\DD(\RR_+,\RR^d)$, $n\in\NN$.
	Then $[\bcU^{(n)},\Phi_n(\bcU^{(n)})]^\top \distr [\bcU,\Phi(\bcU)]^\top$ as $n\to\infty$.
\end{enumerate}
\end{Lem}

\noindent\textbf{Proof.}
We first show that part \textup{(i)} is a special case of part \textup{(ii)}.
Let $C$ satisfy the conditions of part $\textup{(i)}$.
Then we have that $C\subset C_{\Phi,(\Phi_n)_{n\in\NN}}\subset \CC(\RR_+,\RR^d)$,
 thus, by part \textup{(i)} of Lemma \ref{Jto_basic}, whenever $f\in C$ and $f_n\in\DD(\RR_+,\RR^d)$, $n\in\NN$,
 are such that $f_n\Jtoo{d} f$ as $n\to\infty$, then we also have that $f_n\lu f$ as $n\to\infty$.
Therefore, by the definition of $C_{\Phi,(\Phi_n)_{n\in\NN}}$,
 we have that $\Phi_n(f_n)\lu\Phi(f)$ as $n\to\infty$.
Since, by Jacod and Shiryaev \cite[Chapter VI, part a) of Proposition 1.17]{JacShi},
 $\Phi_n(f_n)\lu\Phi(f)$ as $n\to\infty$ implies that $\Phi_n(f_n)\Jtoo{q}\Phi(f)$ as $n\to\infty$,
 the set $C$ satisfies the conditions of part \textup{(ii)}.
 Thus, it is enough to prove part \textup{(ii)}.

Now, suppose that $C$ satisfies the conditions of part \textup{(ii)},
 and let $\widetilde\Phi:\DD(\RR_+,\RR^d)\to\DD(\RR_+,\RR^{d+q})$
 and $\widetilde\Phi_n:\DD(\RR_+,\RR^d)\to\DD(\RR_+,\RR^{d+q})$, $n\in\NN$, be defined by
\begin{equation*}
\widetilde\Phi(g):=
	\begin{bmatrix}
	g\\
	\Phi(g)
	\end{bmatrix},
\qquad
\widetilde\Phi_n(g):=
	\begin{bmatrix}
	g\\
	\Phi_n(g)
	\end{bmatrix},
\qquad n\in\NN, \qquad g\in\DD(\RR_+,\RR^d).
\end{equation*}
Let $f\in C$ and $f_n\in\DD(\RR_+,\RR^d)$, $n\in\NN$, be such that $f_n\Jtoo{d} f$ as $n\to\infty$.
Then, by assumption, we have that $\Phi_n(f_n)\Jtoo{q}\Phi(f)$ as $n\to\infty$.
Since $f\in C$, at least one of the functions $f$ and $\Phi(f)$ is continuous,
 and hence, by part \textup{(iii)} of Lemma \ref{Jto_basic},
 we have that $\widetilde\Phi_n(f_n)\Jtoo{d+q}\widetilde\Phi(f)$ as $n\to\infty$.
The mappings $\widetilde{\Phi}$ and $\widetilde{\Phi}_n$, $n\in\NN$, are Borel measurable,
 which follows from Kallenberg \cite[Lemma 1.8]{Kal}, due to the Borel measurability of the identity map
 and the Borel measurability of $\Phi$ and $\Phi_n$, $n\in\NN$, respectively.
By Kallenberg \cite[Theorem 3.27]{Kal}, this implies the assertion of part \textup{(ii)}.

We mention that the assertion of part \textup{(i)} can also be directly derived from Kallenberg \cite[Theorem 3.27]{Kal}.
\proofend

Recall the following result concerning the composition of c\`adl\`ag functions
 from the left with a continuous function
 (see, e.g., Ethier and Kurtz \cite[Problem 3.11.13]{EthKur}).

\begin{Lem}\label{cont_cont}
Let $d,q\in\NN$, and let $\varphi:\RR^d\to\RR^q$ be continuous.
Then the mapping
 $\DD(\RR_+,\RR^d)\ni f\mapsto \varphi\circ f=\left(\varphi(f(t))\right)_{t\in\RR_+}\in\DD(\RR_+,\RR^q)$
 is continuous.
\end{Lem}

Next, we provide an equivalent characterization of
 convergence in the Skorokhod $J_1$ topology,
 which can be interesting on its own right.

\begin{Lem}\label{alt_J1}
Let $d\in\NN$, and let $f\in\DD(\RR_+,\RR^d)$ and $f_n\in\DD(\RR_+,\RR^d)$, $n\in\NN$.
Then $f_n\Jto f$ as $n\to\infty$ if and only if
 there exist non-negative monotone increasing functions
 $h_n\in\DD(\RR_+,\RR)$, $n\in\NN$, such that
\begin{equation}\label{h_cond1}
\sup_{t\in[0,T]}|h_n(t)-t|\to0 \qquad \text{as $n\to\infty$ for all $T\in\RR_{++}$,}
\end{equation}
and
\begin{equation}\label{h_cond2}
\sup_{t\in[0,T]}\|f_n(t)-f(h_n(t))\|\to0
	\qquad \text{as $n\to\infty$ for all $T\in\RR_{++}$.}
\end{equation}
\end{Lem}
\noindent\textbf{Proof.}
First, suppose that $f_n\Jto f$ as $n\to\infty$.
By Jacod and Shiryaev \cite[Chapter VI, Theorem 1.14]{JacShi},
 there exists a sequence of continuous, strictly increasing functions
 $\lambda_n:\RR_+\to\RR_+$, $n\in\NN$, with $\lambda_n(0)=0$ and $\lim_{t\to\infty}\lambda_n(t)=\infty$, $n\in\NN$,
 such that $\sup_{t\in\RR_+}|\lambda_n(t)-t|\to0$ as $n\to\infty$, and
\[
\sup_{t\in[0,T]}\Vert f_n(\lambda_n(t))-f(t)\Vert  \to0 \qquad \text{as $n\to\infty$ for all $T\in\RR_{++}$.}
\]
It is clear that $\lambda_n$, $n\in\NN$, are invertible,
 and $\lambda_n^{-1}$, $n\in\NN$,
 are non-negative, strictly monotone increasing and continuous
 (in particular, c\`adl\`ag) functions defined on $\RR_+$ with range $\RR_+$.
Furthermore, we have that
\[
\sup_{t\in\RR_+}|\lambda_n^{-1}(t)-t|
	=\sup_{t\in \RR_+}|\lambda^{-1}_n(t)-\lambda_n(\lambda_n^{-1}(t))|
	=\sup_{s\in\lambda_n^{-1}(\RR_+)}|s-\lambda_n(s)|
	=\sup_{t\in \RR_+}|\lambda_n(t)-t|
	\to0 \quad \text{as $n\to\infty$,}
\]
and
\begin{align*}
\sup_{t\in[0,T]}\Vert  f_n(t)-f(\lambda_n^{-1}(t)) \Vert
	&=\sup_{t\in[0,T]}\Vert  f_n(\lambda_n(\lambda_n^{-1}(t)))-f(\lambda_n^{-1}(t)) \Vert
	=\sup_{s\in \lambda_n^{-1}([0,T])}\Vert  f_n(\lambda_n(s))-f(s)\Vert \\
	&=\sup_{t\in [0,\lambda_n^{-1}(T)]}\Vert f_n(\lambda_n(t))-f(t)\Vert
	\to0 \qquad \text{as $n\to\infty$ for all $T\in\RR_{++}$,}
\end{align*}
 which shows that \eqref{h_cond1} and \eqref{h_cond2} are satisfied
 with the functions $h_n:=\lambda_n^{-1}$, $n\in\NN$, as desired.

Next, suppose that there exist non-negative monotone increasing functions
 $h_n\in\DD(\RR_+,\RR)$, $n\in\NN$, such that \eqref{h_cond1} and \eqref{h_cond2} hold.
We need to show that $f_n\Jto f$ as $n\to\infty$.
We check that it is enough to prove that
 for all continuity points $T\in\RR_{+}$ of $f$ and $\vare\in\RR_{++}$,
 there exist strictly monotone increasing, continuous bijections
 $\lambda_{n}:[0,T]\to[0,T]$, $n\in\NN$ (which may depend on $\vare$ as well)
 such that
\begin{align}\label{help_06}
\limsup_{n\to\infty}\sup_{t\in[0,T]}|\lambda_{n}(t)-t|<\vare
\qquad \text{and} \qquad 
\limsup_{n\to\infty}\sup_{t\in[0,T]}\|f_n(\lambda_{n}(t))-f(t)\|<\vare.
\end{align}
Indeed, by Billingsley \cite[Theorem 16.2]{Bil},
 it is enough to show that $\mathrm{d}^\circ_T(f_n|_{[0,T]},f|_{[0,T]})\to0$ as $n\to\infty$
  for all continuity  points $T\in\RR_{+}$ of $f$
 (for the definition of $\mathrm{d}^\circ_T$, see, e.g., Billingsley \cite[(12.16) and page 166]{Bil}).
Furthermore, by Billingsley \cite[Theorem 12.1]{Bil}, instead of the metric $\mathrm{d}^\circ_T$,
 we can work with the equivalent metric $\mathrm{d}_T$ (for its definition, see Billingsley \cite[(12.11), (12.12), and page 166]{Bil}).
In what follows, let $T\in\RR_{+}$ be a continuity point of $f$,  $\vare\in\RR_{++}$, and
 $\lambda_{n}:[0,T]\to[0,T]$, $n\in\NN$, be strictly monotone increasing,
  continuous bijections such that 
 \eqref{help_06} holds.
By Billingsley \cite[(12.11) and (12.12)]{Bil}, we have that
\[
  \mathrm{d}_T(f_n|_{[0,T]},f|_{[0,T]}) \leq \vare
    \qquad \text{for sufficiently large $n\in\NN$.}
\]
By \eqref{help_06} and taking $\limsup_{n\to\infty}$ of both sides of the previous inequalities,
 we get that $\limsup_{n\to\infty} \mathrm{d}_T(f_n|_{[0,T]},f|_{[0,T]}) \leq \vare$.
Since this holds for all $\vare\in\RR_{++}$, we have that
 $\mathrm{d}_T(f_n|_{[0,T]},f|_{[0,T]})\to 0$ as $n\to\infty$, as desired.

Let $\vare\in\RR_{++}$ and $T\in\RR_{+}$ be fixed
 such that $T$ is a continuity point of $f$.
We will prove that there exist strictly monotone increasing, continuous bijections
 $\lambda_{n}:[0,T]\to[0,T]$, $n\in\NN$ (which may depend on $\vare$ as well)
 such that \eqref{help_06} holds.
Let $\delta_n:=\sup_{t\in[0,T]}|h_n(t)-t|$, $n\in\NN$,
 where we have that $\delta_n\to0$ as $n\to\infty$ by \eqref{h_cond1}.

If $T=0$, then choose $\lambda_n$, $n\in\NN$, to be the identity function on $\{0\}$,
 i.e., $\lambda_n(0):=0$, $n\in\NN$.
Then $|\lambda_n(0)-0|=0$, $n\in\NN$, and
\[
\|f_n(0)-f(0)\|\leq\|f_n(0)-f(h_n(0))\|+\|f(h_n(0))-f(0)\|\to0 \qquad \text{as $n\to\infty$,}
\]
 where the first term on the right hand side of the above inequality tends to $0$ as $n\to\infty$
 by \eqref{h_cond2}, and the second term on the right hand side of the above inequality
 tends to $0$ as $n\to\infty$ by the facts that $f$ is right-continuous at $0$,
 $h_n(0)\in\RR_+$, $n\in\NN$, and $h_n(0)\to0$ as $n\to\infty$ by \eqref{h_cond1}.

Next, suppose that $T>0$.
Since $f$ is c\`adl\`ag,
 by Billingsley \cite[Chapter 3, Section 12, Lemma 1]{Bil},
 there exist $m>1$, $m\in\NN$ and real numbers $t_0,\dots,t_m\in\RR_+$
 such that $0=t_0<t_1<\cdots<t_{m}=T$ and
\begin{align}\label{help_08}
\sup_{s,t\in[t_{i-1},t_i)}\|f(s)-f(t)\|<\vare, \qquad i\in\{1,\dots,m\}.
\end{align}
For each $n\in\NN$, let $\tau_{n,0}:=0$, $\tau_{n,m}:=T$, and
\[
\tau_{n,i}:=\inf\{s\in[0,T]:h_n(s)\geq t_i\}, \qquad i\in\{1,\dots,m-1\},
\]
 where $\inf\emptyset:=\infty$.
Since $\delta_n\to0$ as $n\to\infty$,
 for sufficiently large $n\in\NN$, we have that $3\delta_n<T-t_{m-1}$,
 thus, by the definition of $\delta_n$,
 $h_n(T-\delta_n)\in[T-2\delta_n,T]\subset(t_{m-1},T]$,
 therefore the quantities $\tau_{n,i}$, $i\in\{1,\dots,m-1\}$,
 are finite (in fact, less than or equal to $T-\delta_n$) for sufficiently large $n\in\NN$.
Furthermore, using again the definition of $\delta_n$, we have $h_n(t_i+\delta_n)\geq (t_i + \delta_n) - \delta_n = t_i$, thus $\tau_{n,i}\leq t_i+\delta_n$, $i\in\{0,\dots,m-1\}$.
Similarly, for all sufficiently large $n\in\NN$ such that $t_1-\delta_n>0$,
 and for all $\eta\in(0,t_1-\delta_n)$, using again the definition of $\delta_n$, 
 we have $h_n(t_i-\delta_n-\eta)\leq (t_i-\delta_n-\eta) + \delta_n= t_i-\eta<t_i$, 
 and consequently, $t_i-\delta_n-\eta< \tau_{n,i}$, yielding that $t_i-\delta_n \leq \tau_{n,i}$,
 $i\in\{1,\dots,m\}$.
Therefore,
\[
0< t_i-\delta_n\leq\tau_{n,i}\leq t_i+\delta_n<T, \qquad i\in\{1,\dots,m-1\}, \quad \text{for sufficiently large $n\in\NN$.}
\]
Consequently, using also the fact that $0=t_0<t_1<\dots<t_{m-1}<T$, for
 sufficiently large $n\in\NN$, we have that $0=\tau_{n,0}<\tau_{n,1}<\dots<\tau_{n,m-1}<\tau_{n,m}=T$.
Since $\delta_n\to 0$ as $n\to\infty$, we have $\delta_n<\vare$ for sufficiently large $n\in\NN$.
Hence, in what follows, we can assume without loss of generality that $\delta_n< \vare$,
 $0< t_i-\delta_n\leq\tau_{n,i}\leq t_i+\delta_n<T$, $i\in\{1,\dots,m-1\}$ and
 $0=\tau_{n,0}<\tau_{n,1}<\dots<\tau_{n,m-1}<\tau_{n,m}=T$ hold for all $n\in\NN$.
For each $n\in\NN$, let $\lambda_n:[0,T]\to[0,T]$ be
 the piecewise linear interpolation connecting the points
 $(t_i,\tau_{n,i})$, $i\in\{0,\dots,m\}$.
Then $\lambda_n$ is strictly monotone increasing and continuous
 such that $\lambda_n(0)=0$ and $\lambda_n(T)=T$, $n\in\NN$.
By the definition of $\lambda_n$, $n\in\NN$, we have that
\begin{align}\label{help_07}
\sup_{t\in[0,T]}|\lambda_n(t)-t|
	\leq\max_{i\in\{0,\dots,m\}}|\lambda_n(t_i)-t_i|
	=\max_{i\in\{1,\dots,m-1\}}|\tau_{n,i}-t_i|
	\leq \delta_n < \vare \qquad n\in\NN,
\end{align}
 yielding that the first inequality in \eqref{help_06} holds.
Note that if $t\in[0,T]$ is such that $t\in[t_{i-1},t_i)$ for some $i\in\{1,\dots,m\}$,
 then we have that $\lambda_n(t)\in[\tau_{n,i-1},\tau_{n,i})$, $n\in\NN$.
By the definition of $\tau_{n,i}$, $n\in\NN$, we have that $h_n(s)<t_i$, $s\in[0,\tau_{n,i})$, $i\in\{1,\dots, m-1\}$, $n\in\NN$, and, using also that
 $h_n$ is monotone increasing and continuous from the right,
 we have that $h_n(s)\geq t_i$, $s\in[\tau_{n,i},T]$, $i\in\{1,\dots,m-1\}$, $n\in\NN$.
We also have that $h_n(t)\geq t_0=0$, $t\in[0,T]$, $n\in\NN$.
Thus we get that $h_n(\lambda_n(t))\in[t_{i-1},t_i)$ for $t\in[t_{i-1},t_i)$,
 $i\in\{1,\dots,m-1\}$, $n\in\NN$.
Consequently, by \eqref{help_08}, we get
\begin{equation}\label{h_help3}
\sup_{t\in[0,t_{m-1})}\|f(h_n(\lambda_n(t)))-f(t)\|
 \leq\max_{i\in\{1,\dots,m-1\}}\sup_{s,t\in[t_{i-1},t_i)}\|f(s)-f(t)\|
 <\vare, \qquad n\in\NN.
\end{equation}
If $t\in[t_{m-1},T]$, then $\lambda_n(t)\geq\tau_{n,m-1}$,
 and hence $h_n(\lambda_n(t))\geq t_{m-1}$, $n\in\NN$.
Moreover, by the definition of $\delta_n$, $n\in\NN$,
\[
h_n(\lambda_n(t))
\leq\lambda_n(t)+\delta_n
\leq T+\delta_n, \qquad t\in[0,T], \quad n\in\NN.
\]
Since $\delta_n\to0$ as $n\to\infty$, we can assume without loss of generality
 that $6\delta_n+n^{-1}<T-t_{m-1}$, $n\in\NN$.
Using that $\sup_{t\in[0,T]}|\lambda_n(t)-t|\leq \delta_n$, $n\in\NN$ (see \eqref{help_07}), the definition of $\delta_n$, and the monotonicity of $h_n$ and $\lambda_n$, $n\in\NN$,
 we have that
\[
h_n(\lambda_n(T-3\delta_n-n^{-1}))\leq h_n(T-2\delta_n-n^{-1})\leq T-\delta_n-n^{-1}<T,
 \qquad n\in\NN,
\]
and
\[
h_n(\lambda_n(T-3\delta_n-n^{-1}))\geq h_n(T-4\delta_n-n^{-1})\geq T-5\delta_n-n^{-1}> t_{m-1} + \delta_n\geq t_{m-1}, \qquad n\in\NN,
\]
and thus $h_n(\lambda_n(T-3\delta_n-n^{-1}))\in(t_{m-1},T)$, $n\in\NN$.
Consequently, using that $h_n\circ \lambda_n$ is monotone increasing and that $h_n(\lambda_n(t_{m-1}))\geq t_{m-1}$, $n\in\NN$,
 we have that
\begin{align*}
\sup_{t\in[t_{m-1},T]}\|f(h_n(\lambda_n(t)))-f(t)\|
	\leq&\sup_{t\in[t_{m-1},T-3\delta_n-n^{-1})}\|f(h_n(\lambda_n(t)))-f(t)\|\\
	&+\sup_{t\in[T-3\delta_n-n^{-1},T]}\|f(h_n(\lambda_n(t)))-f(t)\|\\
	\leq&\sup_{s,t\in[t_{m-1},T)}\|f(s)-f(t)\|
	+\sup_{s,t\in[T-5\delta_n-n^{-1},T+\delta_n]}\|f(s)-f(t)\|
\end{align*}
 for each $n\in\NN$.
Since $f$ is continuous at $T$ and $\delta_n\to0$ as $n\to\infty$,
 the second term on the right hand side of the above inequality tends to $0$ as $n\to\infty$.
Thus, by \eqref{help_08}, we have that
\begin{equation}\label{h_help4}
\limsup_{n\to\infty}\sup_{t\in[t_{m-1},T]}\|f(h_n(\lambda_n(t)))-f(t)\|<\vare.
\end{equation}
Hence \eqref{h_help3} and \eqref{h_help4} together yield that
\begin{equation}\label{h_help5}
\limsup_{n\to\infty}\sup_{t\in[0,T]}\|f(h_n(\lambda_n(t)))-f(t)\|<\vare.
\end{equation}
By the triangle inequality, we have that
\begin{align*}
\sup_{t\in[0,T]}\|f_n(\lambda_n(t))-f(t)\|
&\leq\sup_{t\in[0,T]}\|f_n(\lambda_n(t))
      -f(h_n(\lambda_n(t)))\| \\
&\quad+\sup_{t\in[0,T]}\|f(h_n(\lambda_n(t)))-f(t)\|, \qquad n\in\NN.
\end{align*}
For the first term on the right hand side of the above inequality,
 using that $\lambda_n([0,T])=[0,T]$, $n\in\NN$, by \eqref{h_cond2}, we have that
\begin{align*}
\sup_{t\in[0,T]}\|f_n(\lambda_n(t))-f(h_n(\lambda_n(t)))\|
\leq\sup_{t\in[0,T]}\|f_n(t)-f(h_n(t))\|\to0 \qquad \text{as $n\to\infty$.}
\end{align*}
Together with \eqref{h_help5}, we get that
\begin{align*}
\limsup_{n\to\infty}\sup_{t\in[0,T]}\|f_n(\lambda_n(t))-f(t)\|
	<\vare,
\end{align*}
 yielding that the second inequality in \eqref{help_06} holds, as desired.
\proofend

The following result concerns the relationship between the composition of functions and convergence
 in the Skorokhod $J_1$ topology.
Similar results but without conditions on the boundary point $0$, have already appeared in the literature, see, e.g., Whitt 
 \cite[Theorem 3.1]{Whitt} or Whitt \cite[Theorem 13.2.2]{Whitt2}. 
We think that the result as stated there do not hold in general, as the first
 counterexample in Remark \ref{Whitt_example} shows.
Our choice of the boundary condition $g(0) = 0$ in part (ii) of Lemma \ref{timechange_cont2} is sufficient for our case. 
Note also that Broutin et al.\ \cite[Theorem B.5]{BroDuqWan} cite the above mentioned result of Whitt \cite[Theorem 3.1]{Whitt} 
 with the additional conditions (in our notation) $g(0) = 0$ and $g_n(0) = 0$, $n \in\NN$,
 which suggests that the necessity of some boundary conditions may already be known, but we
 could not find any reference for it with a proof (in fact, Whitt \cite[page 75, remark after Theorem 3.1]{Whitt} 
 notes that in some cases, additional assumptions are needed).
Because of this, we decided to give a complete proof of Lemma \ref{timechange_cont2}, in which we do not refer to Whitt 
 \cite{Whitt, {Whitt2}}.

\begin{Lem}\label{timechange_cont2}
Let $d\in\NN$, and for all $g\in\DD(\RR_+,\RR)$
 such that $g$ is monotone increasing and $g(\RR_+)\subset\RR_+$,
 let $\psi_g(f):=f\circ g$, $f\in\DD(\RR_+,\RR^d)$.
\begin{enumerate}[label=(\roman*)]
\item
	Then $\psi_g(\DD(\RR_+,\RR^d))\subset\DD(\RR_+,\RR^d)$, and the mapping
	$\psi_g:\DD(\RR_+,\RR^d)\to\DD(\RR_+,\RR^d)$ is Borel measurable.
\item
	Suppose, in addition, that $g$ is continuous and strictly monotone increasing
	 with $g(0)=0$ and $\lim_{t\to\infty}g(t)=\infty$,
	 and $g_n\in\DD(\RR_+,\RR)$, $n\in\NN$,
	 are non-negative, monotone increasing functions
	 such that $g_n\lu g$ as $n\to\infty$.
	Then for all $f\in\DD(\RR_+,\RR^d)$ and $f_n\in\DD(\RR_+,\RR^d)$, $n\in\NN$,
	 such that $f_n\Jto f$ as $n\to\infty$,
	 we have that $\psi_{g_n}(f_n)\Jto\psi_g(f)$ as $n\to\infty$.
	In particular, if $f_n\Jto f$ as $n\to\infty$, then $\psi_g(f_n)\Jto\psi_g(f)$ as $n\to\infty$,
	 that is, in the considered special case, $\psi_g$ is continuous.
\end{enumerate}
\end{Lem}

\noindent
\textbf{Proof.}
\textup{(i).}\
First, we show that $\psi_g(\DD(\RR_+,\RR^d))\subset\DD(\RR_+,\RR^d)$.
For this part of the proof, let $f\in\DD(\RR_+,\RR^d)$ be fixed.
Let $t\in\RR_+$, and let $t_n\in\RR_+$, $n\in\NN$, be a monotone sequence with $\lim_{n\to\infty}t_n=t$.
Since $g$ is c\`adl\`ag, we have that the sequence $g(t_n)$, $n\in\NN$, is convergent.
Since $g$ is monotone increasing, we have that $g(t_n)$, $n\in\NN$, is monotone as well,
with the same monotonicity as $t_n$, $n\in\NN$.
Then, since $f$ is c\`adl\`ag, we get that $\lim_{n\to\infty}f(g(t_n))$ exists,
and if $t_n$, $n\in\NN$, is decreasing and $t_n\to t$ as $n\to\infty$, then $\lim_{n\to\infty}f(g(t_n))=f(g(t))$, yielding that $f\circ g\in\DD(\RR_+,\RR^d)$.

Next, we show the Borel measurability of $\psi_g$.
Using that the finite dimensional sets in $\DD(\RR_+,\RR^d)$ generate the Borel $\sigma$-algebra on $\DD(\RR_+,\RR^d)$
(see, e.g., Jacod and Shiryaev \cite[Chapter VI, Theorem 1.14, part c)]{JacShi}), to check the Borel measurability of $\psi_g$
 it is enough to verify that the mapping $\pi_t\circ \psi_g:\DD(\RR_+,\RR^d)\to\RR^d$
 is Borel measurable for  all $t\in\RR_+$.
For all $t\in\RR_+$, we have that $\pi_t\circ\psi_g=\pi_{g(t)}$, where the right hand side is Borel measurable,
 since $g(t)\in\RR_+$ and $\pi_s$ is Borel measurable for all $s\in\RR_+$
 (see, e.g., Billingsley \cite[part (ii) of Theorem 16.6]{Bil}).
Thus we get that $\psi_g$ is Borel measurable.

\textup{(ii).}\
Let $f\in\DD(\RR_+,\RR^d)$ and
 $f_n\in\DD(\RR_+,\RR^d)$, $n\in\NN$, be such that $f_n\Jto f$ as $n\to\infty$,
 and let $g\in\CC(\RR_+,\RR)$ and $g_n\in\DD(\RR_+,\RR)$, $n\in\NN$,
 be such that $g$ is strictly monotone increasing with $g(0)=0$ and $\lim_{t\to\infty}g(t)=\infty$,
 and $g_n$, $n\in\NN$, are non-negative and monotone increasing
 such that $g_n\lu g$ as $n\to\infty$.
Since $f_n\Jto f$ as $n\to\infty$,
 by Jacod and Shiryaev \cite[Chapter VI, Theorem 1.14]{JacShi},
 there exist strictly monotone increasing,
 continuous functions  $\lambda_n:\RR_+\to\RR_+$, $n\in\NN$,
 with $\lambda_n(0)=0$ and $\lim_{t\to\infty} \lambda_n(t)=\infty$, $n\in\NN$,
 such that $\sup_{t\in\RR_+}|\lambda_n(t)-t|\to0$ as $n\to\infty$ and
\begin{equation}\label{help_10}
\sup_{t\in[0,T]}\|f_n(t)-f(\lambda_n(t))\|\to0
	\qquad \text{as $n\to\infty$ for all $T\in\RR_{++}$.}
\end{equation}
For all $t\in\RR_+$, let $\delta_n(t):=\sup_{s\in[0,t]}|\lambda_n(s)-s|$
 and $\vare_n(t):=\sup_{s\in[0,t]}|g_n(s)-g(s)|$, $n\in\NN$,
 where we have that $\delta_n(t)\to0$ as $n\to\infty$
 and $\vare_n(t)\to0$ as $n\to\infty$ for all $t\in\RR_+$.
Using that 
 \[
  g_n([0,t])\subset [g_n(0),g_n(t)] \subset [0,g_n(t)],\qquad t\in\RR_+,\;\; n\in\NN,
 \]
 we obtain
\begin{align}\label{help_09}
 \begin{split}
 &\sup_{s\in[0,t]}|(\lambda_n\circ g_n)(s)-g(s)|
	\leq\sup_{s\in[0,t]}|\lambda_n(g_n(s))-g_n(s)|
		+\sup_{s\in[0,t]}|g_n(s)-g(s)|\\
&\qquad \leq\sup_{s\in[0,g_n(t)]}|\lambda_n(s)-s|
		+\vare_n(t)
	=\delta_n(g_n(t))+\vare_n(t), \qquad t\in\RR_+, \quad n\in\NN.
 \end{split}
\end{align}
Note that $g(\RR_+)=\RR_+$, and there exists an inverse function $g^{-1}:\RR_+\to\RR_+$ of $g$,
 which is strictly monotone increasing and continuous.
For each $n\in\NN$, let
 $\widetilde\lambda_n:=g^{-1}\circ\lambda_n\circ g_n$.
It is clear that for each $n\in\NN$, $\widetilde\lambda_n$ is non-negative,
 c\`adl\`ag, and monotone increasing.
Let $T\in\RR_+$ be fixed, and let
\begin{align*}
K_0:=&\ g(T)\vee\sup_{n\in\NN}g_n(T)
	\leq g(T)+\sup_{n\in\NN}\vare_n(T)<\infty,\\
K_1:=&\ K_0\vee\sup_{n\in\NN}\lambda_n(g_n(T))
      \leq K_0\vee \sup_{n\in\NN} (g_n(T) + \delta_n(g_n(T)))\\
	\leq&\, K_0\vee \left(\sup_{n\in\NN}g_n(T) +\sup_{n\in\NN}\delta_n(g_n(T))\right)
	\leq K_0\vee\left(K_0+\sup_{n\in\NN}\delta_n(g_n(T))\right)\\
	\leq&\, K_0+\sup_{n\in\NN}\delta_n(K_0)<\infty,
\end{align*}
 where the finiteness of $K_0$ and $K_1$ is a consequence of the facts that
 $\vare_n(T)\to 0$ as $n\to\infty$ and $\delta_n(K_0)\to 0$ as $n\to\infty$.
Note that, by \eqref{help_09}, we have
 $\sup_{t\in[0,T]}|(\lambda_n\circ g_n)(t)-g(t)|\leq\delta_n(K_0)+\vare_n(T):=e_n$, $n\in\NN$,
 yielding that $\sup_{t\in[0,T]}|(\lambda_n\circ g_n)(t)-g(t)|\to0$ as $n\to\infty$.
In addition, using that $g$ and $\lambda_n\circ g_n$, $n\in\NN$, are monotone increasing, we have that
\begin{align*}
\sup_{t\in[0,T]}|\widetilde \lambda_n(t)-t|
	=&\sup_{t\in[0,T]}|g^{-1}((\lambda_n\circ g_n)(t))-g^{-1}(g(t))|
	\leq \sup_{\substack{t,s\in[0,K_1]\\|t-s|\leq e_n}}|g^{-1}(t)-g^{-1}(s)|,
  \quad n\in\NN.
\end{align*}
Consequently, since $e_n\to0$ as $n\to\infty$ and $g^{-1}$ is uniformly continuous on compact intervals,
 \begin{equation}\label{help_11}
 \sup_{t\in[0,T]}|{\widetilde \lambda_n(t)}-t| \to 0 \qquad \text{as $n\to\infty$.}
 \end{equation}
Furthermore, we have that
\begin{align}
\begin{split}\label{help_12}
&\sup_{t\in[0,T]}\|(f_n\circ g_n)(t)-(f\circ g)(\widetilde\lambda_n(t))\|
	=\sup_{t\in[0,T]}\|f_n(g_n(t))-f((g\circ g^{-1}\circ\lambda_n\circ g_n)(t))\|\\
&\qquad\quad
		=\sup_{t\in[0,T]}\|f_n(g_n(t))-f(\lambda_n(g_n(t)))\|
		\leq\sup_{t\in[0,K_0]}\|f_n(t)-f(\lambda_n(t))\|\to0
\end{split}
\end{align}
 as $n\to\infty$, where the convergence follows from \eqref{help_10}.

Consequently, by \eqref{help_11}, \eqref{help_12}, and Lemma \ref{alt_J1}
 (with the choices $h_n:=\widetilde\lambda_n$, $n\in\NN$),
 we have that $f_n\circ g_n\Jto f\circ g$ as $n\to\infty$, as desired.

With the special choices $g_n:=g$, $n\in\NN$,
 we have that $\psi_g(f_n)\Jto\psi_g(f)$ as $n\to\infty$,
 that is, in the considered special case, $\psi_g$ is continuous.
\proofend

In the next remark, we give examples which show that, in general, the condition $g(0) = 0$ and
 that $g$ is strictly monotone increasing in part (ii) of Lemma \ref{timechange_cont2} cannot be removed.

\begin{Rem}\label{Whitt_example}
First, we give an example which shows that, in general, the condition $g(0) = 0$ in
 part (ii) of Lemma \ref{timechange_cont2} cannot be removed.
Let $f(t):=\bbone_{[1,\infty)}(t)$, $t\in\RR_+$,
 and $f_n(t):=\bbone_{[1+n^{-1},\infty)}(t)$, $t\in\RR_+$, $n\in\NN$,
 and let $g(t):=1+t$, $t\in\RR_+$, and $g_n:=g$, $n\in\NN$.
Then we have that $f\in\DD(\RR_+,\RR)$ and $f_n\in\DD(\RR_+,\RR)$, $n\in\NN$,
 and $g\in\CC(\RR_+,\RR)$ and $g_n\in\CC(\RR_+,\RR)$, $n\in\NN$,
 where $g$ and $g_n$, $n\in\NN$, are non-negative and strictly increasing.
Consequently, by part (i) of Lemma \ref{timechange_cont2}, we have $f\circ g\in\DD(\RR_+,\RR)$ and $f_n\circ g_n\in\DD(\RR_+,\RR)$, $n\in\NN$.
Furthermore, $g_n\lu g$ as $n\to\infty$, $g(0)=1$, $g(t)\to\infty$ as $t\to\infty$, and we check that $f_n\Jto f$ as $n\to\infty$.
By Jacod and Shiryaev \cite[Chapter VI, Theorem 1.14]{JacShi}, it is enough
 to show that there exists a sequence of continuous, strictly increasing functions 
 $\lambda_n:\RR_+\to\RR_+$, $n\in\NN$, with $\lambda_n(0)=0$ and $\lim_{t\to \infty}\lambda_n(t)=\infty$, $n\in\NN$,
 such that $\sup_{t\in\RR_+}\vert \lambda_n(t) - t\vert\to0$ as $n\to\infty$, and 
 \[
   \sup_{t\in[0,T]} \vert f_n(\lambda_n(t)) - f(t)\vert\to 0 \qquad \text{as $n\to\infty$ for all $T\in\RR_{++}$.}
 \]
For each $n\in\NN$, let $\lambda_n:\RR_+\to\RR_+$ be defined by
 \[
  \lambda_n(t):=\begin{cases}
                  \left(1+\frac{1}{n}\right)t & \text{if $t\in[0,1]$,}\\
                  t+\frac{1}{n}  & \text{if $t>1$.}
                \end{cases} 
 \]
Then $\lambda_n$ is a continuous, strictly increasing bijection of $\RR_+$ onto itself, and 
 \[
   \sup_{t\in\RR_+} \vert \lambda_n(t) - t\vert =\frac{1}{n}\to 0 \qquad \text{as $n\to\infty$.}
 \]
If $t\in[0,1)$, then $\lambda_n(t)<1+\frac{1}{n}$, and hence $f_n(\lambda_n(t))=0$ and $f(t)=0$.
If $t\geq 1$, then  $\lambda_n(t)\geq \lambda_n(1)=1+\frac{1}{n}$, and hence $f_n(\lambda_n(t))=1$ and $f(t)=1$.
Thus $f_n(\lambda_n(t)) = f(t)$, $t\in\RR_+$, yielding that 
 \[
   \sup_{t\in\RR_+}\vert f_n(\lambda_n(t)) - f(t)  \vert =0,\qquad n\in\NN.
 \]
However, since $(f\circ g)(t) = f(t+1)=1$, $t\in\RR_+$, and 
 \[
 (f_n\circ g_n)(t)=\bbone_{[1+n^{-1},\infty)}(1+t)= \bbone_{[n^{-1},\infty)}(t), \qquad t\in\RR_+, \quad n\in\NN,
 \] 
 we have that $(f_n\circ g_n)(0)=0$, $n\in\NN$, and $(f\circ g)(0)=1$.
Consequently, it does not hold that $f_n\circ g_n\Jto f\circ g$ as $n\to\infty$.
Indeed, on the contrary, supposing that $f_n\circ g_n\Jto f\circ g$ as $n\to\infty$,
 the continuity of the natural projection $\pi_0$ (see, e.g., Billingsley \cite[Chapter 3, part (i) of Theorem 16.6]{Bil}) 
 would yield that 
 \[
 0=(f_n\circ g_n)(0) = \pi_0(f_n\circ g_n)\to\pi_0(f\circ g)=(f\circ g)(0)=1 \qquad \text{as \ $n\to\infty$,}
 \]
  which is false.
This shows that part (ii) of Lemma \ref{timechange_cont2} without the condition $g(0)$ does not hold in general.

Next, we give an example which shows that, in general, the condition that $g$ is strictly monotone increasing 
 in part (ii) of Lemma \ref{timechange_cont2} cannot be removed.
Let $f(t):=\bbone_{[1,\infty)}(t)$, $t\in\RR_+$,
 and $f_n:=f$, $n\in\NN$,
 and let $g(t):=\min(t,1)$, $t\in[0,2]$,
 	$g(t):=t/2$, $t\in(2,\infty)$,
 	and $g_n(t):=\frac{n}{n+1}g(t)$, $t\in\RR_+$, $n\in\NN$.
Then we have that $f\in\DD(\RR_+,\RR)$ and $f_n\in\DD(\RR_+,\RR)$, $n\in\NN$,
 and $g\in\CC(\RR_+,\RR)$ and $g_n\in\CC(\RR_+,\RR)$, $n\in\NN$, where $g$ and $g_n$, $n\in\NN$, are  non-negative
 and increasing (but not strictly increasing).
Further, $f_n\Jto f$ as $n\to\infty$, $g(0)=0$, $g(t)\to\infty$ as $t\to\infty$, and $g_n\lu g$ as $n\to\infty$, since if $T\in\RR_+$, then
	\begin{align*}
	\sup_{t\in[0,T]}\vert g_n(t) - g(t)\vert
		&\leq\sup_{t\in[0,T+3]}\vert g_n(t) - g(t)\vert
		\leq\sup_{t\in[0,2]}\vert g_n(t) - g(t)\vert
		+\sup_{t\in[2,T+3]}\vert g_n(t) - g(t)\vert\\
		&=\sup_{t\in[0,2]}\left\vert \frac{n}{n+1}\min(t,1) - \min(t,1)\right\vert
			+\sup_{t\in[2,T+3]}\left\vert \frac{n}{(n+1)}\frac{t}{2} - \frac{t}{2}\right\vert\\
		&=\left(1-\frac{n}{n+1}\right)+\left(1-\frac{n}{n+1}\right)\frac{T+3}{2}
		\to0 \qquad \text{as $n\to\infty$.}
	\end{align*}
We have that $(f\circ g)(t) =\bbone_{[1,\infty)}(t) = f(t)$, $t\in\RR_+$.
Moreover, for each $n\in\NN$, if $t\in[0,2+2n^{-1})$, then
 $g_n(t)<g_n(2+2n^{-1})=\frac{n}{n+1}(1+n^{-1})=1$,
  and similarly, if $t\in[2+2n^{-1},\infty)$,
   then $g_n(t)\geq g_n(2+2n^{-1})=1$.
  Thus, since $f_n=f$, $n\in\NN$,
  we have that $(f_n\circ g_n)(t)=\bbone_{[2+2n^{-1},\infty)}(t)$, $t\in\RR_+$, $n\in\NN$.
Similarly to the previous example in this remark, one can show that $f_n\circ g_n\Jto (\bbone_{[2,\infty)}(t))_{t\in\RR_+}$ as $n\to\infty$,
 where the limit function does not equal $f\circ g$.
This shows that part (ii) of Lemma \ref{timechange_cont2} without the condition that $g$ is strictly increasing 
 does not hold in general.
\proofend
\end{Rem}

\begin{Lem}\label{int_cont}
For each $d\in\NN$, let $\cI^{(d)}:\DD(\RR_+,\RR^d)\to\CC(\RR_+,\RR^d)$,
\begin{equation*}
\left(\cI^{(d)}(g)\right)(t):=\int_0^t g(s)\,\dd s,\qquad t\in\RR_+,\qquad g\in\DD(\RR_+,\RR^d).
\end{equation*}
Then $\cI^{(d)}$ is continuous for each $d\in\NN$.
In particular, if $d\in\NN$ and $f\in\DD(\RR_+,\RR^d)$, $f_n\in\DD(\RR_+,\RR^d)$, $n\in\NN$,
are such that $f_n\Jto f$ as $n\to\infty$, then $\cI^{(d)}(f_n)\lu\cI^{(d)}(f)$ as $n\to\infty$.
\end{Lem}

\noindent\textbf{Proof.}
First, note that $\cI^{(d)}(g)$ is continuous for each $g\in \DD(\RR_+,\RR^d)$
 (since it is an integral with a variable upper limit, the integral function of $g$),
 so the range of $\cI^{(d)}$ is indeed a subset of $\CC(\RR_+,\RR^d)$.
The continuity of $\cI^{(1)}$ follows by Ethier and Kurtz \cite[Problem 3.11.26]{EthKur}.
For the rest of this proof, let $d\in\NN\setminus\{1\}$ be fixed.
For $g\in\DD(\RR_+,\RR^d)$, denote by $g^{(i)}$ the $i$-th coordinate function of $g$ for each $i\in\{1,\dots,d\}$,
that is, $g=[g^{(1)},\dots,g^{(d)}]^\top$.
Let $f\in\DD(\RR_+,\RR^d)$ and $f_n\in\DD(\RR_+,\RR^d)$, $n\in\NN$, be such that $f_n\Jtoo{d}f$ as $n\to\infty$.
Then it follows that $f_n^{(i)}\Jtoo{1} f^{(i)}$ as $n\to\infty$ for each $i\in\{1,\dots,d\}$,
 since $|h_1^{(i)}(t)-h_2^{(i)}(t)|\leq\|h_1(t)-h_2(t)\|$ for all $h_1,h_2\in\DD(\RR_+,\RR^d)$
 and $t\in\RR_+$, $i\in\{1,\dots,d\}$.
Since $\cI^{(1)}$ is continuous, we have that $\cI^{(1)}(f_n^{(i)})\Jtoo{1}\cI^{(1)}(f^{(i)})$
 as $n\to\infty$ for each $i\in\{1,\dots,d\}$.
Further, since for each $i\in\{1,\dots,d\}$, $\cI^{(1)}(f^{(i)})$ is continuous,
 by part \textup{(i)} of Lemma \ref{Jto_basic},
 we have that $\cI^{(1)}(f_n^{(i)})\lu\cI^{(1)}(f^{(i)})$ as $n\to\infty$.
Note that for all $T\in\RR_{++}$, we have
\begin{equation*}
\sup_{t\in[0,T]}\left\|\left(\cI^{(d)}(f_n)\right)(t)-\left(\cI^{(d)}(f)\right)(t)\right\|
	\leq\sum_{i=1}^d\sup_{t\in[0,T]}\left|\left(\cI^{(1)}(f_n^{(i)})\right)(t)-\left(\cI^{(1)}(f^{(i)})\right)(t)\right|, \qquad n\in\NN.
\end{equation*}
Since each term on the right hand side of the above inequality converges to $0$ as $n\to\infty$, 
 we have that $\cI^{(d)}(f_n)\lu\cI^{(d)}(f)$ as $n\to\infty$,
 that is, the second assertion of the lemma holds.
Thus, using again part (i) of Lemma \ref{Jto_basic} and the continuity of $\cI^{(d)}(f)$, we get that
  $\cI^{(d)}(f_n)\Jtoo{d}\cI^{(d)}(f)$ as $n\to\infty$, as desired.
\proofend

%\section*{Acknowledgements}

\section*{Statements and Declarations}

The authors declare that they have no known competing financial interests
or personal relationships that could have appeared to influence the work presented in this paper. 
Data sharing is not applicable to this article as no datasets were generated or analyzed during the current study.

\end{document}